\documentclass[11pt]{article}
\usepackage{amsmath,amssymb,amsthm,xcolor,graphicx,bbm}
\usepackage{accents}
\usepackage[unicode,breaklinks=true,colorlinks=true,citecolor = {magenta}]{hyperref}
\usepackage{esint}  %
\usepackage{listings}  %

\usepackage[top=1in, bottom=1in, left=1in, right=1in]{geometry}

\usepackage{thmtools} %
\usepackage[]{authblk}

\DeclareSymbolFont{yhlargesymbols}{OMX}{yhex}{m}{n}
\DeclareMathAccent{\wideparen}{\mathord}{yhlargesymbols}{"F3}

\numberwithin{equation}{section}

\newtheorem{theorem}{Theorem}[section]

\newtheorem{lemma}[theorem]{Lemma}
\newtheorem{proposition}[theorem]{Proposition}

\theoremstyle{remark}
\newtheorem{remark}[theorem]{Remark}
\newtheorem{example}[theorem]{Example}
 
\newcommand{\bke}[1]{\left ( #1 \right )}
\newcommand{\bkt}[1]{\left [ #1 \right ]}
\newcommand{\bket}[1]{\left \{ #1 \right \}}
\newcommand{\norm}[1]{\left  \| #1  \right \|}

\newcommand{\abs}[1]{\left | #1 \right |}

\newcommand\al{\alpha}
\newcommand\be{\beta}

\newcommand\de{\delta}

\newcommand\ve{\varepsilon}
\newcommand\e {\varepsilon}

\newcommand\ka{\kappa}
\newcommand\la{\lambda}
\newcommand\si{\sigma}

\newcommand\Ga{\Gamma}
\newcommand\De{\Delta}

\newcommand\Om{\Omega}
\newcommand{\R}{\mathbb{R}}

\newcommand{\ZZ}{\mathbb{Z}}
\newcommand{\NN}{\mathbb{N}}

\renewcommand{\div}{\mathop{\rm div}\nolimits}

\newcommand{\supp} {\mathop{\mathrm{supp}}\nolimits}

\newcommand{\esssup} {\mathop{\rm ess\,sup}}

\newcommand{\cM}{\mathcal{M}}

\newcommand{\pd}{\partial}
\newcommand{\nb}{\nabla}

\newcommand{\td}{\tilde}

\newcommand{\lec}{{\ \lesssim \ }}
\newcommand{\gec}{{\ \gtrsim \ }}

\newcommand{\I}{\infty}
\newcommand{\oo}{\infty}

\newcommand{\Eq}[1]{\begin{equation*}#1\end{equation*}}
\newcommand{\EQ}[1]{\begin{equation}#1\end{equation}}
\newcommand{\EQS}[1]{\begin{equation}\begin{split} #1 \end{split}\end{equation}}
\newcommand{\EQN}[1]{\begin{equation*}\begin{split} #1 \end{split}\end{equation*}}
\newcommand{\EN}[1]{\begin{enumerate} #1 \end{enumerate}}

\usepackage{mathtools} %

\newcommand{\loc}{\mathrm{loc}}
\newcommand{\uloc}{\mathrm{uloc}}
\newcommand{\one}{\mathbbm{1}}

\begin{document}
\title{\vspace{-1.2cm}Mild solutions and spacetime integral bounds for Stokes and Navier-Stokes flows in Wiener amalgam spaces}

\author[1]{\rm Zachary Bradshaw}
\author[2]{\rm Chen-Chih Lai \thanks{Corresponding author; Email address: cl4205@columbia.edu}} 
\author[3]{\rm Tai-Peng Tsai}
\affil[1]{
\footnotesize Department of Mathematics, University of Arkansas, Fayetteville, AR 72701, USA}
\affil[2]{
\footnotesize Department of Mathematics, Columbia University, New York, NY 10027, USA}
\affil[3]{
\footnotesize Department of Mathematics, University of British Columbia, Vancouver, BC V6T 1Z2, Canada}

%\author{Zachary Bradshaw \and 
% Chen-Chih Lai \and
% Tai-Peng Tsai}
\date{}

\maketitle

\vspace{-1.0cm}

\begin{abstract}
We first prove decay estimates and spacetime integral bounds for Stokes flows in amalgam spaces $E^r_q$ which connect the classical Lebesgue spaces to the spaces of uniformly locally $r$-integrable functions. Using these estimates, we construct mild solutions of the Navier-Stokes equations in the amalgam spaces satisfying the corresponding spacetime integral bounds. Time-global solutions are constructed for small data in $E^3_q$, $1\le q \le 3$. Our results provide new bounds for the  strong solutions classically constructed by Kato and the more recent solutions in uniformly local spaces constructed by Maekawa and Terasawa.  As an application we obtain a result on the stability of suitability for weak solutions to the perturbed Navier-Stokes equation where the drift velocity solves the Navier-Stokes equations and has small data in a local $L^3$ class. Extending an earlier result, we also construct {global-in-time} local energy weak solutions in {$E^2_q$}, $1\le q <2$. 
\end{abstract}

\renewcommand{\baselinestretch}{0.8}\normalsize
\tableofcontents
\renewcommand{\baselinestretch}{1.0}\normalsize

\section{Introduction}

This paper considers mild solutions with decay estimates and spacetime integral bounds of the nonstationary Stokes system  and the Navier-Stokes equations in $\R^d$, $d=3$, for initial data in uniformly local $L^q$ spaces and Wiener amalgam spaces. The nonstationary Stokes system in $\R^d$ reads
\begin{equation}\label{Stokes}
\left.
\begin{aligned}
u_{t}-\Delta u+\nabla \pi=f \\
 \div u=0
\end{aligned}\ \right\}\ \ \mbox{in}\ \ \R^{d}\times (0,\infty),
\end{equation}
with initial condition
\begin{equation}\label{E1.2}
u(\cdot,0)=u_0, \quad \div u_0=0.
\end{equation}
Here  $u=(u_{1},\ldots,u_{d})$ is the velocity, $\pi$ is the pressure, and $f=(f_{1},\ldots,f_{d})$ is the external force. 
The Navier-Stokes equations reads
\begin{equation}\label{NS}
\left.
\begin{aligned}
u_{t}-\Delta u+\nabla \pi= - u\cdot \nb u \\
 \div u=0
\end{aligned}\ \right\}\ \ \mbox{in}\ \ \R^{d}\times (0,\infty),
\end{equation}
with initial condition \eqref{E1.2}.

A classical approach to studying \eqref{NS} is to first establish estimates for the heat equation and \eqref{Stokes} and then extend these to the nonlinear system \eqref{NS} using a Picard iteration. This was done in $L^p$ spaces for ${d}<p<\I$ by Fabes, Jones and Riviere \cite{FJR}  and {$L^d$} by Kato \cite{Kato}. The endpoint case $p=\I$ has been examined by Giga, Inui and Matsui \cite{GIM} as well as Kukavica \cite{Kukavica}.  The solutions generated by these arguments satisfy an integral formula and are referred to as \emph{mild solutions}. In the present context, they are smooth and unique. The cases {$L^d$} and $L^p$ for {$p>d$} are qualitatively distinct in that, for {$L^d$}, global existence is known but only  for small data and local existence is known for large data but the time-scale is not related to the size of the data. On the other hand,  for $L^p$ where {$p>d$}, local existence is known for arbitrarily large data and the time-scale is dependent on the size of the data, but global existence is unknown.

To develop a better understanding of this subject, note that large and small scales play different roles in the regularity of solutions to the Navier-Stokes equations, which represents the fact that, for $p_1 > p_2$, $L^{p_1}$ has better decay at small scales and worse decay at large scales than $L^{p_2}$. Additionally, for parabolic equations, small physical scales in the initial data act primarily on short time scales in the solution while the same is true for large physical scales and large time scales. Hence, if the data is sufficiently regular at small scales, then it should be smooth for small enough times. {Decay} of the data at large scales should lead to regularity at large times. For $d=3$, this theme holds true for the Navier-Stokes equations in $L^p$ spaces where the endpoint space for both scenarios is $L^3$. Indeed, for data in $L^2$, the time-global weak solutions of Leray  exhibit eventual regularity. For data in $L^p$ with $p>3$, the time-local mild solutions of Fabes, Jones and Riviere \cite{FJR}  as well as those of Giga,  Inui and Matsui  \cite{GIM} are smooth.

The small and large scale decay assumptions on the data can be refined. For example, 
Maekawa and Terasawa \cite{MaTe} extended the $L^\I$ theory to a scale of spaces $L^r_\uloc$ defined by the   norm 
\[
\| a\|_{L^r_\uloc}:= \sup_{x_0\in \R^d} \| a\|_{L^r(B_1(x_0))}.
\]

In the case of $L^d_\uloc$, we are considering a space with local $L^d$ integrability properties and globally worse than $L^d$ decay properties. Hence we only expect short time existence of a smooth solution for small data and this is what is proven in \cite{MaTe}---the long time result of \cite{Kato} is out of reach due to a lack of decay. When $d<r<\I$, Maekawa and Terasawa \cite{MaTe} prove local existence of strong mild solutions in analogy with \cite{FJR}. Note that space-time integral estimates of Kato \cite{Kato} (due to Giga) aren't included in \cite{MaTe}. 

The classical Lebesgue spaces and the uniformly local spaces can be connected in a sense by the Wiener amalgam spaces.   These spaces have a rich literature \cite{BL,BS,CKS,FoSt,Holland,KNTYY,Lakey} and have been applied to the analysis of fluids in, e.g., \cite{KNTYY,AB,BT4}. The Wiener amalgam spaces are denoted $E^p_q$ and defined by the norm  
\[
\|  a \|_{E^p_q} :=\bigg\|   \bigg( \int_{B_1(k)}|a(x)|^p\,dx  \bigg)^{1/p} \bigg\|_{\ell^q(k\in {\ZZ^d})}<\I.
\]
We identify $E^p_\I$ with $L^p_\uloc$.\footnote{Let $E^p$ be the closure of the test functions under the $L^p_\uloc$ norm, a notation used in \cite{LR2}. Our convention in this paper is $E^p \subsetneqq E^p_\infty = L^p_\uloc$. We used $E^p_\infty=E^p$ in 
\cite{BT4} since $E^p$ preserves a decay property at spatial infinity  which is also enjoyed by $E^p_q$ where $q<\I$.} So, $E^p_q$ encodes local $L^p$ integrability and global $L^q$ decay. We have the embeddings
\EQ{\label{eq1.4}
E^p_q \subset E^p_m \text{ for }m>q\text{ and }E^p_q \subset E^r_q \text{ for }r<p.
}
For the classical Lebesgue spaces, $L^p=E^p_p$, we therefore have the embeddings
\[
L^p \subset E^p_q  \text{ and }L^q\subset E^p_q \text{  if }q>p,
\]
as well as
\[
L^p \supset E^p_q  \text{ and }L^q\supset E^p_q \text{  if }p>q.
\]
We will use the H\"older inequality for $p,p_1,p_2,q,q_1,q_2 \in [1,\infty]$
\EQ{\label{Holder}
\norm{fg}_{E^p_q} \le  \norm{f}_{E^{p_1}_{q_1}}  \norm{g}_{E^{p_2}_{q_2}}, \quad \frac 1p= \frac 1{p_1}+\frac1{p_2},\quad
\frac 1q\le \frac 1{q_1}+\frac1{q_2}.
}It implies \eqref{eq1.4} by taking $g=1$.%

The main goal of this paper is to develop a theory of strong solutions in the Wiener amalgam spaces $E^p_q$ where $d=3\le p\leq \infty$ and $1\le q\le\infty$. 
Our theory encompasses the following classical results on strong solutions in $L^p$-type spaces in $\R^d$ when $d=3$, in chronological order (this is an incomplete list, of course; it does not include solutions in Sobelev or Besov spaces, nor weak solutions):

\EN{

\item Fabes, Jones and Riviere \cite{FJR} show local existence for any $u_0 \in L^r(\R^d)$, $r>d\ge3$. Their
 Theorem (4.3) gives global existence for small $u_0\in   L^{r_1} \cap L^{r_2}(\R^d)$,    $ r_1 < d < r_2$.
Their solutions belong to $L^s (0,T; L^p)$, $\frac dp + \frac 2s \le 1$, for each finite $T$. %

\item Giga-Miyakawa \cite{GiMi} (submitted before \cite{Kato} although published later) includes global solutions for small $L^d$ initial data in bounded domains in $\R^d$---see Theorem 2.6. It estimates the nonlinear term by means of suitable fractional powers of the Stokes operator. This approach is extended to exterior domains (see, e.g.,~\cite{GS1, GS2}) and, although it has not been explored in the literature, it seems likely that it also extends to $\R^d$ ($d\ge3$) where the Stokes operator and Helmholtz projector have explicit forms.
%which does not seem to extend directly to entire $\R^d$.

\item Kato \cite{Kato} constructs global strong solutions for small initial data in $L^d(\R^d)$ and establishes
the spacetime integral bound $u \in L^s(0,\oo;L^p(\R^d))$ using Giga's estimate \eqref{Giga-est}; see \cite{Giga}. This  spacetime integral bound is not necessary for existence, but is used in the proof of Theorem 2$'$ which asserts  that $\|u(t)\|_{L^d} \to 0$ in time average as $t \to \infty$; see \eqref{timeaveragezerolimit} and \cite[(2.15)]{Kato}.

\item Giga \cite{Giga} shows \eqref{Giga-est} and develops the $L^r$-theory for general domains and equations. 

\item  Giga, Inui and Matsui \cite{GIM} and Kukavica \cite{Kukavica} prove local existence  of strong solutions in $L^\infty({\R^d})$.
The analyticity of the Stokes semigroup is established by Abe and Giga \cite{AG} for bounded domains and by Hieber and Maremonti \cite{HM} for an exterior domains in $\R^d$ for $d\ge3$. 
More results on the Navier-Stokes equations are given by Abe \cite{Abe2015} in the $L^\infty$ framework for a class of domains including bounded and exterior domains.

\item Maekawa and Terasawa \cite{MaTe} establish local existence of strong solutions  in   $L^r_\uloc(\R^d)$, $r \ge d$.
}

Global spacetime integral bounds are shown only in \cite{Kato,Giga}.

Note that because $E^r_q \subset L^{{r}}_\uloc$ where a strong solution theory exists, the novel contribution of our work concerns  the persistence of norms in the Wiener amalgam spaces, in addition to new spacetime estimates.  Moreover, our second exponent $q$ in $E^r_q$ (measuring spatial decay), is allowed to go down to $q=1$, which will lead to solutions which are more localized in space. The lower bound $r \ge 3=d$   is expected since it is the borderline for local regularity theory.

Some classical estimates in Wiener amalgam spaces can be found in \cite{BS}. The theme of ``filling in'' spaces between a classical Lebesgue spaces and a uniformly locally Lebesgue space using amalgam spaces is explored for \emph{weak} solutions in spaces $E^2_q$ between $L^2$ and $L^2_\uloc$ by the first and last author in \cite{BT4}.  Hence this paper can be considered a companion of \cite{BT4}.  The exponent $q$ in \cite{BT4} for weak solutions is limited to $2\le q \le \infty$. We will sketch necessary changes to extend the weak solution theory to $E^q$, $1\le q <2$, in Section \ref{sec:global-weak}.

\subsection{Spacetime integrals}

Spacetime integral bounds are useful in analyzing solutions of the Navier-Stokes equations because they provide more information about spatial-temporal decay.
In particular, the borderline integrability inferred from spacetime integral bounds at $t=0+$ and $t=T_{\max}-$ can be used to study the regularity of the solutions.
Moreover, spacetime integral bounds are used by Kato \cite[(2.15)]{Kato} to show the convergence of time average of the $L^d$ norm
\EQ{\label{timeaveragezerolimit}
\frac1T \int_0^T \norm{u(t)}_{L^d(\R^d)}\to0\ \text{ as }\ T\to\infty,
}
for all $d\ge2$, and $\norm{u(t)}_{L^2}\to0$ for $d=2$.

We will consider two kinds of spacetime integrals: For $0<T\le \I$, $x\in \R^d$, and $1\le s,p,q\le\infty$, define the norms $L^s_T E^p_q$ and $E^{s,p}_{T,q}$ as follows:
\EQ{\label{LsEpq-def}
\| u\|_{L^s_T E^p_q} := \|  u \|_{L^s(0,T; E^p_q(\R^d))}, 
}
and  
\EQ{\label{Espq-def}
\norm{u}_{E^{s,p}_{T,q}}:=
\norm{\norm{ u}_{L^s_TL^p(B_1(k))}}_{\ell^q(k\in \ZZ^d)}.
}
These norms are different from each other when $s\neq q$.
By Minkowski's integral inequality,
\EQ{\label{ineq:embedding.parabolic.space}
\norm{u}_{L^s _T E^p_q} \le \norm{u}_{E^{s,p}_{T,q}},
\qquad \text{ if }q\le s,
}
and
\EQ{\label{ineq:embedding.parabolic.space2}
\norm{u}_{E^{s,p}_{T,q}} \le \norm{u}_{L^s _T E^p_q},
\qquad \text{ if }q\ge s.
}

The reversed inequalities are typically wrong. See Example \ref{example52}.

The class $E^{s,p}_{T,q}$ seems more appropriate to parabolic equations than the class $L^s_T E^p_q$.
Indeed, for the heat semigroup in $\R^d$, when $s<\I$ it follows from 
Giga \cite{Giga}  that
\EQ{\label{Giga-est}
\norm{e^{t\De} a}_{L^s(0,\infty; L^p(\R^d))} \lec \norm{a}_{L^r(\R^d)},
}
if
\EQ{\label{0812b}
1<r\le s\le\infty,\quad r\le p<\infty, \quad
\frac 2s + \frac dp = \frac dr.
}
The case $s=\I$ is classical.
A special case of our Lemma \ref{lemma:linear.heat.amalgam.spacetime-new} extends the above to
\EQ{\label{0201a}
\mathbbm{1}_{m\le s}  \norm{e^{t\De} a}_{L^s_{T=\infty} E^p_{m} }  \lec \|e^{t\Delta}a\|_{E^{s,p}_{T=\infty,m}} \lesssim  \|a\|_{E^r_q},	} 
if
\[
1<r\le s\le\infty,\quad r\le p<\infty, \quad
\frac 2s + \frac dp = \frac dr, \quad 1<q<m<\infty, \quad \frac 2s + \frac dm \le \frac dq.
\]
Lemma \ref{lemma:linear.heat.amalgam.spacetime-new} also has  finite time estimates which allow $1\le q\le m\le\infty$ without the restriction $\frac 2s + \frac dm \le \frac dq$.

In \eqref{Giga-est}, the case $r<s$ is proved in \cite[Lemma, p.~196]{Giga} using the Marcinkiewicz interpolation theorem. 
It is mentioned in \cite[Acknowledgments]{Giga} that the case $r=s$ in \eqref{Giga-est} is also valid if one appeals to the generalized Marcinkiewicz theorem. We will give details for $r=s$ in the Appendix (\S\ref{sec5.1}). The case $r>s$ was unclear. We will show in Example \ref{example2} that \eqref{Giga-est} is actually false when $r>s$ even for $T<\infty$.
Moreover, Example \ref{example2} shows that if $s< q \le \infty$, then there is $a \in E^r_q$ with $e^{t\De}a \not \in L^s_{T=1} E^p_m$. Thus the factor $\mathbbm{1}_{m\le s} $ in \eqref{0201a} is necessary. In particular, when $a\in L^r_\uloc$,  $e^{t\De}a $ is bounded in $E^{s,p}_{T=1,\infty}$, but not in $L^s_{T=1} E^p_\infty$.

\subsection{Mild solutions in amalgam spaces}\label{sec1.2}
 
We now present three theorems on mild solutions  to \eqref{NS} in $\R^3$. 
Recall that a mild solution is a solution $u$ to \eqref{NS} satisfying
\[
u(x,t) =e^{t\Delta}u_0 -B(u,u)(t); \qquad B(f,g)(t)= \int_0^t e^{(t-s)\Delta} \mathbb P\nb\cdot (f\otimes g)\,ds,
\]
where $\mathbb P$ is the Helmholtz projection operator.  When the spatial domain is $\R^3$, we can express $B$ by the Oseen tensor, see \eqref{Bfg.def}.

The following theorem concerns data which is locally \emph{subcritical}, i.e., $u_0 \in E^r_q$ with $r>3$. When $q=\I$ it includes  \cite[Theorem 1.1 (i)]{MaTe} and when $r=q<\I$ it includes the results from \cite{FJR}.

\begin{theorem}[Subcritical data]\label{thrm.subcritical}
Let $r\in (3,\I]$  and $q\in[1,\infty]$. If $u_0\in E^r_q$ is divergence free, then, for any  positive time $T=T(\|u_0\|_{E^r_q})$ chosen so that
\EQ{ \label{def:T.subcrit}    T^{1/2 - 3/(2r)} +  T^{1/2 }	  
\lec \|u_0\|_{E^r_q}^{-1},
} there exists a unique mild solution $u\in L^\infty(0,T;E^r_q) \cap C((0,T);E^r_q)$ to \eqref{NS}. Moreover, $u$ satisfies
\EQ{ \label{th1.1eq2}
\sup_{0\leq t\leq T} \| u(t)\|_{E^r_q} \leq C \|u_0\|_{E^r_q}. 
}

If $q,r<\I$, then $u\in C([0,T];E^r_q)$. If $q=\I$ or $r=\I$, then we still have $\|e^{t\Delta}u_0- {u(t)}\|_{E^r_q}\to 0$ as $t\to 0^+$. 

Furthermore, if $r<\infty$, then for any $s \in [r,\infty]$ and $p\in[r,3r]$ 
with $\frac2s + \frac3p = \frac3r$,
\[
\norm{u}_{ E^{s,p}_{T,m\ge q}} \le C \norm{u_0}_{E^r_q}
\]
provided $(1+T^{\frac1s+\epsilon})(T^{\frac12-\frac3{2r}} + T^{1-\frac1s}) \lec \norm{u_0}_{E^r_q}^{-1}$ for all $\epsilon>0$. 
\end{theorem}

\emph{Comments on Theorem \ref{thrm.subcritical}:}
\begin{enumerate}

\item The apparent dimensional mismatch between terms in \eqref{def:T.subcrit} is corrected by suppressed dimensional constants of magnitude $\sim 1$ and this comment holds throughout the paper.  

\item Our uniqueness class does \emph{not} require the bound \eqref{th1.1eq2}.
 The uniqueness class can be weakened to the uniqueness class of  Maekawa and Terasawa \cite{MaTe} by the embedding $E^r_q \subset L^r_\uloc$. 
 
\item Also note that, when  $q=\I$ or $r=\I$, we do not have  $\|e^{t\Delta}u_0- u_0\|_{E^r_q}\to 0$ as $t\to 0^+$. See 
Lemma \ref{lemma:heat.eq.convergence.data}. For $q=\I$, we   have by \cite[Theorem 1.2]{MaTe} that,  for any ball $B$,
\[
\lim_{t\to 0^+}\| u-u_0\|_{L^r(B)} =0.
\]

\item {When $q \le 3 < r$, we expect global solutions for small data in $E^r_q$ in
\EQ{\label{1.15}
u \in C([0,\infty); E^r_q)  .
}
It should be the same solution considered in
$
u \in C([0,\infty); E^3_q)$ in Theorem \ref{thrm:critical2}. The estimate \eqref{1.15} does not give better spatial-temporal decay than Theorem \ref{thrm:critical2}. %
Since we expect regularity for $t>1$, the main gain of \eqref{1.15} is extra integrability for $t\sim 0_+$, which
may alternatively be obtained from local smoothing estimates in \cite{JS,BP2020,KMT}. Hence we do not pursue it.}
\end{enumerate}

We now turn to the case of data $u_0\in E^3_q$, which we refer to as \emph{critical}.

\begin{theorem}[Critical data I]\label{thrm:critical} Let $q\in[1,\I]$. Fix $T>0$.
There exists $\ve=\ve(T)>0$ such that for all {divergence-free} $u_0\in E^3_q$ with $\norm{u_0}_{E^3_q}\le\ve$, there exists a mild solution $u$  to \eqref{NS} with 
\[
u \in L^\infty(0,T; E^3_q)\quad \text{and}\quad t^{\frac12} u \in L^\infty (0,T; E^\oo_q).
\]
The solution is unique in the class 
\EQ{\label{eq-uniquesmall}
\sup_{0<t<T} t^{\frac14} \norm{u}_{E^6_q} \le 2 \sup_{0<t<T} t^{\frac14} \norm{e^{t\De} u_0}_{E^6_q}.
}
Furthermore, $\norm{u}_{L^\infty_T E^3_q} + \norm{t^{1/2}u}_{L^\infty_T E^\oo_q} \lec \norm{u_0}_{E^3_q}$. We have $u\in C((0,T); E^3_q) $ for $q=\I$ and $u\in C([0,T); E^3_q)$ for $q<\I$. If $q=\I$, then we have for any ball $B$ and $\delta\in (0,2]$ that 
\EQ{\label{convergencetodata}
\lim_{t\to 0^+}\|  u (t) - u_0\|_{L^{3-\delta}(B)}=0.
}

For any $s\in [3,\infty]$ and $p\in [3,9]$ given by $\frac2s + \frac3p = 1$, by taking $\ve \le \ve_0(T,s)$ sufficiently small, this solution further satisfies%
\EQ{\label{eq1.7}
\norm{u}_{E^{s,p}_{T,m}} + \one_{q \le s} \norm{u}_{L^s_T E^p_m} \le C\norm{u_0}_{E^3_q},\quad \forall m \in [q,\infty].
}
\end{theorem} 

\emph{Comments on Theorem \ref{thrm:critical}:}
\begin{enumerate}

\item When ${1\le q}\le 3$, this theorem and Theorem \ref{thrm:critical2} extend Kato \cite{Kato}, see also Giga \cite{Giga}.
Note that, if ${1\le q} < 3$, then $E^3_q  \subset (L^3 \cap L^q)$, considered in \cite[Theorems 3\&4]{Kato}. The only spacetime integral estimate in \cite{Kato,Giga} is $L^s L^p$.  

When $q=\infty$, this theorem extends Maekawa-Terasawa \cite[Theorem 1.1 (iii)]{MaTe}. The spacetime integral bound in $E^{s,p}_{T,m=\infty}$ for $s<\infty$ is not considered in \cite{MaTe}.

\item The uniqueness part of Theorem \ref{thrm:critical}, unlike Theorem \ref{thrm.subcritical}, assumes smallness given by the condition \eqref{eq-uniquesmall}. Such a smallness condition for uniqueness is also implicitly assumed in \cite[Theorem 1.1 (iii)]{MaTe} when $q=\infty$.

\item \label{th2.remark.convergence1} We refer to the convergence in \eqref{convergencetodata} as convergence in $L^{3-\delta}_\loc$. For sub-critical data $u_0$ in $L^p_\uloc$, $p>3$, it is shown in \cite{MaTe} that $u\to u_0$ in $L^p_\loc$. When $p=3$, it is shown in \cite{MaTe} that $u\to u_0$ in $L^3_\loc$ provided $u_0$ is in the $L^3_\uloc$ closure of bounded uniformly continuous functions. This complements our result \eqref{convergencetodata} which does not require $u_0$ to be in this closure. 
\end{enumerate}

\begin{theorem}[Critical data II]\label{thrm:critical2}
Let $1\le q\leq 3$. 
For all divergence-free $u_0\in E^3_q$, there exist $T=T(u_0)>0$ and 
a unique mild solution $u$  to \eqref{NS} satisfying
\[
u \in BC([0,T); E^3_q)\quad \text{and}\quad t^{\frac12} u \in L^\I (0,T; E^\oo_{q_2}),
\]
with $1/q_2 = 1/q-1/3$, $q_2 \in [\frac 32,\infty]$. For any 
$s\in [3,\infty)$, $\frac2s + \frac3p = 1$, and $m \in [q,\infty]$,
there is $T_1 \in (0,T]$ such that 
\EQ{\label{eq-thmIII-Ebound}
u \in E^{s,p}_{T_1,m}.
}
Furthermore, there is $\e(q)>0$ such that $T=\oo$ if $\norm{u_0}_{E^3_q} \le \e(q)$.
If we assume further%
\EQ{\label{eq-thmIII-mbound}
m > p' = \frac p{p-1},\quad \text{and}\quad
m \ge m_1,\quad
\frac 2 s +\frac 3 {m_1} = \frac 3 q, 
}
with $m>m_1(s,q)$ when $q=1$,
then there exists $\e_1(s,q,m)>0$ such that $T_1=\oo$ if $\norm{u_0}_{E^3_q} \le \e_1(s,q,m)$.
Instead of \eqref{eq-thmIII-mbound}, %
if we assume
\EQ{
m \ge \max(p',m_1), \quad \text{and}\quad \left\{
\begin{aligned}
m>m_1&\quad \text{if }\quad q=1,\\
m \ge p &\quad \text{if }\quad 3s<5q,
\end{aligned}\right.
}
then there exists $\e_2(s,q,m)>0$ such that $u \in L^s_{T=\infty} E^p_m$ if $\norm{u_0}_{E^3_q} \le \e_2(s,q,m)$.

\end{theorem}

\emph{Comments on Theorem \ref{thrm:critical2}:}
\begin{enumerate}

\item
Theorem \ref{thrm:critical2} is limited to $q\in [1,3]$ and improves Theorem \ref{thrm:critical} in two aspects: First, it does not require smallness of the norm $\norm{u_0}_{E^3_q}$ for local existence. Second, it gives  
global existence for small data.
For global existence and for either $\frac 2 s +\frac 3 m < \frac 3 q$ or $\frac 2 s +\frac 3 m = \frac 3 q$, we need $q<m$. Unlike Theorem \ref{thrm:critical}, we require $s<\infty$ in \eqref{eq-thmIII-Ebound}.

\item The inclusion $t^{\frac12} u \in L^\I (0,T; E^\oo_{q_2})$ is worse than  $t^{\frac12} u \in L^\I (0,T; E^\oo_{q})$ in Theorem \ref{thrm:critical} since $q<q_2$.
This relaxed choice of space, possible only when $q\le 3$, allows us to remove the smallness assumption on the initial data for local existence.

\item For global space-time integral estimates, 
the exponent $m$ is always no less than $\max(p',m_1) \in (q,p]$. In particular, $q<m$.
The rectangle $[\frac 1s,\frac 1q] \in (0,\frac 13]\times [\frac 13,1]$ is divided into 3 regions as shown in Figure \ref{fig:m*}:

\begin{figure}[h!] %
\setlength{\unitlength}{1.15mm} %
\noindent
\begin{center}
\begin{picture}(70,50)
\put(1,5){\vector(1,0){66}}  
\put (1,45){\vector(0,1){6}}
\multiput(1,5)(0,4){10}{\line(0,1){2.5}}
\put(1,45){\line(1,0){60}}   
\put (61,5){\line(0,1){40}}
\put(1,25){\line(2,1){40}}
\put(37,5){\line(3,2){24}}
\put(69,5){\makebox(0,0)[c]{\scriptsize $\frac 1s$}}
\put(3,49){\makebox(0,0)[c]{\scriptsize $\frac1q$}}
\put(10,38){\makebox(0,0)[c]{\scriptsize I}}
\put(9,27){\makebox(0,0)[c]{\scriptsize $L_1$}}
\put(32,24){\makebox(0,0)[c]{\scriptsize II}}
\put(40,9){\makebox(0,0)[c]{\scriptsize $L_2$}}
\put(55,11){\makebox(0,0)[c]{\scriptsize III}}
\put(51,46.5){\makebox(0,0)[c]{\scriptsize $L_3$}}

\put(-1,45){\makebox(0,0)[c]{\scriptsize $1$}}
\put(-1,25){\makebox(0,0)[c]{\scriptsize $\frac 23$}}
\put(-1,5){\makebox(0,0)[c]{\scriptsize $\frac 13$}}
\put(1,3){\makebox(0,0)[c]{\scriptsize $0$}}
\put(37,3){\makebox(0,0)[c]{\scriptsize $(\frac 15, \frac 13)$}}
\put(37,5){\makebox(0,0)[c]{\scriptsize $\bullet$}}
\put(41,48){\makebox(0,0)[c]{\scriptsize $(\frac 14, 1)$}}
\put(41,45){\makebox(0,0)[c]{\scriptsize $\bullet$}}
\put(61,3){\makebox(0,0)[c]{\scriptsize $(\frac 13, \frac 13)$}}
\put(61,5){\makebox(0,0)[c]{\scriptsize $\bullet$}}
\put(66,20){\makebox(0,0)[c]{\scriptsize $(\frac 13, \frac 59)$}}
\put(61,21){\makebox(0,0)[c]{\scriptsize $\bullet$}}
\put(66,45){\makebox(0,0)[c]{\scriptsize $(\frac 13, 1)$}}
\put(61,45){\makebox(0,0)[c]{\scriptsize $\bullet$}}
\end{picture}

\end{center}
\caption{Regions of exponents for the lower bound of $m$}
\label{fig:m*}
\end{figure}

They are bordered by two line segments:
\[
L_1:\ \frac 1q= \frac23 + \frac 4{3s}, \qquad
L_2:\ \frac 1q=  \frac 5{3s}.
\]

For $ E^{s,p}_{T=\infty,m}$ estimates, the lower bound of $m$ is $p'$ in region I, and $m_1$ in both regions II and III. Equality is disallowed on region I, line $L_1$ and the line $L_3: \frac14\le \frac 1s \le \frac 13$, $q=1$.

For $L^s E^p_m$ estimates, the lower bound of $m$ is $p'$ in region I, $m_1$ in region II, and $p$ in region III. There is a jump across the line $L_2$. 
Equality is disallowed on line $L_3$ only.

\item The lower bound of $m$ for global $L^s E^p_m$ estimates has a jump across the line $L_2$. It is because our linear $L^s E^p_m$ estimates in Lemma \ref{lemma:linear.heat.amalgam.spacetime-new} are based on  $ E^{s,p}_{T=\infty,m}$ estimates and not optimal. We may hope to decrease the lower bound of $m$ in region III if we could prove Marcinkiewicz interpolation theorem for \emph{subadditive} maps on the Wiener amalgam spaces $E^{p}_q $.
See Remark \ref{rem25} (iii).

\item
For the classical case $q=3$, small $u_0 \in E^3_3 =L^3(\R^3)$, and $T_1=\infty$, 
we have $m_1(s,3)=p$ for $\frac2s + \frac3p = 1$ and 
Theorem \ref{thrm:critical2} gives
$u \in E^{s,p}_{\infty,m} \cap L^s_\infty E^p_m$ if $3< p \le 9$ and $m\ge p$, recovering the classical result of \cite{Kato,Giga}.

\item As soon as $q<3$, we can choose $m=\max(p',m_1)<p$ for the $E^{s,p}_{\infty,m}$ estimate.  
To choose $m<p$ for the $L^s_\infty E^p_m$ estimate, we also need $\frac 1q \ge \frac 5{3s}$.

\end{enumerate} 

We will prove estimates for the solutions of the Stokes equations \eqref{Stokes} in Section \ref{sec.stokes}, and construct mild solutions of Navier-Stokes equations \eqref{NS} in Section \ref{sec.ns}, proving Theorems \ref{thrm.subcritical}--\ref{thrm:critical2}.  

\subsection{Weak solutions in amalgam spaces}

The existence of global local energy weak solutions in Wiener amalgam spaces $E^2_q$, $2 \le q<\infty$, in $\R^3$  was considered by  the first and last author in \cite{BT4}, as a way to bridge (or interpolate) the classical theories in $L^2$ and $L^2_\uloc$. It turns out that the theory can be extended to $E^2_q$ for $1\le q<2$ (extrapolation).  We will study such local energy solutions, in the sense of  \cite[Definition 1.1]{BT4}, in Section \ref{sec:global-weak}.

Smaller $q$ means that the solutions are more spatially localized. An advantage is that we have a sequence of a priori bounds whose time spans go to infinity. Such property is shown in \cite{BT4} for $2\le q<6$. As a result, when we construct global solutions, there is no need to consider perturbed Navier-Stokes equations as in \cite{BT4}. On the other hand, as Young's convolution inequality for sequence
\[
\norm{a * b}_{\ell^r} \le C\norm{a }_{\ell^1} \cdot \norm{b}_{\ell^r} 
\]
is valid only for $r \ge 1$, we need to adjust many estimates of the pressure from far field.

As in \cite{BT4}, we will first show \emph{eventual regularity} for local energy weak solutions with $u_0 \in E^2_q$, $1\le q<2$, in Theorem \ref{bt4-thm-1.3}. 

Moreover, if such a solution has finite ${\bf LE}_q(0,T)$-norm 
\EQ{\label{ETq.def}
\norm{u}_{{\bf LE}_q(0,T)}
:= \norm{u}_{E^{\infty,2}_{T,q}} + \norm{\nb u}_{E^{2,2}_{T,q}},
}
which we refer to as 
$\ell^q$ \emph{local energy}, then it satisfied the a priori bounds in Lemma \ref{lem.A0qbound} for all scales up to time $T$.

Finally, we will prove the following existence theorem.

\begin{theorem}[Existence in $E^2_q$]\label{th4.8}
Assume $u_0\in E^2_q$ where $1\le q<2$ and is divergence free. Then, there exists a time-global local energy solution $u$ and associated pressure $\pi$ to the Navier-Stokes equations \eqref{NS} in $\R^3$  with initial data $u_0$ so that, for any $0<T<\infty$,
\EQ{\label{eq-LEq-global}
\norm{u}_{{\bf LE}_q(0,T)}<\infty.
}
In particular, $u\in L^\infty(0,T; E^2_q)$ and satisfies the a priori bounds in Lemma \ref{lem.A0qbound}. 
\end{theorem}

\emph{Comments on Theorem \ref{th4.8}:}
\begin{enumerate}

\item  Theorem \ref{th4.8} extends the range of $q\in[2,\infty)$ in \cite[Theorem 1.5]{BT4} to $q\in[1,2)$.
It can also be viewed as a $\ell^q$-version of \cite[Theorem 1.5]{BT8} except the decay condition \cite[(1.12)]{BT8} on initial data is not needed.

\item
The main point of Theorem \ref{th4.8} is the finiteness of estimate \eqref{eq-LEq-global}.
Since $u_0\in E^2_q\subset L^2$, Leray's original global weak solution is already a local energy solution, but it may not satisfy \eqref{eq-LEq-global}.
\end{enumerate}

\medskip

In Section \ref{Sec4.2}, we  discuss the local existence of a local energy solution $u$ of the perturbed Navier-Stokes equations
\eqref{PNS} when the perturbation $v$ is small in $L^\infty(0,T;L^p_\uloc)$, $p=3$. Such a solution $u$ is needed in the construction of time-global local energy solutions of \eqref{NS} with initial data in $E^2_q$, $2\le q<\infty$ in \cite{BT4}. The choice of $p=3$ is natural, but the bound $v \in L^\infty(0,T;L^3_\uloc)$ was insufficient
in \cite{BT4}, due to lack of compactness. It turns out the spacetime integral bound {\eqref{eq1.7}}  in Theorem \ref{thrm:critical} for small mild solution $v\in L^\infty L^3_\uloc$ can be used to show the local energy inequality of $u$.
This was actually a motivation for the present paper, and will be shown in Proposition \ref{corollary}.

 \medskip \noindent \textbf{Organization:} Section \ref{sec.stokes} contains estimates for the solutions of the Stokes equation while Section \ref{sec.ns} contains the proofs of our main results.  Section \ref{Sec4} considers weak solutions. In the appendix Section \ref{sec5}, we give details of the end point case of Giga's estimate \eqref{Giga-est}, and gives examples showing the strict inclusions of various functional spaces.

\section{Linear estimates in Wiener amalgam spaces}\label{sec.stokes}

In this section we consider linear equations in $\R^d$ for general space dimension $d\in \NN$. We first prove decay estimates of $e^{t\De}$, its gradient, and ${\mathbb P}e^{t\De}\nb\cdot$ and limits of $e^{t\De}$
in amalgam spaces in Lemmas \ref{Wa-estimate}--\ref{lemma:heat.eq.convergence.data}. We then show spacetime integral estimates for $e^{t\De}$ and the Duhamel term in Lemmas \ref{lemma:linear.heat.amalgam.spacetime-new} and \ref{lemma:duhamel.amalgam}-\ref{lem-bilinear-LsEpm} ($d=3$), and provide examples in Example \ref{example2}.
At the end of this section, we derive $\ell^q$ local energy estimates of $e^{t\De}$ and ${\mathbb P}e^{t\De}\nb\cdot$ for $d=3$ in Lemma \ref{KwTs-lem2.4}.

Let $S_{ij}(x,t)$ denote the Oseen tensor in $\R^d$, $d\ge2$, the fundamental solution of the Stokes system \eqref{Stokes} in $\R^d$, found by Oseen \cite{Oseen} for $d=3$. We have the following pointwise estimate by Solonnikov \cite{Solonnikov},
\EQ{\label{Oseen.est}
|\pd_t^m\nb_x^k S(x,t)|\le \frac{C_{k,m}}{(|x|+\sqrt t)^{d+k+2m}}.
}

\subsection{Decay estimates and continuity of heat semigroup}

\begin{lemma}\label{Wa-estimate}
Let {$d\in \NN$,} $1 \le \tilde{p} \le p \le \infty$, $1 \le \tilde{q} \le q \leq \infty$. Then for any $f\in E^{\tilde{p}}_{\tilde{q}}(\R^d)$, we have for $h=0,1$ 
\EQ{\label{lemma:linear.stokes.amalgam} 
\norm{\nb^h e^{t\De} f}_{E^p_q} 
\lec\bke{ \frac1{t^{\frac{d}2\bke{\frac1{\tilde{p}} - \frac1p} + \frac{h}2 } } + \frac {\one_{t>1}}{t^{\frac{d}2\bke{\frac1{\tilde{q}} - \frac1q} + \frac{h}2}  }  } \norm{f}_{E^{\tilde{p}}_{\tilde{q}}}.
}
For $d \ge 2$ and $F\in (E^{\tilde{p}}_{\tilde{q}})^{d\times d}$, we have
\EQ{\label{Wa-estimate-bilinear}
\norm{ \int_{\R^d}  \pd_l S_{ij}(x-y,t) F_{lj}(y)\, dy }_{E^p_q} 
\lec  \bke{  \frac1{t^{\frac{d}2\bke{\frac1{\tilde{p}} - \frac1p} + \frac12} }  + \frac {\one_{t>1}}{t^{\frac{d}2\bke{\frac1{\tilde{q}} - \frac1q} + \frac12}  } 
}\norm{F}_{E^{\tilde{p}}_{\tilde{q}}}.
}

When $\frac1{\tilde{p}} - \frac1p \le \frac1{\tilde{q}} - \frac1q$, \eqref{lemma:linear.stokes.amalgam}  and \eqref{Wa-estimate-bilinear} reduce to
\EQ{\label{lemma:linear.stokes.amalgam-2} 
\norm{e^{t\De} f}_{E^p_q} 
\lec \frac1{t^{\frac{d}2\bke{\frac1{\tilde{p}} - \frac1p}} }  \norm{f}_{E^{\tilde{p}}_{\tilde{q}}},
}
and
\EQ{\label{Wa-estimate-bilinear-2}
\norm{ \int_{\R^d}  \pd_l S_{ij}(x-y,t) F_{lj}(y)\, dy }_{E^p_q} 
\lec  \frac1{t^{\frac{d}2\bke{\frac1{\tilde{p}} - \frac1p} + \frac12} } \norm{F}_{E^{\tilde{p}}_{\tilde{q}}}.
}
\end{lemma}

The above lemma recovers the known cases of $L^{\td p}$-$L^p$ estimates \cite{FJR}  when $q=p$ and $\td q = \td p$ as well as $L^{\td p}_\uloc$-$L^{p}_\uloc$ estimates when $q=\td q = \I$ 
(\eqref{lemma:linear.stokes.amalgam} by \cite{Arriera} and \eqref{lemma:linear.stokes.amalgam}-\eqref{Wa-estimate-bilinear} by
\cite{MaTe}).  Estimates \eqref{lemma:linear.stokes.amalgam-2} and \eqref{Wa-estimate-bilinear-2}
are similar to the usual $L^{\td p}$-$L^p$ estimates, and are convenient for our applications in Section \ref{sec.ns}.
Note that $\td q\le q$ since we cannot improve decay, and $\td p \le p$ means we can improve regularity. We can allow $\td p > p$ by the imbedding $E^{\tilde{p}}_{\tilde{q}}\subset E^{p}_{\tilde{q}}$. In that case we replace $\bke{\frac1{\tilde{p}} - \frac1p}$ by $\bke{\frac1{\tilde{p}} - \frac1p}_+$. We will give two proofs of Lemma \ref{Wa-estimate}. {We do not use the case $h=1$ in \eqref{lemma:linear.stokes.amalgam} or later in \eqref{ineq:ForBusbySmith1}, but include it because its proofs are carried out simultaneously to those for $h=0$.}

\begin{proof}[First proof of Lemma \ref{Wa-estimate}]
For $x \in B_1(k)$, $k \in \ZZ^d$, decompose $\nb^h e^{t\De} f = f_1^k + f_2^k$ where $f_1^k = \nb^h e^{t\Delta} ( f\chi_{B_4(k)})$ and $f_2^k =\nb^h e^{t\Delta} (f (1-\chi_{B_4(k)}))$.
 By Minkowski's inequality
\[
\|\nb^h e^{t\Delta }f \|_{E^p_q}\leq \big\|  \| f_1^k\|_{L^p(B_1(k))}\big\|_{\ell^q} + \big\|  \| f_2^k\|_{L^p(B_1(k))}\big\|_{\ell^q}.
\]
We have by standard $L^{\td p}$-$L^p$ estimates that
\EQN{
\big\|  \| f_1^k\|_{L^p(B_1(k))}\big\|_{\ell^q} (t)&\lec t^{-\frac{h}2 - \frac d 2 (\frac 1 {\td p} -\frac 1 p)}\big\|  \| f\|_{L^{\td p}(B_4(k))}\big\|_{\ell^q}\lec t^{-\frac{h}2 - \frac d 2 (\frac 1 {\td p} -\frac 1p)} \|f\|_{E^{\td p}_{q}}\lec  t^{-\frac{h}2 - \frac d 2 (\frac 1 {\td p} -\frac 1p)} \|f\|_{E^{\td p}_{\td q}},
}
using the embedding $\ell^{\td q}\subset \ell^q$.
On the other hand, for $x\in B_1(k)$ we have 
\EQN{
|f_2^k|(x,t)\lec \sum_{|k'|\geq 1} \frac 1 {t^{\frac d2 + h}} |k'|^h e^{-|k'|^2/(4t)}  |B_1|^{1-\frac 1 {\td p}} \| f\|_{L^{\td p } (B_1(k'-k))},
}
and, consequently, 
\EQN{
\| f_2^k\|_{L^p(B_1(k))}(t)\lec |B_1|^{1-\frac 1 {\td p} {+\frac 1 p}} \sum_{|k'|\geq 1} \frac 1 {t^{\frac d2 + h}} |k'|^h e^{-|k'|^2/(4t)} \| f\|_{L^{\td p } (B_1(k'-k))}.
}
For small $t\le 1$, applying the $\ell^q$ norm and using Young's inequality on the discrete convolution leads to
\EQN{
\big\|\| f_2^k\|_{L^p(B_1(k))} \big\|_{\ell^q}(t) &\leq  |B_1|^{1-\frac 1 {\td p}{+\frac 1 p}} \norm{ t^{-\frac d2 - h} |k'|^h e^{-|k'|^2/(4t)}\chi_{k'\neq 0} }_{\ell^{1}} \| f\|_{E^{\td p}_{ q}} 
\lec t^{-\frac{h}2}\|f\|_{E^{\td p}_{\td q}},
}
where we've used the embedding $\ell^{\td q}\subset \ell^q$. For large $t>1$, we get
\EQN{
\big\|\| f_2^k\|_{L^p(B_1(k))} \big\|_{\ell^q}(t) &\leq  |B_1|^{1-\frac 1 {\td p}+\frac 1 p} \| t^{-\frac d2 - h} |k'|^h e^{-|k'|^2/(4t)}\chi_{k'\neq 0} \|_{\ell^{r}} \| f\|_{E^{\td p}_{\td q}},\quad \frac1q + 1 = \frac1r + \frac1{\td q},\\
&\lec t^{-\frac{h}2 - \frac{d}2\bke{\frac1{\td q}-\frac1q}} \|f\|_{E^{\td p}_{\td q}}.
}
This proves \eqref{lemma:linear.stokes.amalgam}.

\bigskip The proof of \eqref{Wa-estimate-bilinear} is logically similar except we set $F_1^k = \int_{\R^d}  \pd_l S_{ij}(x-y,t) (F_{lj}\chi_{B_4(k)})(y)\, dy$ and $F_2^k = \int_{\R^d}  \pd_l S_{ij}(x-y,t) (F_{lj}(1-\chi_{B_4(k)}))(y)\, dy$. Using Young's convolution inequality,
\[
\norm{F_1^k}_{L^p(B_1(k))} 
\le \norm{F_1^k}_{L^p(\R^d)}
\lec \norm{\pd_l S_{ij} }_{L^r(\R^d)} \norm{F\chi_{B_4(k)}}_{L^{\td p}},\quad \frac1r+\frac1{\td p} = \frac1p + 1.
\]
Then, by Oseen tensor estimate \eqref{Oseen.est}, we have
\[
\norm{\pd_l S_{ij} }_{L^r(\R^d)} \lec t^{-\frac12 + \frac{d}2 \bke{\frac1p - \frac1{\td p}} }.
\]
So 
\EQN{
\norm{F_1^k}_{L^p(B_1(k))} 
&\lec t^{-\frac12 + \frac{d}2 \bke{\frac1p - \frac1{\td p}} } \norm{F}_{L^{\td p}(B_4(k))}.
}
Taking $\ell^q$-norm over $k\in\ZZ^d$ on both sides and using the embedding $\ell^{\td q}\subset\ell^q$ for $\td q\le q$, we derive
\EQN{
\norm{ \norm{F_1^k}_{L^p(B_1(k))} }_{\ell^q}(t)
&\lec t^{-\frac12 + \frac{d}2 \bke{\frac1p - \frac1{\td p}} } \norm{F}_{E^{{\td p}}_q}
\lec t^{-\frac12 + \frac{d}2 \bke{\frac1p - \frac1{\td p}} } \norm{F}_{E^{{\td p}}_{\td q}}.
}

\medskip

For %
$F_2^k$, we carry out the same estimates as for $f_2^k$ but replace $t^{-d/2}e^{-|k'|^2/(4t)}\chi_{k'\neq 0} $ by  $(|k'|+\sqrt t)^{-(d+1)} \chi_{k'\neq 0} $. For small $t$, since 
\[
\|(|k'|+\sqrt t)^{-(d+1)} \chi_{k'\neq 0} \|_{\ell^{1}} \lec t^{-1/2},
\]
we obtain
\[
\big\|\| F_2^k\|_{L^p(B_1(k))} \big\|_{\ell^q}(t) \lec t^{-\frac 1 2 }  \| F\|_{E^{\td p}_{q}}\lec t^{-\frac 1 2 }  \| F\|_{E^{\td p}_{\td q}}.
\]
For large $t$, we instead use
\EQN{
\|(|k'|+\sqrt t)^{-(d+1)} \chi_{k'\neq 0} \|_{\ell^r} &\leq t^{\frac d {2r} -\frac d 2 -\frac 1 2},\quad \frac1q+1=\frac1r+\frac1{\td q},\\
&= t^{-\frac 1 2 - \frac d 2 (\frac 1 {\td q} -\frac 1 {q})},
}
and obtain
\[
\big\|\| F_2^k\|_{L^p(B_1(k))} \big\|_{\ell^q}(t) \leq t^{-\frac 1 2 - \frac d 2 (\frac 1 {\td q} -\frac 1 {q})}  \| F\|_{E^{\td p}_{\td q}}.
\]
Combining these estimates for $F_1^k$ and $F_2^k$ leads to the desired bound \eqref{Wa-estimate-bilinear}.
\end{proof}

It is worth noting that Lemma \ref{Wa-estimate} can also be proved by using a Young-type convolution inequality for amalgam spaces from Busby and Smith \cite{BS}. We include the details to paint a complete picture of the available tools in the amalgam spaces.
The aforementioned inequality, \cite[Theorem 4.2]{BS}, states: For $1\le p_1,p_2,q_1,q_2\le\infty$ with $1/p_2+1/q_2\ge1$, 
\EQ{
1/r_1 = \max\{0,1/p_1 + 1/q_1 - 1\}\quad \text{ and }\quad
1/r_2 = 1/p_2 + 1/q_2 - 1,
}
we have
\EQ{\label{ineq:BusbySmith}
\norm{f * g}_{E^{r_1}_{r_2}} 
\lec \norm{f}_{E^{p_1}_{p_2}} \norm{g}_{E^{q_1}_{q_2}}.
}
To use this inequality, we first establish an elementary lemma.

\begin{lemma}\label{lemma:ForBusbySmith1}
Let $d \in \NN$ and $p,q \in [1,\oo]$.
For the heat kernel  $\Gamma(x,t)=(4\pi t)^{-d/2} e^{-x^2/4t}$ in $\R^d$,  we have for ${h}=0,1$
\EQ{\label{ineq:ForBusbySmith1}
\norm{\nb^h\Gamma(\cdot,t)}_{E^{p}_{q}} 
\le C  t^{-\frac{h}2} \bke{ t^{- \frac d2 + \frac d{2p}} + 
\one_{t>1} t^{-\frac{d}2 + \frac{d}{2q}} }, \quad (0<t<\oo).
}
For $\Phi(x,t)=(|x|+\sqrt t)^{-d-1}$ defined for $(x,t)\in \R^d \times \R_+$, we have 
\EQ{\label{ineq:ForBusbySmith2}
\norm{\Phi(\cdot,t)}_{E^{p}_{q}} 
\le C t^{-\frac 12} \bke{ t^{-\frac d2 + \frac d{2p}} +  
\one_{t>1} t^{-\frac{d}2 + \frac{d}{2q}} }, \quad (0<t<\oo).
}
\end{lemma}

\begin{proof}
First, we have
\[
\norm{\nb^h \Ga(\cdot,t)}_{E^{p}_{1}} 
\lec
 \sum_{k\in \ZZ^d}\bke{\int_{B_1(k)} \abs{t^{-\frac d2-h} |x|^h e^{-\frac{|x|^2}{4t}}}^{p}\, dx}^{1/p}.
 \]
Note that if $|k|\geq 3$ and $x\in B_1(k)$, then $|x| \in [\frac 23|k|,\frac 43|k|]$.
Hence,
\EQN{ %
\norm{\nb^h\Gamma(\cdot,t)}_{E^{p}_{1}} 
&\lec
\bke{\int_{B_4(0)} |\nb^h \Ga(x,t)|^{p}\, dx}^{1/p}+
 \sum_{|k|\ge 3} \bke{\int_{B_1(k)} \abs{t^{-\frac d2-h} |k|^h e^{-\frac{|k|^2}{20t}}}^{p}\, dx}^{1/p}
\\&\lec
\bke{\int_{\R^d} |\nb^h \Ga(x,t)|^{p}\, dx}^{1/p}+
 \sum_{|k|\ge 3} t^{-\frac d2-h} |k|^h e^{-\frac{|k|^2}{20t}}
 \\&
\lec  t^{-\frac{h}2 -\frac{d}2+ \frac d{2p}} + t^{-\frac d2 - h} \int_{\R^d} |y|^h e^{-\frac{|y|^2}{20t}}\, dy
\lec t^{-\frac{h}2-\frac{d}2+ \frac d{2p}}+t^{-\frac{h}2}.
}
The embedding $\ell^1 \subset \ell^{q}$ implies 
\[
\norm{\nb^h \Gamma(\cdot,t)}_{E^{p}_{q}} 
\le C \bke{ t^{-\frac{h}2 -\frac d2 + \frac d{2p}} + t^{-\frac{h}2}}.
\]
On the other hand, 
\EQN{
\norm{\nb^h \Ga(\cdot, t)}_{E^p_q} 
&\lec \bkt{\bke{\int_{B_4(0)} |\nb^h \Ga(x,t)|^p\, dx}^{q/p} + \sum_{|k|\ge 3} \bke{\int_{B_1(k)} \abs{t^{-\frac d2 - h} |k|^h e^{\frac{-|k|^2}{20t}} }^p dx}^{q/p} }^{1/q}\\
&\lec \bkt{\bke{\int_{\R^d} |\nb^h \Ga(x,t)|^p\, dx}^{q/p} + \sum_{|k|\ge 3} t^{-(\frac{d}2 + h)q} |k|^{hq} e^{-\frac{|k|^2}{20t} q} }^{1/q}\\
&\lec \bkt{ t^{\bke{-\frac{h}2 -\frac{d}2 + \frac{d}{2p} }q } + t^{-(\frac{d}2 + h) q} \int_{\R^d} |y|^{hq} e^{-\frac{|y|^2}{20t} q}\, dy }^{1/q}\\
&\lec \bkt{ t^{\bke{-\frac{h}2 -\frac{d}2 + \frac{d}{2p} }q } + t^{-(\frac{d}2 + h) q + \frac{d}2 + \frac{hq}2} }^{1/q}
\sim t^{-\frac{h}2 -\frac d2 + \frac d{2p}} + t^{-\frac{h}2 - \frac d2 + \frac d{2q}}.
}
This proves \eqref{ineq:ForBusbySmith1}.
Similarly,
\EQN{ %
\norm{\Phi(\cdot,t)}_{E^{p}_{1}} 
&\lec
\bke{\int_{B_4(0)} \Phi(x,t)^{p}\, dx}^{1/p}+
 \sum_{|k|\ge 3} \bke{\int_{B_1(k)}  \Phi(k,t)^{p}\, dx}^{1/p}
 \\
&\lec
\bke{\int_{\R^d} \Phi(x,t)^{p}\, dx}^{1/p}+
 \sum_{|k|\ge 3} \Phi(k,t)
 \\
&
\lec  t^{-\frac{d}2+ \frac d{2p}-\frac 12}+ \int_{\R^d} \Phi(x,t)\, dx
\lec t^{-\frac{d}2+ \frac d{2p}-\frac12}+t^{-\frac12}.
}
Also,
\EQN{
\norm{\Phi(\cdot, t)}_{E^p_q} 
&\lec \bkt{\bke{\int_{B_4(0)} \Phi(x,t)^p\, dx}^{q/p} + \sum_{|k|\ge 3} \bke{\int_{B_1(k)} \Phi(x,t)^p dx}^{q/p} }^{1/q}\\
&\lec \bkt{\bke{\int_{\R^d} \Phi(x,t)^p\, dx}^{q/p} + \sum_{|k|\ge 3} \Phi(k,t)^q }^{1/q}\\
&\lec \bkt{ t^{\bke{-\frac{d}2+ \frac d{2p}-\frac 12}q} + \int_{\R^d} \Phi(x,t)^q\, dx }^{1/q}
\lec t^{-\frac{d}2+ \frac d{2p}-\frac12} + t^{-\frac{d}2+ \frac d{2q}-\frac12 }.
}
This proves \eqref{ineq:ForBusbySmith2} and completes the proof of Lemma \ref{lemma:ForBusbySmith1}.
\end{proof}
\emph{Remark.} Estimate \eqref{ineq:ForBusbySmith1} for $h=1$ also follows from \eqref{ineq:ForBusbySmith2} since $|\nb \Ga|\lec \Phi$.

\begin{proof}[Second proof of Lemma \ref{Wa-estimate}]
Define $p_0$ and $q_0$ by
\[
\frac 1 p =   \frac 1 {p_0}+\frac 1 {\td p} -1  \text{ and }  \frac 1 q  = \frac 1 {q_0} +\frac 1 {\td q} -1, 
\]
which implies $\frac 1 {q_0} + \frac 1 {\td q} \geq 1$.
We have 
\Eq{ 
\norm{\nb^h e^{t\De} f}_{E^p_q} 
\lec   \| \nb^h \Ga\|_{E^{p_0}_{q_0}} \|f\|_{E^{\td p}_{\td q}}
\lec \frac1{t^{\frac{h}2}}\bke{ \frac1{t^{\frac{d}2\bke{\frac1{\tilde{p}} - \frac1p}} }+
{\frac {\one_{t>1}}{t^{\frac{d}2\bke{\frac1{\tilde{q}} - \frac1q}}  } } } \norm{f}_{E^{\tilde{p}}_{\tilde{q}}},
}
where we first use Busby and Smith's inequality \eqref{ineq:BusbySmith}, and then \eqref{ineq:ForBusbySmith1} of Lemma \ref{lemma:ForBusbySmith1}. This proves \eqref{lemma:linear.stokes.amalgam}. To prove \eqref{Wa-estimate-bilinear}, we replace $ \| \Ga\|_{E^{p_0}_{q_0}}$ above by  $\| \nb S\|_{E^{p_0}_{q_0}}$, use the upper bound on first spatial derivatives of the Oseen tensor \eqref{Oseen.est}, and use \eqref{ineq:ForBusbySmith2} of Lemma \ref{lemma:ForBusbySmith1}.
\end{proof}

When $p,q<\I$, the spaces $E^p_q$ preserve the good convergence properties of $e^{t\Delta}u_0$ to $u_0$ as $t\to 0$, as the next lemma shows.   If $p=\I$ or $q=\I$ the lemma fails. For $q=\I$ this failure is detailed in \cite{MaTe}.%

\begin{lemma}
[Continuity and vanishing of $e^{t\De}f$ in $E^p_q$]
\label{lemma:heat.eq.convergence.data}
Assume $f\in E^{p}_q=E^{p}_q(\R^d)$, $d\in \NN$, $1\leq p < \I$ and $1\leq q<\I$. Then 
\EQ{\label{eq2.4}
\lim_{|\tau|\to 0} \|  f(\cdot + \tau ) - f(\cdot )\|_{E^p_q}   = 0,
}
and
\EQ{\label{eq2.5}
\lim_{t\to 0_+} \|e^{t\Delta}f - f\|_{E^p_q} = 0.
}
Moreover, for $t>0$ and $p,q\in [1,\infty]$, \textup{(}including $q=\I$ or $p=\I)$, we have 
\EQ{\label{eq2.12}
\lim_{h\to 0} \|e^{(t+h)\De} f  -e^{t\De}  f\|_{E^p_q} = 0,
}
with no restriction on the sign of $h$.

Moreover, if $1 \le \tilde{p} < p \le \infty$ (with strict inequality) and $1 \le \tilde{q} \le q \leq \infty$, 
then for any $f\in E^{\tilde{p}}_{\tilde{q}}(\R^d)$, and $f\in E^{\td p}$ if $\td q = q=\oo$,
we have
\EQ{\label{vanishing-t=0}
\lim _{t\to0_+} t^{\frac{d}2\bke{\frac1{\tilde{p}} - \frac1p}} \norm{e^{t\De} f}_{E^p_q}  = 0.
}
\end{lemma}

When $\td q = q=\oo$, \eqref{vanishing-t=0} is false if we only assume $f\in E^{\td p}_\oo=L^{\td p}_\uloc$; see Example \ref{example2}. However, it is true for $f\in E^{\td p}$, i.e., the closure of $C^\infty_c$ in $L^{\td p}_\uloc$, as the following proof using a density argument still works.

\begin{proof}
For \eqref{eq2.4}, note that for $\tau\in \R^d$ we have the uniform bound
\[
 \|  f(\cdot + \tau ) - f(\cdot )\|_{E^p_q}  \leq  \|  f(\cdot + \tau ) \|_{E^p_q} +  \|  f\|_{E^p_q}  \lec \|  f\|_{E^p_q}
<\I.
\]
For $|\tau|<1$, we can make  $ \big\|\|  f(\cdot + \tau ) - f(\cdot )\|_{L^p(B_1(k))} \big\|_{\ell^q(|k|>m)}$ arbitrarily small by taking $m$ sufficiently large. Once $m$ is fixed, $\big\|\|  f(\cdot + \tau ) - f(\cdot )\|_{L^p(B_1(k))} \big\|_{\ell^q( |k|\le m)}\to 0$ as $|\tau|\to0$ by properties of $L^p$-spaces.  These show \eqref{eq2.4}.

For \eqref{eq2.5}, note that
\[
(e^{t\Delta}f - f)(x) = \int_{\R^d} e^{-|z|^2 /4}  g(x,z,t)\,dz, \quad g(x,z,t) = f(x-\sqrt t z) - f(x).
\]
By Minkowski's integral inequality in $x \in B_1(k)$,
\[
\norm{e^{t\Delta}f - f}_{L^p_x(B_1(k))} \le \int_{\R^d} e^{-|z|^2 /4}  \norm{ g(\cdot,z,t)}_{L^p_x(B_1(k))}\,dz, 
\]
By Minkowski's integral inequality again in $k \in \ZZ^d$,
\EQN{
\| e^{t\Delta}f - f\|_{E^p_q} & = \norm{\norm{e^{t\Delta}f - f}_{L^p_x(B_1(k))} }_{\ell^q_k}
\\
& \le \int_{\R^d} e^{-|z|^2 /4}  \norm{ \norm{ g(\cdot,z,t)}_{L^p_x(B_1(k))}}_{\ell^q_k}\,dz
\\
&=\int_{\R^d} e^{-|z|^2 /4}  \| f(\cdot-\sqrt t z) - f(\cdot)\|_{E^p_q}\,dz.
}
The last integral vanishes as $t\to 0_+$ by the dominated convergence theorem and \eqref{eq2.4}. This shows \eqref{eq2.5}.

We now show \eqref{eq2.12}, continuity at $t>0$. For any $p,q\in[1,\infty]$, by the Busby-Smith convolution inequality \eqref{ineq:BusbySmith}, we have 
\[
\norm{e^{(t+h)\De}f - e^{t\De}f}_{E^p_q} \lec \norm{\Ga_{t+h} - \Ga_t}_{L^1} \norm{f}_{E^p_q} \to0,
\]
as $h\to0$ from either the left or right, by the dominated convergence theorem with $|\Ga_{t+h} - \Ga_t| \le Ct^{-d/2} e^{-|x|^2/(6t)}$ when $|h|<t/2$.

We now show \eqref{vanishing-t=0}.  Since we will send $t \to 0_+$, we assume $t\leq 1$. Denote $\si =\frac{d}2\bke{\frac1{\tilde{p}} - \frac1p} >0 $.
For any $\e>0$, we can choose $b \in  C^\infty_c$ %
with $\norm{f-b}_{ E^{\tilde{p}}_{\tilde{q}}} \le \e/4C$ where $C$ is the constant in \eqref{lemma:linear.stokes.amalgam}.
Then by \eqref{lemma:linear.stokes.amalgam},
\[
t^\si \norm{e^{t\De}(f-b)}_{ E^p_q} \le 2C\norm{f-b}_{ E^{\tilde{p}}_{\tilde{q}}}
\le
\e/2.
\]
By \eqref{lemma:linear.stokes.amalgam} again, 
\[
t^\si \norm{e^{t\De}b}_{ E^p_q}\le \td  Ct^\si    \norm{b}_{ E^{{{ p}}}_{\tilde{q}}}  \to 0
\quad \text{as } t \to 0_+.
\]
These show \eqref{vanishing-t=0}.
\end{proof}

In the following we give a direct proof of \eqref{eq2.5} which does not use \eqref{eq2.4}. 

\begin{proof}[Second proof of \eqref{eq2.5}]
Fix an arbitrarily small $\e>0$. We may write 
\[
e^{t\Delta} f = e^{t\Delta} (f \chi_{B_R}) + e^{t\Delta} (f (1-\chi_{B_R})).
\]
For large enough $R$ we have by \eqref{lemma:linear.stokes.amalgam} that
\[
\| e^{t\Delta } (f (1-\chi_{B_R}) ) - f (1-\chi_{B_R})\|_{E^{p}_q} \lec \| f (1-\chi_{B_R})\|_{E^p_q} <\e/2,
\]
by summability in $\ell^q$.
We now show for sufficiently large $R$ and small $t$ that
\[
 \| e^{t\Delta} (f \chi_{B_R}) - f \chi_{B_R} \|_{E^p_q} <\e /2.
\]
We treat this term in two steps using the decomposition
\EQ{\label{eq2.9}
e^{t\Delta} (f \chi_{B_R}) - f \chi_{B_R} = \chi_{B_{2R}} (e^{t\Delta} (f \chi_{B_R}) - f \chi_{B_R} ) + (1- \chi_{B_{2R}} )e^{t\Delta} (f \chi_{B_R}).
}
Focusing on the  second term, we have for $|k|\geq 2R$ that
\[
\|  e^{t\Delta} (f \chi_{B_R}) \|_{L^p(B_1 (k))} \lec \frac 1 {t^{d/2}} e^{-c|k|^2/t} R^{{d-d/q}} \|f\|_{E^p_q}.
\]

Hence, noting $\ell^1\subset \ell^q$,
\EQN{
\|(1- \chi_{B_{2R}} )e^{t\Delta} (f \chi_{B_R})\|_{E^p_q} &\lec  \|  {t^{-d/2}} e^{-c|k|^2/t} \|_{l^q(|k|\geq 2R)} R^{{d-d/q}} \|f\|_{E^p_q}
\\&\lec  \|  {t^{-d/2}} e^{-c|k|^2/t} \|_{l^1(|k|\geq 2R)} R^{{d-d/q}} \|f\|_{E^p_q}
\\&\to 0 \text{ as }R\to \I,
}
provided $t\leq 1$.
Hence we choose $R$ large enough that the left hand side above is bounded by $\e/4$.
Then,  the first term of \eqref{eq2.9} is bounded by
\EQ{
\|\chi_{B_{2R}} (e^{t\Delta} (f \chi_{B_R}) - f \chi_{B_R} )\|_{E^p_q} \lec_{R} \|  (e^{t\Delta} (f \chi_{B_R}) - f \chi_{B_R} )\|_{L^p  } \to 0\text{ as } t\to 0.
}
Since $\e>0$ is arbitrarily small, the above shows \eqref{eq2.5}.
\end{proof}

\subsection{Space time integral estimates}

We now develop spacetime integral bounds for the heat equation with data in $E^r_q$. Recall that the spacetime integrals $L^s_T E^p_q$ and $E^{s,p}_{T,q}$ are defined in \eqref{LsEpq-def} and \eqref{Espq-def}.

\begin{lemma}\label{lemma:linear.heat.amalgam.spacetime-new}

Let $d\in \NN$, $1<r\le s\le\infty$, $r\le p<\infty$,
$\frac 2s + \frac dp = \frac dr$, and $1 \le q \le m \le \infty$. Suppose $f \in E^r_q(\R^d)$.

(a) For $0<  T<\infty$, we have 
\begin{align}
\label{lemma:linear.sotkes.amalgam.a}
\norm{e^{t\Delta}f}_{E^{s,p}_{T,m}} &\lesssim (1 + T^\be )\|f\|_{E^r_q},
\\
\label{lemma:linear.sotkes.amalgam.b}
\mathbbm{1}_{\{q\le s\}}
  \norm{e^{t\De} f}_{L^s_T E^p_{m} }  
& \lesssim \bke{1+ \mathbbm{1}_{\{m\le s\}} T^\be +\mathbbm{1}_{\{q\le s< m\}\setminus\{q\le r,p\le m \}}T^{\be'}  }\|f\|_{E^r_q},
\end{align}
for any $\be,\be' \in [0,\infty)$ and $\be > \al$, $\be' > \al'$ with
\[
\al = \frac d{2m}-\frac d{2q}+ \frac 1s,\quad
\al' = \frac d{2s}-\frac d{2q}+ \frac 1s. 
\]
We can take $\be=\al$ if $ \al \ge 0$ and $1<q<m<\infty$, and we can take $\be'=\al'$ if $ \al' \ge 0$ and $1<q<s$.

(b) For $T=\infty$, we have 
\begin{align}
\label{lemma:linear.sotkes.amalgam.c}
\one_A \norm{e^{t\Delta}f}_{E^{s,p}_{T=\infty,m}} 
&\lesssim  \|f\|_{E^r_q},
\\
\label{lemma:linear.sotkes.amalgam.d}
(\mathbbm{1}_{\{m\le s\} \cap A}  + \mathbbm{1}_E) \norm{e^{t\De} f}_{L^s_{T=\infty} E^p_{m} } & 
\lesssim  \|f\|_{E^r_q},
\end{align}
where $A$ is the subset of parameters that further satisfy
\[
A=\bket{\al\le 0; \ 1<q<m<\infty \text{ if }\al=0},
\]
and
$E=E_1\cup E_2$ is the subset of parameters of $\{q \le s < m\}$, unrelated to $A$,%
\EQN{
E_1&=\bket{ q \le r \le s \le p \le m,\, s\neq m,\, r\neq p} , \\
E_2&=\bket{ s<m \text{ and } q \le \frac {ds}{d+2}\ \text{ with }\  q>1\text{ if }q = \frac {ds}{d+2}}.
}

\end{lemma}

\begin{remark}\label{rem25}
(i) For finite $T<\infty$, we can allow $\frac 2s + \frac dp > \frac dr$ for estimates \eqref{lemma:linear.sotkes.amalgam.a}--\eqref{lemma:linear.sotkes.amalgam.b}. Since such estimates are not optimal and \eqref{lemma:linear.sotkes.amalgam.b}  follows directly from Lemma  \ref{Wa-estimate}, they are avoided for simplicity.
\medskip

(ii) This lemma extends Giga \cite[Lemma, p.{~196}]{Giga}. Indeed, if we take 
$1<q=r<m=p<\infty$ (so that $\al=0$) in \eqref{lemma:linear.sotkes.amalgam.d}, then we get
\[
\norm{e^{t\De} f}_{L^s_{T=\infty} L^p} \lec \norm{f}_{L^r},\quad 1<r\le s\le\infty,\quad r\le p<\infty,
 \quad
\frac 2s + \frac dp = \frac dr,
\]
which recovers \eqref{Giga-est}--\eqref{0812b}. 

\medskip
(iii) %
Giga's estimate \eqref{Giga-est}
is proved in \cite{Giga} using the Marcinkiewicz interpolation theorem (MIT) for \emph{subadditive} maps on $L^p$ (\cite[Appendix]{Stein70}). As our proof below uses \eqref{Giga-est}, it uses MIT implicitly. It is natural to ask if we can prove Lemma \ref{lemma:linear.heat.amalgam.spacetime-new} by MIT directly. 
Lakey \cite{Lakey} shows that Wiener amalgam spaces $E^{p}_q $ are interpolation spaces, and a Riesz-Thorin type interpolation theorem is valid on $E^{p}_q $ (\cite[Theorem 4.5]{Lakey}). However, it is unclear whether the Marcinkiewicz interpolation theorem for subadditive maps holds on $E^{{p}}_q $. If yes, then we may {try to} prove \eqref{lemma:linear.sotkes.amalgam.d} directly and {possibly} enlarge the set of admissible parameters. 

\medskip
(iv) Recall $q\le m$. When $m \le s$, we have estimates of $\norm{e^{t\De} f}_{L^s_{T=\infty} E^p_{m} }$ in
\eqref{lemma:linear.sotkes.amalgam.b} and \eqref{lemma:linear.sotkes.amalgam.d}. When $s<q$,  $\norm{e^{t\De} f}_{L^s_{T<\infty} E^p_{m} }$ cannot be controlled by $ \|f\|_{E^r_q}$,
as to be shown by Example \ref{example2}.
For the remaining case $q\le s<m$, we have finite time estimate in \eqref{lemma:linear.sotkes.amalgam.b}, but no infinite time estimate \eqref{lemma:linear.sotkes.amalgam.d}
unless the parameters belong to the set $E$. Its complement set $\{q\le s<m\} \setminus E$ is unclear to us.

\end{remark}

\begin{proof} 
We first prove the $E^{s,p}_{T,m}$ estimates.
For every $k\in\ZZ^d$ and $x\in B_1(k)$, we decompose
\[
e^{t\De} f (x,t)= e^{t\De} (f\chi_{B_4(k)}) + e^{t\De} (f(1-\chi_{B_4(k)})) =: f_1^k(x,t) + f_2^k(x,t).
\]
Then
\EQN{
&\norm{e^{t\De} f}_{E^{s,p}_{T,m}}
= \norm{ \norm{ \norm{ e^{t\De} f }_{L^p(B_1(k))} }_{L^s(0,T)} }_{\ell^{m}(k\in\ZZ^d)}\\
&\le \norm{ \norm{ \norm{ f_1^k(\cdot,t) }_{L^p(B_1(k))} }_{L^s(0,T)} }_{\ell^{m}(k\in\ZZ^d)} + \norm{ \norm{ \norm{ f_2^k(\cdot,t) }_{L^p(B_1(k))} }_{L^s(0,T)} }_{\ell^{m}(k\in\ZZ^d)}\\
&=: A_1 + A_2.
} 
For $A_1$, we use  Giga's integral estimate \eqref{Giga-est} to get
\[
\norm{ \norm{f_1^k(\cdot,t)}_{L^p(B_1(k))} }_{L^s(0,T)} 
\le \norm{e^{t\De}(f\chi_{B_4(k)})}_{L^s(0,\infty; L^{p}(\R^d))}\lec 
\norm{f}_{L^r(B_4(k))}.
\]
Hence
\EQ{\label{0402-a}
A_1 \le \norm{\norm{f}_{L^r(B_4(k))} }_{\ell^{m}(k\in\ZZ^d)} \approx \|f\|_{E^r_{m}} \lec  \|f\|_{E^r_q},
}
since $q\leq m$ so $\ell^{q}$ embeds continuously in $\ell^{m}$.

\medskip

For $A_2$, as
\[
|f^k_2(x,t)| \lec \sum_{k':\ |k'-k|>4} \int_{B_1(k)} t^{-d/2} e^{- \frac {|x-y|^2}{4t}} |f(y)|\,dy,
\]
we have
\[
\norm{f^k_2(\cdot,t)}_{L^p(B_1(k))} \lec t^{-d/2} \sum_{|k'|\ge1} e^{-|k'|^2/(4t)} \norm{f}_{L^r(B_1(k'-k))}.
\]
Taking $L^s_T$ on both sides and applying Minkowski's integral inequality, we get
\EQN{
\norm{f^k_2}_{L^s_TL^p(B_1(k))} &\lec \bigg[ \int_0^T \bigg( t^{-d/2} \sum_{|k'|\ge1} e^{-|k'|^2/(4t)} \norm{f}_{L^r(B_1(k'-k))}\bigg) ^s dt \bigg]^{1/s}\\
&\le \sum_{|k'|\ge1} \bke{\int_0^T \abs{t^{-d/2} e^{-|k'|^2/(4t)} \norm{f}_{L^r(B_1(k'-k))} }^s dt}^{1/s}\\
&= \sum_{|k'|\ge1} \norm{f}_{L^r(B_1(k'-k))} \bke{\int_0^T t^{-\frac{d}2 s}\, e^{-\frac{|k'|^2}{4t} s}\, dt}^{1/s},\quad t=|k'|^2 \tau\\
&= \sum_{|k'|\ge1} \norm{f}_{L^r(B_1(k'-k))} |k'|^{-d+\frac2s} \bke{\int_0^{\frac{T}{|k'|^2}} \tau^{-\frac{d}2 s}\, e^{-\frac{s}{4\tau} }\, d\tau}^{1/s}.
}
Since $\frac dr=\frac2s+\frac dp>\frac2s$, we have
$ds>2$ and
\[
\int_0^\infty \tau^{-\frac{d}2 s}\, e^{-\frac{s}{4\tau} }\, d\tau < \infty.
\]
On the other hand, if $S<1$,
\[
\int_0^S \tau^{-\frac{d}2 s}\, e^{-\frac{s}{4\tau} }\, d\tau
\lec \int_0^S  e^{-\frac{s}{8\tau} }\, d\tau \lec e^{-\frac{s}{8S} }.
\]
This estimate also holds for $S>1$. 
We conclude
\EQ{
\norm{f^k_2}_{L^s_TL^p(B_1(k))} \lec  \sum_{|k'|\ge1} \norm{f}_{L^r(B_1(k'-k))} |k'|^{-d+\frac2s} e^{-\frac {|k'|^2}{8T}}
.
}

For fixed $0<T< \infty$, we can bound for any $\be \in [0,\oo)$
\[
|k'|^{-d+\frac2s}
 {e^{-\frac {|k'|^2}{8T}}}
\lec |k'|^{-d+\frac2s} 
 {\bke{\frac {|k'|^2}{8T}}^{-\be}}
\lec |k'|^{-d+\frac2s-2\be} T^{\be}.
\]
By the Young's convolution inequality, 
\Eq{
A_2 \lec \norm{\sum_{|k'|\ge1} \norm{f}_{L^r(B_1(k'-k))}  |k'|^{-d+\frac2s-2\be} T^{\be} }_{\ell^{m}(k\in\ZZ^d)}
\lec \norm{f}_{E^r_q}\norm{|k|^{-d+\frac2s-2\be} T^{\be} }_{\ell^{n}(k\not=0\in\ZZ^d)},
}
where $\frac 1{m}+1=\frac1q+\frac 1n$, and $n \ge 1$ thanks to $m \ge q$. For the last norm to be finite, we need
\[
d-\frac2s+2\be > \frac dn = \frac d{m}+d-\frac dq,
\]
that is, $2\be > \frac d{m}-\frac dq+ \frac 2s$. Thus if $\frac d{m}-\frac dq+ \frac 2s<0$, we can take $\be=0$. If  $\frac d{m}-\frac dq+ \frac 2s\ge 0$, we take any $\be>\al$, with
\[
\al = \frac d{2m}-\frac d{2q}+ \frac 1s.
\]
In all cases, we obtain 
\EQ{\label{0125a}
A_2\lec \norm{f}_{E^r_q} T^{\be}.
}
If $\al \ge 0$ and $1<q<m<\infty$,  we can choose $\be=\al$ and use the discrete version of the Hardy--Littlewood--Sobolev inequality in \cite[Proposition (a)]{Stein-Wainger-discreteHLS} to derive \eqref{0125a}, noting that
$d-\frac2s+2\al = \frac d{m}+d-\frac dq\in (0,d)$.

Eq.~\eqref{0402-a} and \eqref{0125a} show the $E^{s,p}_{T,m}$ estimate in 
\eqref{lemma:linear.sotkes.amalgam.a}. 

For $L_T^sE^p_m$ estimate in \eqref{lemma:linear.sotkes.amalgam.b}, note that it follows from the $E^{s,p}_{T,m}$ estimates when $m \le s$ using \eqref{ineq:embedding.parabolic.space} due to Minkowski’s integral inequality.

If $s< m$, then
\EQS{\label{0213a}
\norm{e^{t\De} f}_{L_T^sE^p_m} &\lec 
\norm{e^{t\De} f}_{L_T^sE^p_s}  = \norm{e^{t\De} f}_{E^{s,p}_{T,s}} \\
&\lec (1+T^{\be'})\norm{f}_{E^r_q} 
}
by $E^{s,p}_{T,m}$ estimate in \eqref{lemma:linear.sotkes.amalgam.a} (with $m$ replaced by $s$)
for any $\be' \in [0,\infty)$ and $\be' > \al'$ with
\[
\al' = \frac d{2s}-\frac d{2q}+ \frac 1s. 
\]
The condition to use \eqref{lemma:linear.sotkes.amalgam.a} is $1\le q \le s \le \infty$. Hence we need to further assume $q \le s$, and we have \eqref{0213a} if $q\le s<m$.
We can take $\be'=\al'$ if $ \al' \ge 0$ and $1<q<s$.

Alternatively, if $p\le m$ and $q\le r$, (then $\al=\frac d{2m}-\frac d{2q}+ \frac 1s\le \frac d{2p}-\frac d{2r}+ \frac 1s=0$), we have
\EQN{
\norm{e^{t\De} f}_{L_T^sE^p_m} = \norm{ \norm{e^{t\De} f}_{E^p_m} }_{L_T^s}
&\lec \norm{ \norm{e^{t\De }f}_{L^p} }_{L_T^s}\ \text{(since $p\le m$)}\\
&\lec \norm{f}_{L^r}\ \text{(by Giga's estimate \eqref{Giga-est})}\\
&\lec \norm{f}_{E^r_q}\ \text{(since $q\le r$)}.
}
In this case we can take $\be'=0$.
This and \eqref{0213a} prove the $L^s_TE^p_m$ estimate in \eqref{lemma:linear.sotkes.amalgam.b}.

Part (b) is a consequence of part (a), when the constants do not depend on $T$ and we can send $T\to \infty$. The constant in \eqref{lemma:linear.sotkes.amalgam.a} is independent of $T$ if $\al\le 0$ (and $\al=0$ requires $1<q<m<\infty$). Hence we have \eqref{lemma:linear.sotkes.amalgam.c}. For \eqref{lemma:linear.sotkes.amalgam.d}, 
 when $m \le s$, we have the estimate if the parameters are in the set $A$. When $q \le s<m$,
we need one of the following 2 cases:
\[
1.\quad q\le r<p\le m\qquad \text{ or }\qquad
2.\quad q\le s<m\ \text{ and }\ \al'\le0\  \text{($\al'=0$ requires $1<q<s$)}.
\]
Note that $\al'\le0$ if and only if $q\le \frac{ds}{d+2}$. This implies that $q=\frac{ds}{d+2}<s$ when $\al'=0$. So we only need to require $q>1$ for $\al'=0$. {Case 2 corresponds to the set $E_2$.}
There are two subcases for Case 1: 
\[
1.1.\quad q\le r, \quad s< p\le m,\qquad\qquad
1.2.\quad q\le r, \quad p\le s<m.
\]
Since $p\le s$ if and only if $r\le\frac{ds}{d+2}$, Subcase 1.2 is included in Case 2. (The endpoint case $1=q=r=\frac{ds}{d+2}$ does not occur since $r>1$ by assumption.)  {Subcase 1.1 corresponds to the set $E_1$.}
This completes the proof of the lemma.
\end{proof}

\medskip
We digress to discuss the condition $q \le s$ in the factor $\mathbbm{1}_{\{q \le s\}}$ in \eqref{lemma:linear.sotkes.amalgam.b}. In particular, it prevents spacetime integral estimates in $L^sL^p_\uloc$---namely for the Maekawa-Terasawa solutions---of the form
\EQ{\label{th2.1-eq2}
\norm{u}_{L^s_T L^p_\uloc}  \lec \norm{ a}_{L^3_\uloc}, \quad \frac 2s+\frac 3p \le 1,
}
when $d=r=3$. This means the spaces $E^{s,p}_{T,q}$ are better suited for the analysis of the Navier-Stokes problem. A similar situation is evident for  Besov spaces \cite[Section 2.6.3]{BCD}   where the space-time integral is applied to individual  Littlewood-Paley blocks and then summed.
If we estimate $\norm{u}_{L^s_T L^p_\uloc}$ naively using the linear estimate of Maekawa-Terasawa \cite[Corollary 3.1]{MaTe}, i.e., \eqref{lemma:linear.stokes.amalgam}  with $q=\td q = \infty$, 
then we obtain for $d=r=3$
\[
\norm{\norm{ e^{t \De}a}_{L^p_\uloc} }_{L^s_T}
 \lec  \norm{ (1+t^{-\frac12 + \frac 3{2p}}) \norm{ a}_{L^3_\uloc} }_{L^s_T},
\]
which diverges if $2/s+3/p\le1$. It converges if we decrease $s$ so that $\frac 2s+\frac 3p > 1$, but the integrability at time 0 is then not as good as $E^{s,p}_{T,q}$ without decreasing $s$.

 In fact, our next example shows that \eqref{th2.1-eq2} is wrong. 
Thus, for the existence theorem, we cannot replace the space $ E^{5,5}_{T,\I}$ by the strict subspace $L^5_T L^5_\uloc$.

\medskip
\begin{example}\label{example2} 

Let the exponents $p,q,r,s,m$ be as in Lemma \ref{lemma:linear.heat.amalgam.spacetime-new}.
The following example shows that $e^{t\De}a$ is not necessarily in $L^s_TE^p_m$ for $a \in E^r_q$ if $ s < q$.
In particular, the factor $\mathbbm{1}_{\{q\le s\} } $ in 
\eqref{lemma:linear.sotkes.amalgam.b} is necessary.

First note that for $\de>0$, $x_0 \in \R^d$, $x \in B_\de(x_0)$ and $t \in (\de^2,2\de^2]$,
\EQ{\label{0318a}
\int _{B_\de(x_0)} \Ga(x-y,t)dy = \int _{B_\de(x_0)} Ct^{-d/2}  e^{-|x-y|^2/4t}dy 
\ge C,
}
where $C$ is independent of $\de$.
For $\beta=\frac dr$ and $\{c_k\}_{k \in \NN} \in \ell^q$, $c_k \ge 0$,
define
\EQN{
a(x) &=  \sum_{k\in \NN}  c_k a_k(x) ,
\\
a_k(x)& = 
2^{\frac{k\beta}2} \cdot\chi_{B_k}(x), \quad B_k =B(x_k , 2^{-k/2}). %
}
We may take $x_k =  2ke_1$. Note that $\norm{a_k}_{L^r} $ is constant in $k$, and hence $\norm{a}_{E^r_q} \approx \norm{c_k}_{\ell^q}$.

Let $u(x,t) = \int_{\R^d}  \Ga(x-y,t)a(y)dy $, i.e., $u(\cdot,t) = e^{t \De}a$.
By \eqref{0318a},
\[
u(x,t) \ge c_k 2^{\frac{k\beta}2}, \quad\text{if }x\in B_k\ \text{and}\ t\in (2^{-k},2^{-k+1}].
\]
Hence for $t\in (2^{-k},2^{-k+1}]$ and $1\le m \le \infty$,%
\EQ{\label{eq2.27}
\norm{u(t)}_{E^p_m}\gec\norm{u(t)}_{L^p_\uloc}\gec
\norm{u(t)}_{L^p(B_1(x_k))}\gec c_k (2^{\frac{k}2})^{\beta-\frac {d}p}\gec c_k t ^{-\frac12(\beta-\frac {d}p)} = c_k t ^{-\frac1s},
}
and
\EQS{\label{eq2.28}
\int_0^1 \norm{u(t)}_{L^{p}_\uloc}^sdt &=\sum_{k=1}^\infty  \int_{2^{-k}}^{2^{-k+1} }  \norm{u(t)}_{L^{p}_\uloc}^s dt 
\\
& \gec \sum_{k=1}^\infty \int_{2^{-k}}^{2^{-k+1}}  \bke{c_kt ^{-\frac1s}}^{s} dt 
 = \ln 2 \sum_{k=1}^\infty c_k ^s.
}

Thus, if $s<q \le \infty$, we can choose $\{c_k\}_{k \in \NN} \in \ell^q$ so that $\{c_k\}_{k \in \NN} \not \in \ell^s$. Then $u \not \in L^s_T E^p_\oo$ and hence $u \not \in L^s_T E^p_m$ for $T=1$ although $a \in E^{r}_q$.

Note that, when $q=\infty$, we can take $c_k \to 0$ very slowly, so that $a$ is in $E^r\subsetneqq E^r_\oo$, (i.e., $a$ is in the closure of $C^1_c$ in $L^r_\uloc$),  but not in $E^r_{q'}$ for any $q' <\infty$. For example, $c_k = (\ln(2+ k))^{-1}$.

The above construction has an interesting application: 
In the special case $q=r$ and $s<r$, we can take $\{c_k\}_{k \in \NN} \in \ell^r \setminus \ell^s$.
Then $a \in E^r_r = L^r$ and $u \not \in L^s_T E^p_\infty$. 
Hence  $u \not \in L^s_T L^p$, and Giga's estimate 
\eqref{Giga-est} is invalid when $s<r$.
In particular, for our application to \eqref{NS} when $q=r=d=3$, 
\eqref{Giga-est} is invalid for $s<3$, i.e., $p>9$.

The above example also serves as a counter example for \eqref{vanishing-t=0} when $\td q = q=\oo$ and $a\in E^p_\oo$ is not in $E^p$. Specifically, let $r<p$, $\beta = \frac dr$ and $c_k=1$. Then $a \in L^r_\uloc$.  For $t\in (2^{-k},2^{-k+1}]$, we have by \eqref{eq2.27} that $\norm{u(t)}_{L^p_\uloc} \gec t^{-\frac{d}2\bke{\frac1{r} - \frac1p}}$.
\hfill $\square$ 
\end{example}

\bigskip 
We next obtain space-time integral estimates for the Duhamel term for solutions to the Stokes equations. For simplicity, we limit ourselves to $d=3$. {On the right side we use $\|F\|_{E^{\td s,\td p}_{T,\td m}}$ in Lemma \ref{lemma:duhamel.amalgam}, and $\norm{F}_{L^{\td s}_T E^{\td p}_{\td m}}$ in Lemma \ref{lem-bilinear-LsEpm}.}

\begin{lemma}\label{lemma:duhamel.amalgam}
For $F: \R^3 \times (0,T)\to \R^{3\times 3}$ let 
\EQ{\label{LF.def}
L(F)_i(x,t)= \int_0^t \int_{\R^3} \pd_l S_{ij}(x-y,t-\tau) F_{lj}(y,\tau)\,dy\,d\tau.
}
Let $1\le\tilde{p}\le p\le\infty$, $1\le \td m \le m \le \infty$ and {$1\le\tilde{s}\le s\le\infty$}. 
Further assume
\EQ{\label{sigmage0}
\si := \frac12 - \frac32\bke{\frac1{\td p} - \frac1p} - \bke{\frac1{\td s}- \frac1s}  \ge0
}
{with $1<\tilde{s}< s<\infty$ in the case of equality $\si=0$.}

(a)
For $0<T<\infty$,
\EQ{\label{eq-est-LF-youngs}
\mathbbm{1}_{m\leq s}   \| L(F)\|_{L^s_T E^p_m}  \lec \| L(F)\|_{E^{s,p}_{T,m}} 
\lec (T^{\si} + T^{\beta}) \|F\|_{E^{\td s,\td p}_{T,\td m}},
}
where  $\beta\in [0,1-  \frac 1{\td s} + \frac 1s]$ and $\be>\al$, with
\[
\al =  \frac12- \frac32\bke{\frac1{\td m} - \frac1m} - \bke{\frac1{\td s}- \frac1s} .
\]
We can take $\beta=\al$ if $\al\ge0$ and $1<\td m< m<\infty$.

(b)
Suppose $\si=0$ and $\al\le 0$. Assume further $1<\td m< m<\infty$ if $\al=0$.
Then
\EQ{\label{eq-est-LF-infty}
\mathbbm{1}_{m\leq s}   \| L(F)\|_{L^s_{T=\infty} E^p_m}  \lec   \| L(F)\|_{E^{s,p}_{T=\infty,m}} \lesssim \|F\|_{E^{\td s,\td p}_{T=\infty,\td m}}.
}

\end{lemma}

For our application to the Navier-Stokes equations \eqref{NS}, we will take $F=u \otimes u$, 
\[
\td p = \frac p2,  \quad \td s = \frac s2, \quad \td m = \max(1,\frac m2).
\]

\begin{proof}
Note that the first inequalities in \eqref{eq-est-LF-youngs} and \eqref{eq-est-LF-infty}
follow from \eqref{ineq:embedding.parabolic.space}.
For every $k\in\ZZ^3$ we may decompose for $x\in B_1(k)$,
\[
L(F)_i(x,t)= \int_0^t \int_{\R^3} \pd_l S_{ij}(x-y,t-\tau) [F_{lj}     \chi_{B_{4}(k)} + F_{lj} (1-\chi_{B_{4}(k)})](y,\tau)\,dy\,d\tau   =: F_1^k +F_2^k.
\]
By Oseen tensor estimates \eqref{Oseen.est} and Young's convolution inequality,
 for $1\le{\tilde{p}\le p}\le\infty$,%
\EQS{\label{est-u1.a}
\norm{F_1^k(\cdot,t)}_{L^p(B_1(k))} 
&\le \int_0^t \norm{\pd_l S_{ij}(\cdot,t-\tau) * ( F(\cdot,\tau)\chi_{B_4(k)})}_{L^p(\R^3)} \, d\tau\\
&\le \int_0^t \norm{\pd_l S_{ij}(\cdot,t-\tau) }_{L^r(\R^3)}\norm{F(\cdot,\tau)\chi_{B_4(k)}}_{L^{\tilde{p}}(\R^3)}\, d\tau\\
&\lec \int_0^t (t-\tau)^{-\frac32\bke{\frac1{\tilde{p}} - \frac1p} - \frac12} \norm{F(\cdot,\tau)}_{L^{\tilde{p}}(B_4(k))}\, d\tau,
}
where $\frac 1r=\frac 1p+1-\frac 1{\tilde p}$, and $r\ge 1$ due to $\td p \le p$. Note that the exponent $\frac32\bke{\frac1{\tilde{p}} - \frac1p} + \frac12 {= -\si + 1 - \bke{\frac1{\td s} - \frac1s} }\in (0,1)$ due to {$\si \ge 0$ with $\td s <s$ in the case of equality} in \eqref{sigmage0}.

If $\si=\frac12 - \frac32\bke{\frac1{\td p} - \frac1p} -\bke{\frac1{\td s}- \frac1s}>0$, we use the Young's convolution inequality for the right hand side of \eqref{est-u1.a} for $t\in[0,T]$ to get
\EQ{\label{est-u1-youngs}
\norm{F_1^k}_{L^s(0,T;L^p(B_1(k)))} 
\le C T^{\si} \norm{F}_{L^{\tilde{s}}(0,T; L^{\tilde{p}}(B_4(k)))}.
}
If $\si=0$ {and $1<\td s < s < \infty$}, we use the Hardy--Littlewood--Sobolev inequality for the right hand side of \eqref{est-u1.a} to get the same estimate \eqref{est-u1-youngs} with $\si=0$.
The constant $C$ in \eqref{est-u1-youngs} is independent of $T$.  
It follows that, for $m\geq \td m$ and $0<T<\infty$,
\EQ{\label{eq-est-F1k}
\norm{ \norm{F_1^k}_{L^s(0,T;L^p(B_1(k)))}  }_{\ell^m(k\in\ZZ^3)} 
\le C T^{\si} \bigg\|   	\|F\|_{L^{\td s}_TL^{\td p} ( B_4(k))}	\bigg\|_{\ell^{m}(k\in \ZZ^3)} \le C T^{\si} \|F\|_{E^{\td s,\td p}_{T,\td m}}.
}
If $\si=0$, the above estimate is uniform for all $T$ and we can take $T=\infty$.

The estimate for $F_2^k$ is more intricate but follows themes already explored in this paper and \cite{BT4}. Note that for $x\in B_1(k)$  and $t\in (0,T)$,
\EQS{\label{eq-F2k}
|F_2^k (x ,t)|&\lesssim \int_0^t \int_{|y|>2} \frac 1 {(|y|+\sqrt \tau)^4} |F(x-y,t-\tau)|\,dy\,d\tau
\\&\lesssim \sum_{|k'|\geq 1}\int_0^t \frac 1 {(|k'|+\sqrt \tau)^4}\int_{B_1(k')} |F(x-y,t-\tau)|\,dy\,d\tau
\\&\lesssim \sum_{|k'|\geq 1} \int_0^t \frac 1 {(|k'|+\sqrt \tau)^4}  \int_{B_2(k-k')} |F(z,t-\tau)|\,dz\,d\tau
\\&\lesssim \sum_{|k'|\geq 1} \int_0^t \frac 1 {(|k'|+\sqrt \tau)^4}\, \| F(t-\tau)\|_{L^{\td p} (B_2(k-\td k))}\, d\tau.
}

Using the above $L^\I$ estimate to bound $\norm{F_2^k(t)}_{L^p(B_1(k))}$ and then  
Young's convolution inequality for $t \in (0,T)$, we get
\EQ{\label{0123c}
a_k:=\norm{F_2^k}_{L^s(0,T;L^p(B_1(k)))}  \lec \sum_{|k'|\ge 1}b_{k'}   f_{k-k'} ,
}
where 
\[
b_{k'} =  \norm{\frac1{(|k'|+\sqrt \tau)^4}}_{L^{s^*}(0,T)}, \quad
f_k = \norm{F}_{L^{\td s}(0,T; L^{\td p}(B_2(k)))}, \quad
1+\frac1s = \frac1{s^*}+\frac1{\td s}.
\]
Note $1/s^* \in [1/2,{1]}$ since 
$\si \ge 0$ %
We may set $b_0=0$. For $k \not =0$, we have
\EQN{
b_k &= \bke{ \int_0^T \frac{d\tau}{(|k|+\sqrt\tau)^{4s^*}}}^{\frac1{s^*}}
=|k|^{-4+\frac2{s^*}}  \bke{ \int_0^{T/|k|^2} \frac{dt}{(1+\sqrt t)^{4s^*}}}^{\frac1{s^*}}.
}
The integral $\int_0^x \frac{dt}{(1+\sqrt t)^{4s^*}}$ is bounded by $1$ for $x\ge1$ and by $x$ for $x \in (0,1)$. Hence it is bounded by $x/(1+x)$ for $0<x<\infty$, and we get
\EQ{
\label{3.15.22.zb}
b_k \lec  |k|^{-4+\frac2{s^*}} \bke{\frac T{T+|k|^2}}^{\frac1{s^*}}
=  |k|^{-2- \bke{ \frac 2{\td s} - \frac 2s}} \bke{\frac T{T+|k|^2}}^{1- \bke{ \frac 1{\td s} - \frac 1s}}.
}

Let $\gamma\in [1,\infty]$ be defined by%
\[
1+ \frac 1m = \frac 1\gamma + \frac 1{\td m}.
\]
Note $b_k \le C T^{\frac1{s^*}} |k|^{-4}$ hence $b_k \in \ell^\gamma$. Applying Young's convolution inequality to \eqref{0123c}, 
we get $\norm{a}_{\ell^m} \lec \norm{ b}_{\ell^\gamma} \norm{ f}_{\ell^{\td m}}$ and hence
\EQ{\label{0123a}
\norm{a}_{\ell^m} \lec  T^{\be}\norm{ f}_{\ell^{\td m}}
}
with $\be=\frac1{s^*}$. By going back to \eqref{3.15.22.zb}, it is clear that, for certain parameters, the choice of $\be$ in \eqref{0123a} can be improved. We pursue this presently.

\medskip

{\bf Case 1}:  $\frac{2}{\td s} - \frac{2}s + \frac3{\td m} - \frac3m >1$. In this case we have
\EQ{\label{0123b}
2+ \bke{ \frac 2{\td s} - \frac 2s} >\frac3\gamma= 3 -  \bke{ \frac 3{\td m} - \frac 3m},
}
and, noting the second factor in \eqref{3.15.22.zb} is bounded by $1$,%
\[
\norm{ b}_{\ell^\gamma} \lec \norm{|k|^{-2- \bke{ \frac 2{\td s} - \frac 2s}} }_{\ell^\gamma} \lec 1.
\]
We get \eqref{0123a} with $\beta=0$.

\medskip  
{\bf Case 2}:  $\frac{2}{\td s} - \frac{2}s + \frac3{\td m} - \frac3m =1$. In this case we have equality in \eqref{0123b}, and hence
\[
 b_k \lec |k|^{-2- \bke{ \frac 2{\td s} - \frac 2s}}= |k|^{-3/\gamma} .
\]
Instead of Young's inequality, we use the discrete version of the Hardy--Littlewood--Sobolev inequality stated in \cite[Proposition (a)]{Stein-Wainger-discreteHLS} to obtain the same estimate \eqref{0123a} with $\be=0$.
For \cite[Proposition (a)]{Stein-Wainger-discreteHLS} we need to avoid end points, hence we assume $1<\td m<m <\infty$. That $3/\gamma<3$ is given by $\td m<m$. %

\medskip

{\bf Case 3}:  $\frac{2}{\td s} - \frac{2}s + \frac3{\td m} - \frac3m <1$. In this case we have inequality ``$<$'' in \eqref{0123b}, and cannot avoid $T$ dependence. If $ 1<\td m<m <\infty$, we choose 
\[
\al =  \frac12- \frac{1}{\td s} + \frac{1}s - \frac3{2\td m} + \frac3{2m} , \quad 0<\al < 1 -\frac{1}{\td s} + \frac{1}s. 
\]
Here, $1-\frac1{\td s} + \frac1s > 0$. In fact, the case $(\td s,s)=(1,\infty)$ is neither in Case 2 nor in Case 3 because $\frac2{\td s} - \frac2s + \frac3{\td m} - \frac3m \ge 2 $ when $\td s = 1$, $s=\infty$.
We have
\[
b_k \lec  |k|^{-2- \bke{ \frac 2{\td s} - \frac 2s}} \bke{\frac T{T+|k|^2}}^{\al} \lec |k|^{-3/\gamma} T^\al.
\]
By the discrete Hardy--Littlewood--Sobolev inequality in \cite{Stein-Wainger-discreteHLS} again,
we get \eqref{0123a} with $\be=\al$.

For all cases including the end points $1=\td m$, $\td m=m$ or $m=\infty$, we can choose $\be \in (\al, 1- \frac 1{\td s}+ \frac 1s]$, $\be\ge0$,
\EQ{\label{0124a}
b_k \lec  |k|^{-2- \bke{ \frac 2{\td s} - \frac 2s}} \bke{\frac T{T+|k|^2}}^{\be} \lec |k|^{-3/\gamma-2(\be-\al)} T^\be,
}
and use Young's inequality.

Combining \eqref{eq-est-F1k} and \eqref{0123a} for all 3 cases of $\al$, we get part (a). Part (b) is a consequence of part (a), when the constants do not depend on $T$ and we can send $T\to \infty$.
This completes the proof of the lemma.
\end{proof}

We finally give bilinear estimates in $L^{s}_T E^p_{m}$. On the right side we use $\norm{F}_{L^{\td s}_T E^{\td p}_{\td m}}$, not $\|F\|_{E^{\td s,\td p}_{T,\td m}}$ in Lemma \ref{lemma:duhamel.amalgam}.

\begin{lemma}\label{lem-bilinear-LsEpm}
Let $0 < T \le \infty$.
For $F: \R^3 \times (0,T)\to \R^{3\times 3}$ let $L(F)$ be defined by \eqref{LF.def}.
For $1\le\td p \le p \le\infty$, $1\le\td m\le m\le\infty$ and $1<\td s<s<\infty$,
if $\frac12-\frac32\bke{\frac1{\td p} - \frac1p} - \bke{\frac1{\td s} - \frac1s} = 0$ and $\frac12-\frac32\bke{\frac1{\td m} - \frac1m} - \bke{\frac1{\td s} - \frac1s} = 0$, then for a $T$-independent constant $C$,
\[
\norm{LF}_{L^{s}_T E^p_{m}} \le C \norm{F}_{L^{\td s}_T E^{\td p}_{\td m}}.
\]
\end{lemma}
\begin{proof}
Denote $\mu= \frac12+\frac32\bke{\frac1{\td p} - \frac1p}= \frac12+\frac32\bke{\frac1{\td m} - \frac1m} =1-\frac1{\td s} + \frac1s$, $0<\mu<1$ since $1<\td s<s<\infty$. %
By Lemma \ref{Wa-estimate},
\EQN{
\norm{LF}_{L^{s}_T E^p_{m}} 
&= \norm{\int_0^t \norm{\int_{\R^3} \pd_lS_{ij}(x-y,t-\tau) F_{lj}(y,\tau)\, dy}_{E^p_m} d\tau}_{L^s_T}\\
&\lec \norm{\int_0^t \frac1{(t-\tau)^{\mu}} \norm{F(\cdot,\tau)}_{E^{\td p}_{\td m}} d\tau}_{L^s_T},
}
which is bounded by $\norm{F}_{L^{\td s}_T E^{\td p}_{\td m}}$ by the Hardy-Littlewood-Sobolev inequality.
\end{proof}

\subsection{Energy estimates}

In this subsection, we prove the energy estimates in Wiener amalgam spaces, which can be viewed as the $E^2_q$-version of \cite[Lemma 2.4]{KwTs}.
The energy estimates are used in Section \ref{sec:global-weak} to construct global-in-time weak solutions in $E^2_q$ for $1\le q<2$. Recall the $\ell^q$ local energy space ${\bf LE}_q(0,T)$ is defined 
by the norm \eqref{ETq.def}.

\begin{lemma}\label{KwTs-lem2.4}%

For any $T>0$ and $q\ge1$, if $f\in E^2_q$ and $F\in E^{2,2}_{T,q}$, {defined for $x\in \R^3$,}  then we have
for any $\be>0$, 
\EQ{\label{eq-KwTs-lem2.4-1}
\norm{e^{t\De} f}_{{\bf LE}_q(0,T)} \lec (1+T^\be) \norm{f}_{E^2_q},
}
and
\EQ{\label{eq-KwTs-lem2.4-2}
\norm{\int_0^t e^{(t-\tau)\De} \mathbb P\nb\cdot F(\tau)\, d\tau}_{{\bf LE}_q(0,T)} \lec (1+T^\beta) \norm{F}_{E^{2,2}_{T,q}}.
}
\end{lemma}
\begin{proof}
We first prove \eqref{eq-KwTs-lem2.4-1}.
By \eqref{lemma:linear.sotkes.amalgam.a} in Lemma \ref{lemma:linear.heat.amalgam.spacetime-new}, $\norm{f}_{E^{\infty,2}_{T,q}} \lec (1+T^\be) \norm{f}_{E^2_q}$ for any $\be>0$.
It remains to estimate $\norm{\nb e^{t\De} f}_{E^{2,2}_{T,q}}$.
The proof is almost identical to that of Lemma \ref{lemma:linear.heat.amalgam.spacetime-new} except that we use the energy estimate instead of Giga's estimate for the localized part.
For every $k\in\ZZ^3$ and $x\in B_1(k)$, we decompose
\[
\nb e^{t\De} f (x,t)= \nb e^{t\De} (f\chi_{B_4(k)}) + \nb e^{t\De} (f(1-\chi_{B_4(k)})) =: f_1^k(x,t) + f_2^k(x,t).
\]
Then
\EQN{
&\norm{\nb e^{t\De} f}_{E^{2,2}_{T,q}}
= \norm{ \norm{ \norm{ \nb e^{t\De} f }_{L^2(B_1(k))} }_{L^2(0,T)} }_{\ell^{q}(k\in\ZZ^3)}\\
&\le \norm{ \norm{ \norm{ f_1^k(\cdot,t) }_{L^2(B_1(k))} }_{L^2(0,T)} }_{\ell^{q}(k\in\ZZ^3)} + \norm{ \norm{ \norm{ f_2^k(\cdot,t) }_{L^2(B_1(k))} }_{L^2(0,T)} }_{\ell^{q}(k\in\ZZ^3)}\\
&=: A_1 + A_2.
}
For $A_1$, we use the usual energy estimate for the heat equation to get
\[
\norm{ \norm{f_1^k(\cdot,t)}_{L^2(B_1(k))} }_{L^2(0,T)} 
\le \norm{\nb e^{t\De}(f\chi_{B_4(k)})}_{L^2(0,\infty; L^{2}(\R^3))}\lec 
\norm{f}_{L^2(B_4(k))}.
\]
Hence
\EQ{\label{eq-KwTs-lem2.4-pf-1}
A_1 \le \norm{\norm{f}_{L^2(B_4(k))} }_{\ell^{q}(k\in\ZZ^3)} \approx \|f\|_{E^2_{q}}.
}

\medskip

For $A_2$, as
\[
|f^k_2(x,t)| \lec \sum_{k':\ |k'-k|{\ge} 4} \int_{B_1(k)} t^{-2} e^{- \frac {|x-y|^2}{8t}} |f(y)|\,dy,
\]
we have
\[
\norm{f^k_2(\cdot,t)}_{L^2(B_1(k))} \lec t^{-2} \sum_{|k'|\ge1} e^{-|k'|^2/(4t)} \norm{f}_{L^2(B_1(k'-k))}.
\]
Taking $L^2_T$ on both sides and applying Minkowski's integral inequality, we get
\EQN{
\norm{f^k_2}_{L^2_TL^2(B_1(k))} &\lec \bigg[ \int_0^T \bigg( t^{-2} \sum_{|k'|\ge1} e^{-|k'|^2/(8t)} \norm{f}_{L^2(B_1(k'-k))}\bigg) ^2 dt \bigg]^{1/2}\\
&\le \sum_{|k'|\ge1} \bke{\int_0^T \abs{t^{-2} e^{-|k'|^2/(8t)} \norm{f}_{L^2(B_1(k'-k))} }^2 dt}^{1/2}\\
&= \sum_{|k'|\ge1} \norm{f}_{L^2(B_1(k'-k))} \bke{\int_0^T t^{-4}\, e^{-\frac{|k'|^2}{4t} }\, dt}^{1/2},\quad t=|k'|^2 \tau\\
&= \sum_{|k'|\ge1} \norm{f}_{L^2(B_1(k'-k))} |k'|^{-3} \bke{\int_0^{\frac{T}{|k'|^2}} \tau^{-4}\, e^{-\frac1{4\tau} }\, d\tau}^{1/2}.
}
We have 
$
\int_0^\infty \tau^{-4}\, e^{-\frac1{4\tau} }\, d\tau < \infty
$
and,
 if $S<1$,
\[
\int_0^S \tau^{-4}\, e^{-\frac1{4\tau} }\, d\tau
\lec \int_0^S  e^{-\frac1{8\tau} }\, d\tau \lec e^{-\frac1{8S} }.
\]
This estimate is also true if $S>1$.
We conclude%
\[
\norm{f^k_2}_{L^2_TL^2(B_1(k))} \lec  \sum_{|k'|\ge1} \norm{f}_{L^2(B_1(k'-k))}{ |k'|^{-3}\,e^{-\frac {|k'|^2}{8T}}.}
\]

For fixed $0<T< \infty$, we can bound for any $\be \in [0,\oo)$ 
\[
{|k'|^{-3}\,e^{-\frac {|k'|^2}{8T}}} \lec |k'|^{-3} { \bigg(  \frac {|k'|^2}{8T}\bigg)^{-\be} =} |k'|^{-3-2\be} T^{\be}.
\]
By Young's convolution inequality,
\EQ{\label{eq-KwTs-lem2.4-pf-2}
A_2 \lec \norm{\sum_{|k'|\ge1} \norm{f}_{L^2(B_1(k'-k))}  |k'|^{-3-2\be} T^{\be} }_{\ell^{q}(k\in\ZZ^3)}
\lec \norm{f}_{E^2_q}\norm{|k|^{-3-2\be} T^{\be} }_{\ell^{1}(k\not=0\in\ZZ^3)},
}
where the last norm is finite since $\be>0$.
The estimate \eqref{eq-KwTs-lem2.4-1} follows from \eqref{eq-KwTs-lem2.4-pf-1} and \eqref{eq-KwTs-lem2.4-pf-2}. 

Next, we prove \eqref{eq-KwTs-lem2.4-2}. 
By \eqref{eq-est-LF-youngs} in Lemma \ref{lemma:duhamel.amalgam}, $\norm{\int_0^t e^{(t-\tau)\De}\mathbb P\nb\cdot F(\tau)\, d\tau}_{E^{\infty,2}_{T,q}} \lec (1+T^\be) \norm{F}_{E^{2,2}_{T,q}}$ for any $\be>0$.
It remains to estimate $\norm{\nb \int_0^t e^{(t-\tau)\De}\mathbb P\nb\cdot F(\tau)\, d\tau}_{E^{2,2}_{T,q}}$.
The proof is almost identical to that of Lemma \ref{lemma:duhamel.amalgam} except that we use the energy estimate instead of the Oseen's tensor estimate for the localized part.
For every $k\in\ZZ^3$ we may decompose for $x\in B_1(k)$,
\EQN{
&\nb \int_0^t e^{(t-\tau)\De}\mathbb P\nb\cdot F(\tau)\, d\tau \\
&= \nb \int_0^t e^{(t-\tau)\De}\mathbb P\nb\cdot (F(\tau) \chi_{B_{4}(k)}) \, d\tau + \nb \int_0^t e^{(t-\tau)\De}\mathbb P\nb\cdot \bkt{ F(\tau) (1- \chi_{B_{4}(k)}) } \, d\tau
=: F_1^k + F_2^k.
}
For $F_1^k$, by the usual energy estimate for the Stokes system we get  
\EQN{
\norm{ \norm{F_1^k(\cdot,t)}_{L^2(B_1(k))} }_{L^2(0,T)} 
&\le \norm{\nb \int_0^t e^{(t-\tau)\De}\mathbb P\nb\cdot (F(\tau) \chi_{B_{4}(k)}) \, d\tau}_{L^2_T L^{2}(\R^3)}\\
&\lec \norm{F}_{L^2_T L^2(B_4(k))}.
}
Hence
\EQ{\label{eq-KwTs-lem2.4-pf-3}
\norm{ \norm{ \norm{F_1^k(\cdot,t)}_{L^2(B_1(k))} }_{L^2(0,T)} }_{\ell^q(k\in\ZZ^3)}  \le \norm{ \norm{F}_{L^2_T L^2(B_4(k))} }_{\ell^{q}(k\in\ZZ^3)} \approx \|F\|_{E^{2,2}_{T,q}}.
}

For $F_2^k$, note that for $x\in B_1(k)$  and $t\in (0,T)$,
\EQS{
|F_2^k (x ,t)|&\lesssim \int_0^t \int_{|y|>2} \frac 1 {(|y|+\sqrt \tau)^5} |F(x-y,t-\tau)|\,dy\,d\tau
\\&\lesssim \sum_{|k'|\geq 1}\int_0^t \frac 1 {(|k'|+\sqrt \tau)^5}\int_{B_1(k')} |F(x-y,t-\tau)|\,dy\,d\tau
\\&\lesssim \sum_{|k'|\geq 1} \int_0^t \frac 1 {(|k'|+\sqrt \tau)^5}  \int_{B_2(k-k')} |F(z,t-\tau)|\,dz\,d\tau
\\&\lesssim \sum_{|k'|\geq 1} \int_0^t \frac 1 {(|k'|+\sqrt \tau)^5}\, \| F(t-\tau)\|_{L^{2} (B_2(k-\td k))}\, d\tau.
}

Using the above $L^\I$ estimate to bound $\norm{F_2^k(t)}_{L^2(B_1(k))}$ and then  
Young's convolution inequality for $t \in (0,T)$, we get
\EQ{\label{0123c-KwTs-lem2.4}
a_k:=\norm{F_2^k}_{L^2(0,T;L^2(B_1(k)))}  \lec \sum_{|k'|\ge 1}b_{k'}   \phi_{k-k'} ,
}
where 
\[
b_{k'} =  \norm{\frac1{(|k'|+\sqrt \tau)^5}}_{L^1(0,T)}, \quad
\phi_k = \norm{F}_{L^{2}(0,T; L^{2}(B_2(k)))}.
\]
We may set $b_0=0$. For $k \not =0$, we have
\EQN{
b_k &=  \int_0^T \frac{d\tau}{(|k|+\sqrt\tau)^{5}}
=|k|^{{-3}} \int_0^{T/|k|^2} \frac{dt}{(1+\sqrt t)^{5}}.
}
The integral $\int_0^x \frac{dt}{(1+\sqrt t)^5}$ is bounded by $1$ for $x\ge1$ and by $x$ for $x \in (0,1)$. Hence it is bounded by $x/(1+x)$ for $0<x<\infty$, and we get
\Eq{
b_k \lec  |k|^{{-3}} \frac T{T+|k|^2} {\lec  |k|^{{-3}} \bke{\frac T{|k|^2}}^{\min(1,\beta)}}.
}
{Hence $\norm{b_k}_{\ell_1} \lec 1+T^\beta$ for any $\beta>0$.}  
Applying Young's convolution inequality to \eqref{0123c-KwTs-lem2.4}
we get $\norm{a}_{\ell^q} \lec \norm{ b}_{\ell^1} \norm{ \phi}_{\ell^{q}}$ and hence 
\EQ{\label{0123a-KwTs-lem2.4}
\norm{ \norm{ \norm{F_2^k(\cdot,t)}_{L^2(B_1(k))} }_{L^2(0,T)} }_{\ell^q(k\in\ZZ^3)} = \norm{a}_{\ell^q} \lec  {(1+T^\beta)}\norm{ \phi}_{\ell^{q}}.
}
Combining \eqref{eq-KwTs-lem2.4-pf-3} and \eqref{0123a-KwTs-lem2.4}, \eqref{eq-KwTs-lem2.4-2} follows.
\end{proof}

\section{The Navier-Stokes equations in Wiener amalgam spaces}\label{sec.ns}

In this section we consider the nonlinear Navier-Stokes equations in space dimension $d=3$.

Recall the Picard contraction principle which states: If $E$ is a Banach space and $B: E\times E \to E$ is a bounded bilinear transform satisfying \EQ{
\label{contraction}
\| B(e,f)\|_{E}\leq C_B \| e\|_E\|f\|_E,
}
and if $\|e_0\|_{E}\le \de\le (4C_B)^{-1}$, then the equation $e = e_0 - B(e,e)$ has a solution with $\| e\|_{E}\leq 2 \delta$ and this solution is unique in $\overline B(0,2\delta)$.
 
For our application to Navier-Stokes equations, $e_0 = e^{t\Delta}u_0$ and 
the bilinear operator is
\EQ{ \label{Bfg.def}
B(f,g)_i(x,t) = \int_0^t
\int_{\R^3}  \pd_l S_{ij}(x-y,t-\tau) f_l g_j(y,\tau)\, dy \, d\tau,
}
which is just the vector components of $B(f,g)$ as defined in Section \ref{sec1.2}, in terms of the Oseen tensor $S_{ij}$.

\subsection{Mild solutions in subcritical spaces}

We first prove Theorem \ref{thrm.subcritical} for subcritical data $u_0 \in E^r_q$, $3<r\le \infty$.
 
 \begin{proof}[Proof of Theorem \ref{thrm.subcritical}]
 Let 
\[
\norm{f}_{\mathcal E_T} = \sup_{0\le t\le T} \norm{f(t)}_{E^r_q}.
\]
By Lemma \ref{Wa-estimate} we have 
\[
\|e^{t\Delta}u_0\|_{E^r_q}\lec \| u_0\|_{E^r_q}.
\]
So
\EQ{\label{eq-linear_mE_T}
\norm{e^{t\De} u_0}_{\mathcal E_T} \le C_1 \| u_0\|_{E^r_q}.
}
For bilinear estimate, we have by Lemma \ref{Wa-estimate}
\EQN{
\| B(f,g)  (t)\|_{E^r_q} 
&\lec \int_0^t \bigg(\frac 1 {(t-\tau)^{\frac 1 2+\frac 3 {2r}}   }  +\frac 1 {(t-\tau)^{\frac 1 2 }   }\bigg)
 \| (f\otimes g)(\tau)\|_{E^{r/2}_q}\, d\tau\\
&\lec (   t^{1/2 - 3/(2r)} +  t^{1/2}	 ) \sup_{0<\tau<t} \| f(\tau)\|_{E^r_q}
\sup_{0<\tau<t} \|g(\tau)\|_{E^r_\infty}\\
&\lec (   t^{1/2 - 3/(2r)} +  t^{1/2}	 ) \sup_{0<\tau<t} \| f(\tau)\|_{E^r_q}
\sup_{0<\tau<t} \|g(\tau)\|_{E^r_q}\\
&\lec (   t^{1/2 - 3/(2r)} +  t^{1/2}	 ) \| f\|_{\mathcal E_T}\|g\|_{\mathcal E_T},
}
where we've used the inclusion $E^r_q\subset E^r_\infty$ in the second to last inequality.
Therefore, we have 
\EQ{\label{eq-bilinear_mE_T}
\norm{B(f,g)}_{\mathcal E_T} \le C_2 (   T^{1/2 - 3/(2r)} +  T^{1/2}	 ) \| f\|_{\mathcal E_T}\|g\|_{\mathcal E_T}.
}
 We seek a solution of the form 
\[
u = e^{t\Delta}u_0 - B(u,u).
\]
Assume $T$ is chosen small enough that $ \norm{u_0}_{E^r_q} < (8C_1C_2(   T^{1/2 - 3/(2r)} +  T^{1/2}	 ))^{-1}$. Then the Picard contraction principle implies there exists a unique strong mild solution satisfying 
\EQ{\label{eq-picard-uniqueclass}
\norm{u}_{\mathcal E_T} \le 2C_1\norm{u_0}_{E^r_q}.
}

We now prove the various statements in the theorem concerning continuity. Assume $r,q<\I$. Then, 
\EQS{
\| u(t) - u_0\|_{E^r_q} &\leq \| B(u,u)(t) \|_{E^r_q} + \| e^{t\Delta}u_0 - u_0\|_{E^r_q}\\
&\lec(   t^{1/2 - 3/(2r)} +  t^{1/2 }	 ) \sup_{0<\tau<t} \| u(\tau)\|_{E^r_q} \sup_{0<\tau<t} \| u(\tau) \|_{E^r_\infty} +\| e^{t\Delta}u_0 - u_0\|_{E^r_q}\\
&\le(   t^{1/2 - 3/(2r)} +  t^{1/2 }	 ) \sup_{0<\tau<t} \| u(\tau)\|_{E^r_q} \sup_{0<\tau<t} \| u(\tau) \|_{E^r_q} +\| e^{t\Delta}u_0 - u_0\|_{E^r_q}.
}
Since the powers on $t$ are positive in the first term and since the latter term vanishes
by Lemma \ref{lemma:heat.eq.convergence.data}, we have 
\[
\| u(t) - u_0\|_{E^r_q} \to 0 \text{ as }t\to 0^+.
\]
The continuity at $t\in (0,T)$ can be shown as usual, see e.g., \cite[lines 3-8, page 86]{Tsai-book}, including $r=\I$ or $q=\I$. Compare the proof of Theorem \ref{thrm:critical}.

If either $r=\I$ or $q=\I$ we no longer have $\| e^{t\Delta}u_0 - u_0\|_{E^r_q}\to 0$ but we do still have 
\[
(   t^{1/2 - 3/(2r)} +  t^{1/2  }	 ) \sup_{0<\tau<t} \| u(\tau)\|_{E^r_q}\sup_{0<\tau<t}  \| u(\tau) \|_{E^r_q}  \to 0.
\]
Hence, we still have
\[
\| u (t)-e^{t\Delta}u_0 \|_{E^r_q}\to 0\text{ as }t\to 0^+,
\]
as asserted in the theorem.

Next we prove the uniqueness of the mild solution in $L^\infty(0,T;E^r_q) \cap C((0,T);E^r_q)$ without the bound \eqref{eq-picard-uniqueclass}. Let $u_1, u_2\in L^\infty(0,T; E^r_q) \cap C((0,T);E^r_q)$ be two mild solutions with initial data $u_0\in E^r_q$. Then for $0<t<T'\le T$
\EQN{
\norm{(u_1 - u_2)(t)}_{E^r_q} 
&\le \norm{B(u_1 - u_2, u_2)(t)}_{E^r_q} + \norm{B(u_2, u_1 - u_2)(t)}_{E^r_q} \\
&\lec (   t^{1/2 - 3/(2r)} +  t^{1/2})  \bke{\norm{u_1}_{\mathcal E_T}+\norm{u_2}_{\mathcal E_T}}
\sup_{0<t<T'} \| (u_1 - u_2)(t)\|_{E^r_q}
,
}
so that we have
\[
\sup_{0<t<T'} \| (u_1 - u_2)(t)\|_{E^r_q}
\lec (   T'^{\frac12 - \frac3{2r}} +  T'^{\frac12}) \bke{\norm{u_1}_{\mathcal E_T}+\norm{u_2}_{\mathcal E_T}}
\sup_{0<t<T'} 
\| (u_1 - u_2)(t)\|_{E^r_q}.
\]
Thus, for sufficiently small $T'>0$ it follows that $u_1 = u_2$ on $(0,T')$. Repeating this argument, we see that $u_1 = u_2$ on $(0,T)$.

To obtain the spacetime integral bound we assume $3<r\le s\le \infty$, $r\le p<\infty$, 
$\frac2s + \frac3p = \frac3r$ and $1\le q=m \le\infty$, and perform the Picard iteration in the Banach space
\[
X_T = \mathcal E_T \cap E^{s,p}_{T,q}.
\]
We may assume $m=q$ since $\norm{u}_{E^{s,p}_{T,m}}\le \norm{u}_{E^{s,p}_{T,q}}$ for $m\ge q$. %
From 
$\frac3p = \frac3r-\frac2s  \ge \frac3r-\frac2r$ we get $p \le 3r<\infty$.
For the linear term, by \eqref{eq-linear_mE_T} and Lemma \ref{lemma:linear.heat.amalgam.spacetime-new} (which needs $r<\infty$ and $r \le s$), we have for a fixed $\epsilon>0$ that
\EQN{
\norm{e^{t\De} u_0}_{X_T} = \norm{e^{t\De} u_0}_{\mathcal E_T} + \norm{e^{t\De} u_0}_{E^{s,p}_{T,q}} \le C_3(1+T^{1/s+\epsilon})\norm{u_0}_{E^r_q}.
}
For the bilinear term, by \eqref{eq-bilinear_mE_T} and Lemma \ref{lemma:duhamel.amalgam}
with $\td p = p/2$ and $\td s = s/2$ so that $\si = \frac12 - \frac3{2r}>0$ due to $r>3$, (allowing $s=\infty$),
\EQN{
\norm{B(f,g)}_{X_T} &= \norm{B(f,g)}_{\mathcal E_T} + \norm{B(f,g)}_{E^{s,p}_{T,q}}\\
&\le C_4 \bkt{ \bke{T^{\frac12-\frac3{2r}} + T^{\frac12}} \norm{f}_{\mathcal E_T} \norm{g}_{\mathcal E_T} + \bke{T^{\frac12 - \frac3{2 r}} + T^{1-\frac1s }} \norm{f\otimes g}_{E^{\frac{s}2,\frac{p}2}_{T,q}} }.
}
Since%
\[
\norm{f\otimes g}_{E^{\frac{s}2,\frac{p}2}_{T,q}} 
\le \norm{f}_{E^{s,p}_{T,q}} \norm{g}_{E^{s,p}_{T,\infty}}
\le \norm{f}_{E^{s,p}_{T,q}} \norm{g}_{E^{s,p}_{T,q}},
\]
we derive%
\[
\norm{B(f,g)}_{X_T} \le 2C_4 \bke{T^{\frac12-\frac3{2r}} + T^{1 - \frac1s}} \norm{f}_{X_T} \norm{g}_{X_T}.
\]
If {$T>0$} is chosen small enough so that $ \norm{u_0}_{E^r_q} < [8C_3 C_4 (1 + T^{\frac1s+\epsilon})(T^{\frac12-\frac3{2r}} + T^{1 - \frac1s})]^{-1}$, then the Picard contraction principle implies there exists a unique strong mild solution $\td u\in X_T$. Since $X_T\subset L^\infty_T E^r_q$, we have $\td u = u$ on $(0,T)$ by the uniqueness in $L^\infty(0,T;E^r_q)$. Thus, $u$ satisfies the spacetime integral bound.
\end{proof}

\subsection{Mild solutions in critical spaces with small data}

We next prove Theorem \ref{thrm:critical} for small critical data $u_0 \in E^3_q$, $1\le q\le \infty$. For simplicity of presentation, we first consider finite time exponent $s<\infty$ in the spacetime integral estimates \eqref{eq1.7} in $E^{s,p}_{T,m}$, and postpone the case $s=\infty$ (hence $p=3$) after we have proved Lemmas \ref{lem-new-linear} and \ref{lem-new-bilinear}.

\begin{proof}[Proof of Theorem \ref{thrm:critical} when $s<\infty$] 
The overall logic of the proof follows that in \cite{MaTe} which, in turn, draws on \cite{Kato}. Let
\[
\norm{f}_{\tilde{\mathcal{E}}_T} 
:= \sup_{0<t<T} \norm{f(\cdot,t)}_{E^3_q}    + \sup_{0<t<T} t^{\frac12} \norm{f(\cdot,t)}_{E^\I_q}
\]
and
\[ 
\norm{f}_{\tilde{\mathcal{F}}_T} 
:= \sup_{0<t<T} t^{\frac14} \norm{f(\cdot,t)}_{E^6_q}.
\]
It is obvious that $\tilde{\mathcal{E}}_T \subset \tilde{\mathcal{F}}_T$. Indeed, 
\EQN{
t^{\frac14} \norm{f}_{ E^6_q}
&= t^{\frac14} \norm{ \norm{f(\cdot,t)}_{L^6(B_1(k))} }_{\ell^q(k\in\ZZ^3)} \\
&\le t^{\frac14}  \norm{ \norm{f(\cdot,t)}_{L^\infty(B_1(k))}^{1/2} \norm{f(\cdot,t)}_{L^3(B_1(k))}^{1/2} }_{\ell^q(k\in\ZZ^3)}
\\
&\le t^{\frac14} \norm{ \norm{f(\cdot,t)}_{L^\infty(B_1(k))}^{1/2} }_{\ell^{2q}(k\in\ZZ^3)} \norm{ \norm{f(\cdot,t)}_{L^3(B_1(k))}^{1/2} }_{\ell^{2q}(k\in\ZZ^3)} 
\\
&= (t^{\frac12}  \norm{f(\cdot,t)}_{E^\infty_q})^{1/2} \norm{f(\cdot,t)}_{E^3_q}^{1/2}.
}
Now, we define
\EQ{
\mathcal{E}_T := \tilde{\mathcal{E}}_T \cap E^{s,p}_{T,m} \qquad \text{ and } \qquad 
\mathcal{F}_T := \tilde{\mathcal{F}}_T \cap E^{s,p}_{T,m} ,
}
with norms
\[
\norm{f}_{\mathcal{E}_T} := \norm{f}_{\tilde{\mathcal{E}}_T} + \norm{f}_{E^{s,p}_{T,m}}\qquad \text{ and } \qquad 
\norm{f}_{\mathcal{F}_T} := \norm{f}_{\tilde{\mathcal{F}}_T} + \norm{f}_{E^{s,p}_{T,m}},
\]
respectively. We also have $\mathcal{E}_T\subset\mathcal{F}_T$ by intersecting $E^{s,p}_{T,m}$ with both sides of the inclusion.

By \eqref{lemma:linear.sotkes.amalgam.a} in Lemma \ref{lemma:linear.heat.amalgam.spacetime-new} we have for a fixed $\epsilon>0$ chosen so that $\epsilon >\frac 3{2m}-\frac 3 {2q}$ from the definition of $\beta$ in Lemma \ref{lemma:linear.heat.amalgam.spacetime-new},
\EQ{
\| e^{t\Delta}u_0\|_{E^{s,p}_{T,m}}
\lec (1+T^{\frac1 s +\epsilon}) \|u_0\|_{E^3_q}.
}
Also,  by Lemma \ref{Wa-estimate}, 
\EQ{
t^{\frac {1}{4}}  \| e^{t\Delta}u_0\|_{E^6_q} 
\lec  (1+ T^{ \frac 1 4}) \|u_0\|_{E^3_q}.
}
Hence
\EQ{\label{3.21a}
\| e^{t\Delta}u_0\|_{\mathcal F_T} \le C_1( 1 + T^{\frac1 s +\epsilon} +T^{\frac {1}4}   ) \| u_0\|_{E^3_q}.
}

For the bilinear term, using $s<\infty$ and Lemma \ref{lemma:duhamel.amalgam} with $\tilde{s}=s/2$, $\td p = p/2$, and $\td m = m$ (so that $\si=0$), we obtain
\EQS{\label{eq-0321-a}
\norm{B(f,g)}_{E^{s,p}_{T,m}} 
\le C(1+T^{1 - \frac 1s}) \norm{f\otimes g}_{E^{\frac s 2,\frac p 2}_{T,m}}
&\le C(1+T^{1 - \frac 1s}) \norm{f}_{E^{s,p}_{T,m}} \norm{g}_{E^{s,p}_{T,\infty}}\\
&\le C(1+T^{1 - \frac 1s}) \norm{f}_{E^{s,p}_{T,m}} \norm{g}_{E^{s,p}_{T,m}},
}
by H\"{o}lder's inequality and we have used the inclusion $E^{s,p}_{T,m} \subset E^{s,p}_{T,\infty}$ in the last inequality. 
Also, by Lemma \ref{Wa-estimate} and H\"older inequality \eqref{Holder}%
\EQN{ 
\norm{B(f,g)}_{E^6_q}(t) 
&\lec \int_0^t \bke{\frac1{(t-\tau)^{\frac12}} + \frac1{(t-\tau)^{\frac34}}} \norm{(f\otimes g)(\tau)}_{E^3_{q}}\, d\tau\\
&\le \int_0^t \bke{\frac1{(t-\tau)^{\frac12}} + \frac1{(t-\tau)^{\frac34}}} \norm{f(\tau)}_{E^6_q} \norm{g(\tau)}_{E^6_q}\, d\tau\\
&\lec (1+t^{-1/4}) \norm{f}_{\tilde{\mathcal{F}}_T} \norm{g}_{\tilde{\mathcal{F}}_T}.
}
Hence
\[
\| B(f,g)\|_{\mathcal F_T} \le C_2 (1+T^{\frac 1 4}+T^{1 - \frac 1s}  )  \|f\|_{\mathcal F_T} \|g\|_{\mathcal F_T}.
\]
By \eqref{3.21a}, taking $\|u_0\|_{E^3_q}$ small it is possible to ensure 
\EQ{
\| e^{t\Delta}u_0\|_{\mathcal F_T}  \leq \big[  4C_2 (   1+T^{1/4}+T^{1 - \frac 1s} ) 	  \big]^{-1}.
}
The Picard contraction theorem then guarantees the existence of a mild solution $u$ to \eqref{NS} 
so that
\[
\| u\|_{\mathcal F_T} \leq  2\| e^{t\Delta}u_0\|_{\mathcal F_T}.
\]
This solution is unique among all mild solutions $v$ with data $u_0$ satisfying $\|v\|_{\mathcal F_T}\le 2\| e^{t\Delta}u_0\|_{\mathcal F_T}$.
In fact, since we can also apply the Picard contraction to $\td {\mathcal F}_T$, the solution is also unique in the class $\|v\|_{\td {\mathcal F}_T}\leq 2\| e^{t\Delta}u_0\|_{\td {\mathcal F}_T}$.

Next, we show that a solution $u\in \mathcal{F}_T$ with small enough initial data $u_0\in E^3_q$ also belongs to $\mathcal{E}_T$. Let $\{u^{(n)}\}_{n\ge1}$ be the Picard iteration sequence in $\mathcal{F}_T$. 
By construction, 
\[
\norm{u^{(n)}}_{ {\mathcal{F}}_T}  \le  2C_1 ( 1 + T^{\frac1 s +\epsilon} +T^{\frac {1}4}   ) \| u_0\|_{E^3_q}< 1.
\]

Note that
\[
\norm{u^{(n)}}_{\td{\mathcal{E}}_T} 
\le \norm{e^{t\De} u_0}_{\td{\mathcal{E}}_T} + \norm{B(u^{(n-1)}, u^{(n-1)})}_{\td{\mathcal{E}}_T}.
\]
By Lemma \ref{Wa-estimate},
\[
\norm{e^{t\De} u_0}_{\td{\mathcal{E}}_T} \le C(1+T^{1/2})  \| u_0\|_{E^3_q}.
\]
As is usual in arguments like this, we now seek estimates for $B(f,g)$ in $\td{\mathcal{E}}_T$ in terms of measurements of $f$ and $g$ in $\td{\mathcal{F}}_T$ and $\td{\mathcal{E}}_T$. 
We have by Lemma \ref{Wa-estimate} and H\"older inequality \eqref{Holder},\EQS{\label{ineq:E3qintegral}
\| B(f,g)\|_{E^3_q}&\lec \int_0^t \bigg(\frac 1 {(t-\tau)^{\frac 1 2  }}+\frac 1 {(t-\tau)^{\frac 3 4  }} \bigg)\|(f\otimes g)(\tau)\|_{E^2_q}\,d\tau\\
&\le \int_0^t \bigg(\frac 1 {(t-\tau)^{\frac 1 2  }}+\frac 1 {(t-\tau)^{\frac 3 4  }} \bigg)  \|f(\tau)\|_{E^3_q}\|g(\tau)\|_{E^6_q}   \,d\tau\\
&\lec 	(1+T^{1/4})			\|f\|_{\td{\mathcal{E}}_T}\|g\|_{\td{\mathcal{F}}_T} ,
}
and
\EQN{
t^{\frac 1 2}\| B(f,g)\|_{E^\I_q} 
&\lec t^{\frac12}\int_0^t \bigg(\frac 1 {(t-\tau)^{\frac 1 2  }}+\frac 1 {(t-\tau)^{\frac 3 4  }} \bigg) \| (f\otimes g)(\tau)\|_{E^6_q} \,d\tau\\
&\le t^{\frac12} \int_0^t \bigg(\frac 1 {(t-\tau)^{\frac 1 2  }}+\frac 1 {(t-\tau)^{\frac 3 4  }} \bigg)  \|f(\tau)\|_{E^\infty_q}\|g(\tau)\|_{E^6_q}   \,d\tau\\
&\lec (1+T^{1/4}) \|f\|_{\td{\mathcal{E}}_T}\|g\|_{\td{\mathcal{F}}_T} .
}
By switching $f,g$ in the estimates, %
\[
\| B(f,g)\|_{\td{\mathcal E}_T} 
\lec (1 + T^{1/4}) \min \bke{ \| f\|_{\td{\mathcal E}_T}\|g\|_{\td{\mathcal F}_T},\, \| g\|_{\td{\mathcal E}_T}\|f\|_{\td{\mathcal F}_T}}. 
\]

Returning to our main goal, we have now that 
\EQN{
\norm{u^{(n)}}_{\td{\mathcal E}_T} 
&\le \norm{e^{t\De} u_0}_{\td{\mathcal E}_T} + \norm{B(u^{(n-1)}, u^{(n-1)})}_{\td{\mathcal E}_T} \\
&\le C_T  \| u_0\|_{E^3_q}  +  C_T' \norm{u^{(n-1)}}_{\td{\mathcal F}_T} \norm{u^{(n-1)}}_{\td{\mathcal E}_T}\\
&\le C_T  \| u_0\|_{E^3_q}  +  C_T' C_T''  \| u_0\|_{E^3_q} \norm{u^{(n-1)}}_{\td{\mathcal E}_T}.
}
Thus, if $\| u_0\|_{E^3_q}  \le (2  C_T' C_T'')^{-1}$, then $\norm{u^{(n)}}_{\td{\mathcal E}_T} 
$ is uniformly bounded by $2C_T  \| u_0\|_{E^3_q} $.
We now obtain the inclusion $u\in \td{\mathcal E}_T$ as follows:
\EQS{\label{ineq:difference.picard.iterates}
\norm{u^{(n+1)} - u^{(n)}}_{\td{\mathcal E}_T} 
&= \norm{B(u^{(n)},u^{(n)}) - B(u^{(n-1)},u^{(n-1)})}_{\td{\mathcal E}_T} \\
&\le \norm{B(u^{(n)}-u^{(n-1)},u^{(n)})}_{\td{\mathcal E}_T} + \norm{B(u^{(n-1)},u^{(n)}-u^{(n-1)})}_{\td{\mathcal E}_T}\\
&\lec \norm{u^{(n)} - u^{(n-1)}}_{\td{\mathcal F}_T} \bke{\norm{u^{(n)}}_{\td{\mathcal E}_T} + \norm{u^{(n-1)}}_{\td{\mathcal E}_T} }
\\&\lec \norm{u^{(n)} - u^{(n-1)}}_{\td{\mathcal E}_T} \bke{\norm{u^{(n)}}_{\td{\mathcal E}_T} + \norm{u^{(n-1)}}_{\td{\mathcal E}_T} }.
}
In particular, the above implies $\{u^{(n+1)}-u^{(n)}\}$ is Cauchy in $\td{\mathcal E}_T$ and hence the limit of $\{u^{(n)}\}$ is in $\mathcal E_T$ where we're noting that convergence in the spacetime norm in the definition of $\mathcal E_T$ is already implied by convergence in $\mathcal F_T$.

When $q \le s$, note that 
\[
\norm{u}_{L^s_T E^p_m} \le \norm{u}_{L^s_T E^p_q} \le \norm{u}_{E^{s,p}_{T,q}},
\]
using \eqref{ineq:embedding.parabolic.space} and $q \le s$ for the second inequality. This shows the $L^s_T E^p_m$-estimate in \eqref{eq1.7}.

We now address convergence to the initial data when $q<\I$.
By Lemma \ref{lemma:heat.eq.convergence.data} we have
\EQ{\label{limit:zero.heat}
\lim_{T'\to 0^+} \sup_{0<t<T'}t^{\frac 1 4} \| e^{t\Delta}u_0\|_{E^6_q} = \lim_{T'\to 0} \| e^{t\Delta}u_0\|_{\td{\mathcal F}_{T'}} = 0,
}
whenever $u_0\in E^3_q$. 
We extend this property to all Picard iterates via induction. Our base case is \eqref{limit:zero.heat} where we note that $u^{(0)}$ can be viewed as $e^{t\Delta}u_0$. 
Our inductive hypothesis is that 
\EQ{\label{limit:inductiveHyp}
\lim_{T'\to 0} \|u^{(n-1)}\|_{\td{\mathcal F}_{T'}} = 0.
}
By our above estimates in the class $\td {\mathcal F}_{T'}$ where we are taking $T'\leq T$, we have
\EQN{\label{inductiveStep}
\| u^{(n)}\|_{\td {\mathcal F}_{T'}} \leq \|e^{t\Delta} u_0\|_{\td {\mathcal F}_{T'}}+ \| B(u^{(n-1)},u^{(n-1)})\|_{\td {\mathcal F}_{T'}}\lec \|e^{t\Delta} u_0\|_{\td {\mathcal F}_{T'}}+\| u^{(n-1)}\|_{\td {\mathcal F}_{T'}}^2. 
}
Hence
\EQ{\label{limit:Picard}
\lim_{T'\to 0} \|u^{(n)}\|_{\td{\mathcal F}_{T'}} = 0,
}
 by \eqref{limit:zero.heat} and \eqref{limit:inductiveHyp}. This completes the induction.

The limit \eqref{limit:zero.heat}, convergence of the Picard iterates in $\td {\mathcal F}_T$ and \eqref{limit:Picard} imply that, by taking $T'$ small, we can make $\sup_{0<t<T'}t^{\frac 1 4} \| u(t)\|_{E^6_q}$ small. To elaborate,  we have 
\EQ{
\| u\|_{\td {\mathcal F}_{T'}} \leq \| u - u^{(n)}\|_{\td {\mathcal F}_{T'}}  + \| u^{(n)}\|_{\td {\mathcal F}_{T'}}.
}
We may choose $n$ large so that the first term is small and then make the second term small by taking $T'$ small.
Hence,
\EQ{\label{limit:zero.ns}
\lim_{T'\to 0^+} \| u\|_{\td{\mathcal F}_{T'}}= 0.
}
Using \eqref{ineq:E3qintegral} and \eqref{limit:zero.ns}, we have
\EQN{
\lim_{T'\to 0^+} \sup_{0<t<T'} \| B(u,u)\|_{E^3_q} (t)=0.
}
This and Lemma \ref{lemma:heat.eq.convergence.data} imply 
\[
\lim_{t\to 0} \| u-u_0\|_{E^3_q}=0.
\]

If $q=\I$ then we have a weaker mode of convergence. Fix a ball $B$. Take $R>0$ large so that $B\subset B_R(0)$.  We re-write the  bilinear form as  \[B(u,u) = B(u,u \chi_{B_R(0)}) + B(u,u(1-\chi_{B_R(0)})).\]
If $1< \omega<3$ then we have  
\[
\|   B(u,u \chi_{B_R(0)})  \|_{L^\omega }\lec \int_0^t \frac 1 {(t-\tau)^{\frac 1 2 + \frac 3 2 (  \frac 2 {3} -\frac 1 \omega )}} \| u^2\chi_{B_R(0)} \|_{L^{3/2}} (\tau) \,d\tau \lesssim_R t^{\frac 1 2 - \frac 3 2(\frac 2 3 - \frac 1 \omega)} \| u\|_{L^\I L^{3}_\uloc}^2.
\]
For any $R>0$, the above vanishes as $t\to 0$ provided $\omega<3$.

By taking $R= 2\max_{x\in B}|x|$, we can ensure that for all $x\in B$ and $|y|\geq R$ we have 
$\frac 1 2 |y|\leq |x-y|\leq \frac 3 2 |y|$. Hence, for $x\in B$,
\EQN{
|B(u,u \chi_{B_R(0)^c})(x,t) | &\lesssim  \int_0^t \int_{ y\in B_R(0)^c} \frac 1 {(|x-y|+\sqrt {t-\tau})^4}  |u|^2(y,\tau)\,dy\,d\tau
\\&\lesssim  t\sup_{0<\tau<t} \sum_{k=1}^\I  \frac 1 {2^{4k-4}R^4} \int_{  R2^{k-1}\leq  |y| < R2^{k} } |u|^2(y,\tau) \,dy
\\&\lesssim t \sup_{0<\tau<t}\sum_{k=1}^\I  \frac {R 2^k} {R^42^{4k-4}}   \bigg(\int_{  |y| < R2^{k} } |u|^3(y,\tau) \,dy\bigg)^\frac 2 3
\\&\lesssim t \sum_{k=1}^\I  \frac {R^3 2^{3k} } {R^42^{4k-4}}  \sup_{0<\tau<t}  \| u(\tau)\|_{L^3_\uloc}^2
\\&\lesssim \frac t R  \sup_{0<\tau<t} \| u(\tau)\|_{L^3_\uloc}^2.
}

Now, 
\EQN{
\| B(u,u) (t)\|_{L^\omega(B)} &\lesssim _R  \| B(u,u \chi_{B_R(0)}) (t)\|_{L^\omega}  +   \| B(u,u \chi_{B_R(0)^c}) (t)\|_{L^\I(B)} 
\\&\lesssim_R  t^{\frac 1 2 - \frac 3 2(\frac 2 3 - \frac 1 \omega)} \| u\|_{L^\I L^{3}_\uloc}^2 + t \sup_{0<\tau<t} \| u(\tau)\|_{L^3_\uloc}^2.
}
Hence 
\[
\lim_{t\to 0^+}\| B(u,u) (t)\|_{L^\omega(B)} = 0.
\]
Referring to  \cite[p. 394]{MaTe},
 we have \[\lim_{t\to 0}\|e^{t\Delta}u_0 -u_0\|_{L^\omega (B)} =0.\]   It follows that 
\[
\lim_{t\to 0}\| u(\cdot,t)-u_0(\cdot)\|_{L^\omega(B)} =0.
\]

We now prove continuity at positive times. Let $t_1>0$ be fixed. We will send $t\to t_1$.  Note that by Lemma \ref{lemma:heat.eq.convergence.data} we have $e^{t\Delta}u_0 -e^{t_1\Delta}u_0\to 0$ in $E^3_q$ as $t\to t_1$. We therefore only need to show $B(u,u)(t)\to B(u,u)(t_1)$. Following \cite[p.~86]{Tsai-book}, we take $\rho$ slightly less than $1$ so that $\rho t_1<t$  and write
\EQN{
B(u,u)(t)- B(u,u)(t_1) &= \int_{\rho t_1}^t e^{(t-\tau)\Delta} \mathbb P\nb \cdot F\,d\tau
{-} \int_{\rho t_1}^{t_1} e^{(t_1-\tau)\Delta} \mathbb P\nb \cdot F\,d\tau 
\\&+\int_{0}^{\rho t_1} \big(e^{(t-\rho t_1)\Delta}-e^{(t_1-\rho t_1)\Delta} \big)e^{(\rho t_1-\tau)\Delta} \mathbb P\nb \cdot F\,d\tau 
}
where $F=u\otimes u(\tau)$. For the first and second terms,
by \eqref{Wa-estimate-bilinear} with $p=\td p = 3$, $\td q = q$ and using the embedding $E^\infty_q\subset E^\infty_\infty$,
we have 
\EQN{
\int_{\rho t_1}^t  \|  e^{(t-\tau)\Delta} \mathbb P\nb\cdot  F\|_{E^3_q}\,d\tau  &\lesssim  \int_{\rho t_1}^{t} \frac 1 {(t-\tau)^{\frac 1 2}\tau^\frac 1 2} \| u\|_{E^3_q} (\tau) \tau^{\frac 1 2}\| u\|_{E^\I_q} \,d\tau
 \lesssim \frac {(t-\rho t_1)^\frac 1 2}	{(\rho t_1)^\frac 1 2}	\|u\|_{\mathcal E_T}^2,
}
and 
\EQN{
\int_{\rho t_1}^{t_1} \|e^{(t_1-\tau)\Delta} \mathbb P\nb \cdot F\|_{E^3_q}\,d\tau  &\lesssim  \int_{\rho t_1}^{t_1} \frac 1 {(t_1-\tau)^{\frac 1 2}\tau^\frac 1 2} \| u\|_{E^3_q} (\tau) \tau^{\frac 1 2}\| u\|_{E^\I_q} \,d\tau
 \lesssim \frac {(t_1-\rho t_1)^\frac 1 2}	{(\rho t_1)^\frac 1 2}	\|u\|_{\mathcal E_T}^2,
}
both of which can be made arbitrarily small by taking $\rho t_1$ close to $t_1$ and $t$ close to $t_1$.

For the third term we note that by Lemma \ref{lemma:heat.eq.convergence.data}, for each $0<\tau<\rho t_1$, we have
\[
\|  \big(e^{(t-\rho t_1)\Delta}-e^{(t_1-\rho t_1)\Delta} \big)e^{(\rho t_1-\tau)\Delta} \mathbb P\nb \cdot F(\tau) \|_{E^3_q}\to 0 \text{ as }t\to t_1,
\]
which follows {(even if $q=\infty$)}  the fact that $e^{(\rho t_1-\tau)\Delta} \mathbb P\nb \cdot F(\tau)\in  E^3_q$, which is a consequence of Lemma \ref{Wa-estimate}.
Additionally,  
\EQN{
& \|  \big(e^{(t-\rho t_1)\Delta}-e^{(t_1-\rho t_1)\Delta} \big)e^{(\rho t_1-\tau)\Delta} \mathbb P\nb \cdot F(\tau)	\|_{E^3_q}
\\&\lesssim  
 \bigg( \frac 1 {(t-\tau)^\frac 1 2 \tau^\frac 1 2} + \frac 1 {(t_1-\tau)^\frac 1 2 \tau^\frac 1 2}\bigg) \| u\|_{\mathcal E_T}^2 \in L^1(0,\rho t_1),
}
where integration in   $L^1(0,\rho t_1)$ is with respect to $\tau$.
So, by Lebesgue's dominated convergence theorem, 
\Eq{
\int_0^{\rho t_1 } \|  \big(e^{(t-\rho t_1)\Delta}-e^{(t_1-\rho t_1)\Delta} \big)e^{(\rho t_1-\tau)\Delta} \mathbb P\nb \cdot F(\tau)	\|_{E^3_q}\,d\tau \to 0 \text{ as }t\to t_1.
}
{The above show the continuity of $u(t)$ at positive times.}
\end{proof}

In the following we will prove spacetime integral bound \eqref{eq1.7} for $p=3$, $s=\infty$, i.e. in $E^{\infty,3}_{T,q}$. For this purpose we will need \emph{time-weighted space-time integral estimates} for both linear and nonlinear terms in the following two lemmas. Because the constants depend on $T$, these estimates cannot be extended to $T=\infty$.

\begin{lemma}\label{lem-new-linear}
Let $d=3$ and $0<T<\infty$.
For $3<p\le \infty$ and $a=\frac{p-3}{2p}\in (0,1/2]$, we have
\EQ{\label{lem-new-linear-eq1}
\norm{t^a e^{t\De}u_0}_{E^{\infty,p}_{T,q}} \lec (1+T^a) \norm{u_0}_{E^3_q}.
}
{For $p=3$ and $a=0$, we also have
\EQ{\label{lem-new-linear-eq2}
\norm{e^{t\De}u_0}_{E^{\infty,3}_{T,q}} \lec \ln(2+T) \norm{u_0}_{E^3_q}.
}}
\end{lemma}

\begin{proof}
For $x\in B_1(k)$ we write
\[t^a |e^{t\De} u_0(x)| \lec \int_{\R^3} \frac{e^{-\frac{|x-y|^2}{ct}}}{(|x-y|+\sqrt t)^{3-2a}} |u_0(y)|\, dy =: U^k_1 + U^k_2,\] where $U^k_1 =  \int_{\R^3} \frac{e^{-\frac{|x-y|^2}{ct}}}{(|x-y|+\sqrt t)^{3-2a}} (|u_0| \chi_{B_4(k)})(y)\, dy$ and $U^k_2 =  \int_{\R^3} \frac{e^{-\frac{|x-y|^2}{ct}}}{(|x-y|+\sqrt t)^{3-2a}} (|u_0| (1 - \chi_{B_4(k)}))(y)\, dy$.
Using Young's convolution inequality,
\[
\norm{U^k_1}_{L^p(B_1(k))} \le \norm{U^k_1}_{L^p(\R^3)} \lec \norm{\frac{e^{-\frac{|\cdot|^2}{ct}}}{(|\cdot|+\sqrt t)^{3-2a}}}_{L^{\frac{3p}{2p+3}}({\R^3})} \norm{u_0}_{L^3 {(B_4(k))}}
\lec \norm{u_0}_{L^3 {(B_4(k))}}.
\]
Thus, 
\[
\norm{\sup_{0<t<T} \norm{U^k_1}_{L^p(B_1(k))}}_{\ell^q(k\in\ZZ^3)} \lec \norm{u_0}_{E^3_q}. 
\]
On the other hand,
\[
\norm{U^k_2}_{L^p(B_1(k))} \lec \sum_{|k'|\ge 1} \frac{e^{-\frac{|k'|^2}{ct}}}{(|k'|+\sqrt t)^{3-2a}} \norm{u_0}_{L^3(k'-k)}.
\]
So
\EQN{
\norm{\sup_{0<t<T} \norm{U^k_2}_{L^p(B_1(k))}}_{\ell^q(k\in\ZZ^3)}
&\lec \norm{\sup_{0<t<T} \sum_{|k'|\ge 1} \frac{e^{-\frac{|k'|^2}{ct}}}{(|k'|+\sqrt t)^{3-2a}} \norm{u_0}_{L^3(B_1(k'-k))}}_{\ell^q(k\in\ZZ^3)}\\
&\le \norm{\sum_{|k'|\ge 1} \frac{e^{-\frac{|k'|^2}{cT}}}{|k'|^{3-2a}} \norm{u_0}_{L^3(B_1(k'-k))}}_{\ell^q(k\in\ZZ^3)}\\
&\le \bigg\|\frac{e^{-\frac{|k|^2}{cT}}}{|k|^{3-2a}} \bigg\|_{\ell^1(k\in{\ZZ^3\setminus\{0\} })} \norm{u_0}_{E^3_q}
\lec T^a \norm{u_0}_{E^3_q}.
}
{In the last inequality we have used $a>0$.}
Therefore,
\EQS{
\norm{t^a e^{t\De} u_0}_{E^{\infty,p}_{T,q}} 
&\lec \norm{\sup_{0<t<T} \norm{U^k_1}_{L^p(B_1(k))}}_{\ell^q(k\in\ZZ^3)} + \norm{\sup_{0<t<T} \norm{U^k_2}_{L^p(B_1(k))}}_{\ell^q(k\in\ZZ^3)}\\
&\lec (1 + T^a) \norm{u_0}_{E^3_q}.
}
{This shows \eqref{lem-new-linear-eq1}. When $p=3$ and $a=0$, since
\[
\bigg\|\frac{e^{-\frac{|k|^2}{cT}}}{|k|^{3}} \bigg\| _{\ell^1(k\in\ZZ^3\setminus\{0\} ) }  \lec \ln (2+T),
\]
the same argument gives \eqref{lem-new-linear-eq2}.} 
\end{proof}

\begin{lemma}\label{lem-new-bilinear}
{Let $d=3$ and $0<T<\infty$.}
For {nonnegative} $a,b$ with $a+b<1$ and $p\in[1,\infty]$, $\td p\in(3,\infty]$ with $\frac12 - \frac3{2\td p} - b\ge0$, 
\[
\norm{t^a B(f,g)}_{E^{\infty,p}_{T,q}} \lec \bke{T^{\frac12 - \frac3{2\td p} - b} + T^{1-b} } \min\bke{\norm{t^af}_{E^{\infty,p}_{T,q}}\norm{t^bg}_{E^{\infty,\td p}_{T,q}}, \norm{t^ag}_{E^{\infty,p}_{T,q}}\norm{t^bf}_{E^{\infty,\td p}_{T,q}}}.
\]
\end{lemma}
\begin{proof}
Decompose in $B_1(k)$ for $k \in \ZZ^3$,
\[
B_i(f,g)(x,t)= \int_0^t (F_1^k(x,t,\tau)+F_2^k(x,t,\tau))\,d\tau,
\]
with $F_1^k(x,t,\tau) = \int_{\R^3}  \pd_l S_{ij}(x-y,t-\tau) (f_lg_j\chi_{B_4(k)})(y,\tau)\, dy$ and $F_2^k(x,t,\tau) = \int_{\R^3}  \pd_l S_{ij}(x-y,t-\tau) (f_lg_j(1-\chi_{B_4(k)}))(y,\tau)\, dy$.
By the computation in \eqref{est-u1.a} without time integration,
\EQN{
\norm{F^k_1(\cdot,t,\tau)}_{L^p(B_1(k))} 
&\lec (t-\tau)^{-\frac32\bke{\frac1r-\frac1p}-\frac12} \norm{f\otimes g}_{L^r(B_4(k))},\quad \frac1r=\frac1p+\frac1{\td p},\\
&\lec (t-\tau)^{-\frac32\bke{\frac1r-\frac1p}-\frac12} \norm{f}_{L^p(B_4(k))} \norm{g}_{L^{\td p}(B_4(k))}(\tau).
}
So
\begin{align}
&\norm{\sup_{0<t<T} t^a \int_0^t \norm{F^k_1(\cdot,t,\tau)}_{L^p(B_1(k))}\, d\tau}_{\ell^q(k\in\ZZ^3)}\nonumber\\
&\quad \le \norm{\sup_{0<t<T} t^a \int_0^t \frac1{(t-\tau)^{\frac12+\frac3{2\td p}}\tau^{a+b}}\, d\tau \norm{t^a f}_{L^\infty_T L^p(B_4(k))} \norm{t^b g}_{L^\infty_T L^{\td p}(B_4(k))} }_{\ell^q(k\in\ZZ^3)}\nonumber\\
&\quad \lec T^{\frac12 - \frac3{2\td p} - b} \norm{ \norm{t^a f}_{L^\infty_T L^p(B_4(k))} \norm{t^b g}_{L^\infty_T L^{\td p}(B_4(k))} }_{\ell^q(k\in\ZZ^3)}\nonumber\\
&\quad \le T^{\frac12 - \frac3{2\td p} - b} \norm{ \norm{t^a f}_{L^\infty_T L^p(B_4(k))}}_{\ell^q(k\in\ZZ^3)} \norm{ \norm{t^b g}_{L^\infty_T L^{\td p}(B_4(k))} }_{\ell^\infty(k\in\ZZ^3)}\nonumber\\
&\quad \lec T^{\frac12 - \frac3{2\td p} - b} \norm{t^af}_{E^{\infty,p}_{T,q}}\norm{t^bg}_{E^{\infty,\td p}_{T,q}}.
\label{lem-new-bilinear-eq1}
\end{align}

On the other hand, the computation in \eqref{eq-F2k} without time integration yields 
\[
\norm{F^k_2(\cdot,t,\tau)}_{L^p(B_1(k))} \lec \sum_{|k'|\ge1} \frac1{(|k'|+\sqrt{t-\tau})^4} \norm{(f\otimes g)(\tau)}_{L^r(B_2(k-k'))},\quad r\ge1.
\]
Choosing $r$ so that $\frac1r=\frac1p+\frac1{\td p}$, one has
\EQN{
&\norm{\sup_{0<t<T} t^a \int_0^t \norm{F^k_2(\cdot,t,\tau)}_{L^p(B_1(k))}\, d\tau}_{\ell^q(k\in\ZZ^3)}\\
&\quad \lec \norm{\sup_{0<t<T}  \int_0^t  \sum_{|k'|\ge 1} \frac{t^a}{(|k'|+\sqrt{t-\tau})^4}  \norm{f(\cdot, \tau)}_{L^p(B_2(k-k'))} \norm{g(\cdot,\tau)}_{L^{\td p}(B_2(k-k'))}\, d\tau}_{\ell^q(k\in\ZZ^3)}\\
&\quad \le \norm{\sup_{0<t<T}  \sum_{|k'|\ge 1} \int_0^{t}  \frac{t^a\,d\tau }{(|k'|+\sqrt{t-\tau})^4 \tau^{a+b}}\, \norm{\tau^a f}_{L^\infty_TL^p(B_2(k-k'))} \norm{\tau^b g}_{L^\infty_TL^{\td p}(B_2(k-k'))}}_{\ell^q(k\in\ZZ^3)}\\
&\quad \lec T^{1-b} \norm{ \sum_{|k'|\ge 1} \frac1{|k'|^4} \norm{t^a f}_{L^\infty_TL^p(B_2(k-k'))} \norm{t^b g}_{L^\infty_TL^{\td p}(B_2(k-'k))} }_{\ell^q(k\in\ZZ^3)}
}
Using Young's convolution inequality for the $\ell^q$ norm above, 
\begin{align}
&\norm{\sup_{0<t<T} t^a \int_0^t \norm{F^k_2(\cdot,t,\tau)}_{L^p(B_1(k))}\, d\tau}_{\ell^q(k\in\ZZ^3)} \nonumber\\
&\quad \lec T^{1-b} \norm{\frac{\chi_{k\neq0}}{|k|^4}}_{\ell^1(k\in\ZZ^3)} \norm{\norm{t^a f}_{L^\infty_TL^p(B_2(k))} \norm{t^b g}_{L^\infty_TL^{\td p}(B_2(k))}}_{\ell^q(k\in\ZZ^3)}\nonumber\\
&\quad \lec T^{1-b} \norm{ \norm{t^a f}_{L^\infty_T L^p(B_4(k))}}_{\ell^q(k\in\ZZ^3)} \norm{ \norm{t^b g}_{L^\infty_T L^{\td p}(B_4(k))} }_{\ell^\infty(k\in\ZZ^3)}\nonumber\\
&\quad \lec T^{1-b} \norm{t^af}_{E^{\infty,p}_{T,q}}\norm{t^bg}_{E^{\infty,\td p}_{T,q}}.\label{lem-new-bilinear-eq2}
\end{align}
The lemma follows by summing \eqref{lem-new-bilinear-eq1} and \eqref{lem-new-bilinear-eq2}, and by switching $f,g$ in the estimates.
\end{proof}

\begin{remark} It is possible to prove time-weighted
space-time integral estimates for the linear and nonlinear terms, with \emph{finite time exponent} $s<\infty$. For example, 
\EQ{\label{0422b}
\norm{t^{a} e^{t\De}u_0}_{E^{s,p}_{T,q}} \lec_T \norm{u_0}_{E^3_q}, \quad a(s,p)=\frac12 -\frac 3{2p}-\frac 1s>0.
}
It extends Lemma \ref{lemma:linear.heat.amalgam.spacetime-new}. We limit ourselves here to $s=\infty$.
\end{remark}

\begin{proof}[Proof of spacetime integral estimates in $E^{\infty,3}_{T,q}$ in Theorem \ref{thrm:critical}]
\mbox{}

To prove the spacetime integral bound \eqref{eq1.7} for $p=3$, $s=\infty$, we work in the spaces with the norms
\[
\norm{f}_{\mathcal E_T^*} = \norm{f}_{E^{\infty,3}_{T,q}} + \norm{t^{\frac12}f}_{E^{\infty,\infty}_{T,q}},\qquad
\norm{f}_{\mathcal F_T^*} = \norm{t^{\frac14} f}_{E^{\infty,6}_{T,q}}.
\]
We also have $\mathcal E_T^* \subset \mathcal F_T^*$ since
\EQN{
\norm{t^{\frac14} f}_{E^{\infty,6}_{T,q}} 
&= \norm{\norm{t^{\frac14}\norm{f}_{L^6(B_1(k))}}_{L^\infty_T}}_{\ell^q(k\in\ZZ^3)} 
\le \norm{\norm{t^{\frac14} \norm{f}_{L^\infty(B_1(k))}^{1/2} \norm{f}_{L^3(B_1(k))}^{1/2} }_{L^\infty_T}}_{\ell^q(k\in\ZZ^3)} \\
&\le \norm{\norm{t^{\frac12} f}_{L^\infty_TL^\infty(B_1(k))}^{1/2} \norm{f}_{L^\infty_TL^3(B_1(k))}^{1/2}  }_{\ell^q(k\in\ZZ^3)}\\
&\le \norm{\norm{t^{\frac12} f}_{L^\infty_TL^\infty(B_1(k))}^{1/2}}_{\ell^{2q}(\ZZ^3)}  \norm{\norm{f}_{L^\infty_TL^3(B_1(k))}^{1/2}  }_{\ell^{2q}(k\in\ZZ^3)}
= \norm{t^{\frac12} f}_{E^{\infty,\infty}_{T,q}}^{1/2} \norm{f}_{E^{\infty,3}_{T,q}}^{1/2}.
}

For the linear estimate in $\mathcal F^*_T$, taking $p=6$ so that $a=1/4$ in Lemma \ref{lem-new-linear}, we have
\EQS{
\norm{e^{t\De} u_0}_{\mathcal F^*_T} 
= \norm{t^{\frac14} e^{t\De} u_0}_{E^{\infty,6}_{T,q}} 
\le C_1^* (1 + T^{\frac14}) \norm{u_0}_{E^3_q}.
}

For bilinear estimate, taking $a=b=1/4$, $p=\td p=6$ in Lemma \ref{lem-new-bilinear},
\[
\norm{B(f,g)}_{\mathcal F_T^*} = \norm{t^{\frac14}B(f,g)}_{E^{\infty,6}_{T,q}}
\le C_2^* (1 + T^{\frac34}) \norm{t^{\frac14}f}_{E^{\infty,6}_{T,q}} \norm{t^{\frac14}g}_{E^{\infty,6}_{T,q}}
= C_2^* (1 + T^{\frac34}) \norm{f}_{\mathcal F_T^*} \norm{g}_{\mathcal F_T^*}.
\]
By choosing $\norm{u_0}_{E^3_q}$ small enough so that $\norm{u_0}_{E^3_q}<\frac1{4C_1^*C_2^*\bke{1+T^{\frac14}}\bke{1+T^{\frac34}}}$, Picard iteration yields a unique mild solution satisfying 
\[
\norm{u}_{\mathcal F_T^*} \le 2 C_1^*(1+T^{\frac14}) \norm{u_0}_{E^3_q}.
\]

Now, we claim that a solution $u\in\mathcal F^*_T$ with small enough $u_0\in E^3_q$ also belongs to $\mathcal E^*_T$.
By \eqref{lem-new-linear-eq2} of Lemma \ref{lem-new-linear} (or by \eqref{lemma:linear.sotkes.amalgam.a} with $d=p=r=3$, $s=\infty$, and constant $1+T^\frac14$),
\[
\norm{e^{t\De} u_0}_{E^{\infty,3}_{T,q}} \lec \ln(2+T) \norm{u_0}_{E^3_q}.
\]
Taking $p=\infty$ so that $a=1/2$ in Lemma \ref{lem-new-linear},
\[
\norm{t^{\frac12} u_0}_{E^{\infty,\infty}_{T,q}} \lec (1+T^\frac12) \norm{u_0}_{E^3_q}.
\]
Therefore, one has 
\EQ{
\norm{e^{t\De}u_0}_{\mathcal E^*_T} \lec (1+T^\frac12)  \norm{u_0}_{E^3_q}.
}

We next show
\EQ{ 
\norm{B(f,g)}_{\mathcal E^*_T} \lec_T \min\bke{\norm{f}_{\mathcal E^*_T}\norm{g}_{\mathcal F^*_T}, \norm{g}_{\mathcal E^*_T}\norm{f}_{\mathcal F^*_T}}.
}
Indeed, by choosing $(a,b,p,\td p)=(0,1/4,3,6)$ and $(a,b,p,\td p)=(1/2,1/4,\infty,6)$ in Lemma \ref{lem-new-bilinear}, we get
\[
\norm{B(f,g)}_{E^{\infty,3}_{T,q}} \lec (1+T^{\frac34}) \min\bke{ \norm{f}_{E^{\infty,3}_{T,q}} \norm{t^{\frac14} g}_{E^{\infty,6}_{T,q}}, \norm{ g}_{E^{\infty,3}_{T,q}} \norm{t^{\frac14} f}_{E^{\infty,6}_{T,q}} }
\]
and
\[
\norm{t^{\frac12} B(f,g)}_{E^{\infty,\infty}_{T,q}} \lec (1+T^{\frac34}) \min\bke{ \norm{t^{\frac12}f}_{E^{\infty,\infty}_{T,q}} \norm{t^{\frac14} g}_{E^{\infty,6}_{T,q}}, \norm{t^{\frac12}g}_{E^{\infty,\infty}_{T,q}} \norm{t^{\frac14} f}_{E^{\infty,6}_{T,q}} },
\]
respectively.
By the same argument as before, we can show $u \in \mathcal E^*_T$, by taking possibly a smaller $T>0$. In particular, $u\in E^{\infty,3}_{T,q}$ and its norm is bounded $\norm{u_0}_{E^3_q}$.
The case $s=\infty$ of Theorem \ref{thrm:critical} then follows from the embeddings $\norm{u}_{L^\infty E^3_q}\le \norm{u}_{E^{\infty,3}_{T,q}}$ and $\norm{t^{\frac12}u}_{L^\infty_TE^\infty_q}\le \norm{t^{\frac12}u}_{E^{\infty,\infty}_{T,q}}$.
\end{proof}

\subsection{Mild solutions in critical spaces with enough decay}
We finally prove Theorem \ref{thrm:critical2} for critical data $u_0 \in E^3_q$ with enough decay, $1\le q\le 3$. For local wellposedness, we can allow the data to be large thanks to the decay. We will assume smallness for global wellposedness.

\begin{proof}[Proof of Theorem \ref{thrm:critical2}]
Although this proof is very similar to the proof of Theorem \ref{thrm:critical}, there are technical modifications throughout which necessitate we revisit the details.
By Lemma \ref{Wa-estimate}, we have for $1\le q \le 3$,
\EQ{\label{u-lin-est}
\sup_{0<t<\I}\bke{ \norm{e^{t\De}u_0}_{E^3_q}  +  t^{\frac12} \norm{e^{t\De}u_0}_{E^\I_{q_2}} +t^{\frac14} \norm{e^{t\De}u_0}_{E^6_{q_1}}} \le C \norm{u_0}_{E^3_q},
}
where
\EQ{
\frac 1{q_1} = \frac 1q - \frac 16, \quad
\frac 1{q_2} = \frac 1q - \frac 13 , \quad
q < q_1 < q_2 \le \oo.
}

Motivated by \eqref{u-lin-est},
for $0<T\le \oo$ and $1\le q \le 3$, let $\mathcal{E}_T$, $\mathcal{F}_T$ be Banach spaces defined as
\EQ{
\mathcal{E}_T 
:= \bket{ f\in L^\infty(0,T;E^3_q):\ t^{\frac12} f(\cdot,t) \in L^\infty(0,T; E^\I_{q_2})
},
}
and
\EQ{
\mathcal{F}_T := \bket{ f :\ t^{\frac14}f(\cdot,t) \in L^\infty(0,T; E^6_{q_1})} ,
}
with norms
\[
\norm{f}_{\mathcal{E}_T} 
:= \sup_{0<t<T} \norm{f(\cdot,t)}_{E^3_q}  + \sup_{0<t<T} t^{\frac12} \norm{f(\cdot,t)}_{E^\I_{q_2}}\quad \text{ and } \quad 
\norm{f}_{\mathcal{F}_T} 
:= \sup_{0<t<T} t^{\frac14} \norm{f(\cdot,t)}_{E^6_{q_1}},
\]
respectively. 
We claim that $\mathcal{E}_T \subset \mathcal{F}_T$. Indeed, using $1/q_1 = 1/(2q_2) + 1/(2q)$,
\EQN{
t^{\frac14} \norm{f(\cdot,t)}_{ E^6_{q_1}}
&= t^{\frac14} \norm{ \norm{f(\cdot,t)}_{L^6(B_1(k))} }_{\ell^{q_1}(k\in\ZZ^3)} \\
&\le t^{\frac14}  \norm{ \norm{f(\cdot,t)}_{L^\infty(B_1(k))}^{1/2} \norm{f(\cdot,t)}_{L^3(B_1(k))}^{1/2} }_{\ell^{q_1}(k\in\ZZ^3)}
\\
&\le t^{\frac14}   \norm{ \norm{f(\cdot,t)}_{L^\oo(B_1(k))}^{1/2} }_{\ell^{2q_2}(k\in\ZZ^3)}\cdot  \norm{ \norm{f(\cdot,t)}_{L^3(B_1(k))}^{1/2} }_{\ell^{2q}(k\in\ZZ^3)}
\\
&= (t^{\frac12}  \norm{f(\cdot,t)}_{E^\infty_{q_2}})^{1/2} \norm{f(\cdot,t)}_{E^3_q}^{1/2}.
}
The above is true even if $q=3$ and $q_2=\I$.

By Lemma \ref{Wa-estimate} again and {H\"older inequality \eqref{Holder} using} $2q \ge q_1$ due to $q\le 3$, 
\EQS{\label{0819a}
\norm{B(f,g)}_{E^6_{q_1}}(t) 
&\lec \int_0^t  \frac1{(t-\tau)^{\frac34}} \norm{f\otimes g(\tau)}_{E^3_{q}}\, d\tau\\
&\le \int_0^t  \frac1{(t-\tau)^{\frac34}} \norm{f(\tau)}_{E^6_{q_1}} \norm{g(\tau)}_{E^6_{q_1}}\, d\tau\\
&\le \int_0^t  \frac1{(t-\tau)^{\frac34}} \tau^{-1/4}\norm{f}_{{\mathcal{F}}_T} \tau^{-1/4}\norm{g}_{{\mathcal{F}}_T}\, d\tau\\
&\lec t^{-1/4} \norm{f}_{{\mathcal{F}}_T} \norm{g}_{{\mathcal{F}}_T}.
}
Hence, %
\[
\| B(f,g)\|_{{\mathcal F}_T} \leq c_*   \|f\|_{{\mathcal{F}}_T} \|g\|_{{\mathcal{F}}_T},
\]
where $c_*$ is a universal constant. 

Concerning the caloric extension of $u_0$, we have for $\|u_0\|_{E^3_q}$ of any size that
\[
\lim_{T\to 0} \| e^{t\Delta}u_0\|_{{\mathcal F}_T}  = 0,
\]
by \eqref{vanishing-t=0} of Lemma \ref{lemma:heat.eq.convergence.data}.  
 Hence, there exists $T=T(u_0)$ so that 
\EQ{\label{small.caloric}
 \| e^{t\Delta}u_0\|_{{\mathcal F}_T} \lesssim c_*^{-1}.
}
If, on the other hand, $\|u_0\|_{E^3_q} \lesssim c_*^{-1}$, then by \eqref{u-lin-est}, %
we have  \eqref{small.caloric} for $T=\I$.
The Picard contraction theorem then guarantees the existence of a mild solution $u$ to \eqref{NS} 
so that
\[
\| u\|_{\mathcal F_T} \leq  2\| e^{t\Delta}u_0\|_{\mathcal F_T}.
\]
This solution is unique among all mild solutions $v$ with data $u_0$ satisfying $\|v\|_{\mathcal F_T}\leq 2\| e^{t\Delta}u_0\|_{\mathcal F_T}$.%

Next, we show that a solution $u\in \mathcal{F}_T$ with initial data $u_0\in E^3_q$ also belongs to $\mathcal{E}_T$. Let $\{u^{(n)}\}_{n\ge1}$ be the Picard iteration sequence in $\mathcal{F}_T$. %
By construction, 
\EQ{\label{unFTbound}
\norm{u^{(n)}}_{ {\mathcal{F}}_T}  \le  2 \|e^{t\Delta}u_0\|_{\mathcal F_T}.
}
Note that
\[
\norm{u^{(n)}}_{\mathcal{E}_T} 
\le \norm{e^{t\De} u_0}_{\mathcal{E}_T} + \norm{B(u^{(n-1)}, u^{(n-1)})}_{\mathcal{E}_T}.
\]
We now bound $B(f,g)$ in $\mathcal{E}_T$ in terms of $f$ and $g$ in $\mathcal{F}_T$ and $\mathcal{E}_T$. 
We have by Lemma \ref{Wa-estimate} {and H\"older inequality \eqref{Holder} using $q_1 \ge 6$,} \EQS{\label{ineq:E3qintegral2}
\norm{B(f,g)}_{E^3_{q}}(t) 
&\lec \int_0^t  \frac1{(t-\tau)^{\frac12}} \norm{f\otimes g(\tau)}_{E^3_{q}}\, d\tau\\
&\le \int_0^t  \frac1{(t-\tau)^{\frac12}} \norm{f(\tau)}_{E^6_{q_1}} \norm{g(\tau)}_{E^6_{q_1}}\, d\tau\\
&\le \int_0^t  \frac1{(t-\tau)^{\frac12}} \tau^{-1/4}\norm{f}_{\mathcal{F}_T} \tau^{-1/4}\norm{g}_{\mathcal{F}_T}\, d\tau\\
&\lec \norm{f}_{\mathcal{F}_T} \norm{g}_{\mathcal{F}_T}.
}
{By $\mathcal{E}_T \subset \mathcal{F}_T$,}
we have
\[
\norm{B(f,g)}_{E^3_{q}}(t) \lec \norm{f}_{\mathcal{E}_T} \norm{g}_{\mathcal{F}_T} \wedge \norm{g}_{\mathcal{E}_T} \norm{f}_{\mathcal{F}_T}.
\]

Also by Lemma \ref{Wa-estimate} {and H\"older inequality \eqref{Holder},}
\begin{align}
\nonumber
\| B(f,g)\|_{E^\I_{q_2}}(t) &\lec\int_0^t \frac 1 {(t-\tau)^{\frac 34 }} \| f\otimes g\|_{E^6_{q_1}}(\tau) \,d\tau
\\&\lec \int_0^t\frac 1 {(t-\tau)^{\frac 34 }\tau^{3/4}} 
 ( \tau^{1/2}	\|f\|_{E^\I_{q_2}}\tau^{1/4}\|g\|_{E^6_{q_1}} \wedge \tau^{1/2}\|g\|_{E^\I_{q_2}}\tau^{1/4}\|f\|_{E^6_{q_1}} )\,d\tau
\nonumber
\\&\lec t^{-\frac 1 2}(\norm{f}_{\mathcal{E}_T} \norm{g}_{\mathcal{F}_T} \wedge \norm{g}_{\mathcal{E}_T} \norm{f}_{\mathcal{F}_T}).
\label{ineq:E3qintegral3}
\end{align}

Based on the above estimates we conclude
\EQ{\label{eq3.38}
\| B(f,g)\|_{\mathcal E_T} \lec \norm{f}_{\mathcal{E}_T} \norm{g}_{\mathcal{F}_T} \wedge \norm{g}_{\mathcal{E}_T} \norm{f}_{\mathcal{F}_T}.
}

We can now conclude that $\{u^{(n)}\}$ is Cauchy in $\mathcal{E}_T$ by the calculation preceding and including \eqref{ineq:difference.picard.iterates}. {However, the smallness of the constant is now provided by \eqref{small.caloric}-\eqref{unFTbound}, not by $\norm{u_0}_{E^3_q}$.}

We now show continuity. For small data, we can try to inherit continuity from Theorem \ref{thrm:critical}. But we will provide a proof valid for general data.
We first address convergence to the initial data.
By Lemma \ref{lemma:heat.eq.convergence.data} we have
\EQ{\label{limit:zero.heat2}
\lim_{T'\to 0^+} \sup_{0<t<T'}t^{\frac 1 4} \| e^{t\Delta}u_0\|_{E^6_{q_1}} = \lim_{T'\to 0} \| e^{t\Delta}u_0\|_{{\mathcal F}_{T'}} = 0,
}
whenever $u_0\in E^3_q$. By our estimates in the class $ {\mathcal F}_{T'}$ where we are taking $T'\leq T$, we have 
\[
\| u^{(n)}\|_{{\mathcal F}_{T'}} \leq \|e^{t\Delta} u_0\|_{ {\mathcal F}_{T'}}+ \| B(u^{(n-1)},u^{(n-1)})\|_{ {\mathcal F}_{T'}}\lec \|e^{t\Delta} u_0\|_{ {\mathcal F}_{T'}}+\| u^{(n-1)}\|_{ {\mathcal F}_{T'}}^2.
\] 
From this and by induction, for any $n$ we have 
\EQN{
\lim_{T'\to 0^+} \| u^{(n)}\|_{{\mathcal F}_{T'}} = 0.
}
The limit \eqref{limit:zero.heat2}, convergence of the Picard iterates in $ {\mathcal F}_T$ and the above inequality imply that, by taking $T'$ small, we can make $\sup_{0<t<T'}t^{\frac 1 4} \| u(t)\|_{E^6_{q_1}}$ small. To elaborate,  we have 
\EQ{
\| u\|_{ {\mathcal F}_{T'}} \leq \| u - u^{(n)}\|_{ {\mathcal F}_{T'}}  + \| u^{(n)}\|_{ {\mathcal F}_{T'}}.
}
We may choose $n$ large so that the first term is small and then make the second term small by taking $T'$ small.
Hence,
\EQ{\label{limit:zero.ns2}
\lim_{T'\to 0^+} { \| u\|_{{\mathcal F}_{T'}}}= 0.
}
Using \eqref{ineq:E3qintegral2}, this implies  
\EQN{
\lim_{T'\to 0^+} \sup_{0<t<T'} \| B(u,u)\|_{E^3_q} (t)=0.
}
This and Lemma \ref{lemma:heat.eq.convergence.data} imply
\[
\lim_{t\to 0} \| u-u_0\|_{E^3_q}=0.
\]

We now prove continuity at positive times using the same argument as in the proof of Theorem \ref{thrm:critical}.  Let $t_1>0$ be fixed. We will send $t\to t_1$.  Note that by Lemma \ref{lemma:heat.eq.convergence.data} we have $e^{t\Delta}u_0 -e^{t_1\Delta}u_0\to 0$ in $E^3_q$ as $t\to t_1$. We therefore only need to show $B(u,u)(t)\to B(u,u)(t_1)$. Take $\rho$ slightly less than $1$ so that $\rho t_1<t$  and write
\EQN{
B(u,u)(t)- B(u,u)(t_1) &= \int_{\rho t_1}^t e^{(t-\tau)\Delta} \mathbb P\nb \cdot F\,d\tau 
- \int_{\rho t_1}^{t_1} e^{(t_1-\tau)\Delta} \mathbb P\nb \cdot F\,d\tau 
\\&+\int_{0}^{\rho t_1} \big(e^{(t-\rho t_1)\Delta}-e^{(t_1-\rho t_1)\Delta} \big)e^{(\rho t_1-\tau)\Delta} \mathbb P\nb \cdot F\,d\tau 
}
where $F=u\otimes u(\tau)$. For the first and second terms we have by the sequence of inequalities in \eqref{ineq:E3qintegral2} that,
\EQN{
\int_{\rho t_1}^t  \|  e^{(t-\tau)\Delta} \mathbb P\nb\cdot  F\|_{E^3_q}\,d\tau  &\lesssim  \int_{\rho t_1}^{t} \frac 1 {(t-\tau)^{\frac 1 2}\tau^\frac 1 2} \| u\|_{\mathcal F_t}^2  \,d\tau
 \lesssim \frac {(t-\rho t_1)^\frac 1 2}	{(\rho t_1)^\frac 1 2}	\|u\|_{\mathcal F_T}^2,
}
and 
\EQN{
\int_{\rho t_1}^{t_1} \|e^{(t_1-\tau)\Delta} \mathbb P\nb \cdot F\|_{E^3_q}\,d\tau  &\lesssim \int_{\rho t_1}^{{t_1}} \frac 1 {(t-\tau)^{\frac 1 2}\tau^\frac 1 2} \| u\|_{\mathcal F_t}^2  \,d\tau
 \lesssim \frac {(t_1-\rho t_1)^\frac 1 2}	{(\rho t_1)^\frac 1 2}	\|u\|_{\mathcal F_T}^2,
}
both of which can be made arbitrarily small by taking $\rho t_1$ close to $t_1$ and $t$ close to $t_1$.

For the third term we note that by Lemma \ref{lemma:heat.eq.convergence.data}, for each $0<\tau<\rho t_1$, we have
\[
\|  \big(e^{(t-\rho t_1)\Delta}-e^{(t_1-\rho t_1)\Delta} \big)e^{(\rho t_1-\tau)\Delta} \mathbb P\nb \cdot F(\tau) \|_{E^3_q}\to 0 \text{ as }t\to t_1,
\]
which follows from the fact that $e^{(\rho t_1-\tau)\Delta} \mathbb P\nb \cdot F(\tau)\in  E^3_q$, which is a consequence of Lemma \ref{Wa-estimate}.
Additionally,  
\EQN{
& \|  \big(e^{(t-\rho t_1)\Delta}-e^{(t_1-\rho t_1)\Delta} \big)e^{(\rho t_1-\tau)\Delta} \mathbb P\nb \cdot F(\tau)	\|_{E^3_q}
\\&\lesssim  
 \bigg( \frac 1 {(t-\tau)^\frac 1 2 \tau^\frac 1 2} + \frac 1 {(t_1-\tau)^\frac 1 2 \tau^\frac 1 2}\bigg) \| u\|_{\mathcal F_T}^2 \in L^1(0,\rho t_1),
}
where integration in $L^1(0,\rho t_1)$ is with respect to $s$.
So, by Lebesgue's dominated convergence theorem, 
\EQ{
\int_0^{\rho t_1 } \|  \big(e^{(t-\rho t_1)\Delta}-e^{(t_1-\rho t_1)\Delta} \big)e^{(\rho t_1-\tau)\Delta} \mathbb P\nb \cdot F(\tau)	\|_{E^3_q}\,d\tau \to 0 \text{ as }t\to t_1.
} 
The above show the continuity of $u(t)$ at positive times.

\medskip

We now prove the spacetime integral bound \eqref{eq-thmIII-Ebound} for $p\in (3,9]$ and $\frac2s + \frac3p = 1$.  Note that we exclude $p=3$, i.e., $s=\infty$.
By imbedding \eqref{eq1.4}, we may assume $m<\infty$. (We do not take $m=q$ since we need $q<m$ for global existence).\
Denote the Banach space
\[
X_T = \mathcal E_T \cap E^{s,p}_{T,m}.
\]
For the linear term, by Lemmas \ref{Wa-estimate} and \ref{lemma:linear.heat.amalgam.spacetime-new}, 
\EQS{\label{beta0}
\norm{e^{t\De} u_0}_{X_T} 
&= \sup_{0<t<T} \norm{e^{t\De} u_0}_{E^3_q} + \sup_{0<t<T} t^{\frac12}\norm{e^{t\De} u_0}_{E^\infty_{q_2}} + \norm{e^{t\De} u_0}_{E^{s,p}_{T,m}}\\
&\le C_3(1+T^{\beta_0})\norm{u_0}_{E^3_q},
}
for any $\be_0 \in [0,\infty)$ and $\be_0 > \al_0=\frac 3{2m}-\frac3{2q}+\frac 1s$. 
Note that
\EQ{\label{smallspacetimeintegral}
\lim_{T \to 0_+ } \norm{e^{t\De} u_0}_{E^{s,p}_{T,m}} = 0. 
}
Here is a proof using $m <\infty$: Denote
\[
a_k(t) = \norm{e^{t\De} u_0}_{L^s(0,t; L^p(B_1(k))} , \quad k \in \ZZ^3.
\]
For any $\e>0$, since $\{a_k(T)\}_k \in \ell^m$, there is $N>0$ such that
$
\sum_{|k|>N} a_k(T)^m \le \e^m$.
Then, since $s<\infty$, there is $t\in (0,T)$ so that
\[
\sup_{|k|\le N} a_k(t) \le  \frac{\e}{N^{3/m}}.
\]
For any $\tau \in (0,t]$, we have
\EQN{
\sum_{k\in \ZZ} a_k(\tau)^m &\le \sum_{|k|\le N} a_k(t)^m + \sum_{|k|>N} a_k(T)^m
 \le C N^3 \bke{\frac \e{N^{3/m}}}^m + \e^m = C\e^m.
}
Hence $\norm{a_k(\tau)}_{\ell^m(k \in \ZZ^3)} \le C \e$ for all $\tau \le t$. This shows \eqref{smallspacetimeintegral}.

For the bilinear term, by Lemma \ref{lemma:duhamel.amalgam} with $\td s=s/2$, $\td p=p/2$, and $\td m = \max(1, m/2)$, so that
\EQ{
\si=0, \quad
\al=
\left\{
\begin{aligned}
\tfrac12 - \tfrac3{2m} - \tfrac1s ,\quad &\text{if}\quad 2 \le m \le \infty,\\[1mm]
-1 + \tfrac3{2m} - \tfrac1s,\quad &\text{if}\quad 1<m<2,
\end{aligned}
\right. \quad \al < 1-\tfrac1s,
}
we have
\EQN{
\norm{B(f,g)}_{E^{s,p}_{T,m}}
&\le C_4 (1+ T^{\be}) \norm{f\otimes g}_{E^{\frac{s}2,\frac p2}_{T,\td m}}
}
for any $\be \in [0,1-\frac 1s]$ and $\be > \al$. Note
\[
\norm{f\otimes g}_{E^{\frac s2,\frac p2}_{T,\td m}}
\le \norm{f}_{E^{s,p}_{T,2\td m}} \norm{g}_{E^{s,p}_{T,2\td m}}\le \norm{f}_{E^{s,p}_{T, m}} \norm{g}_{E^{s,p}_{T, m}}
\]
no matter $m\ge2$ or $1<m<2$.
We conclude, also using \eqref{eq3.38},
\EQS{\label{beta1}
\norm{B(f,g)}_{E^{s,p}_{T, m}} &\leq  C_4 (   1+ T^{\be}   ) \norm{f}_{E^{s,p}_{T, m}} \norm{g}_{E^{s,p}_{T, m}},
\\
\norm{B(f,g)}_{X_T} &\leq  2C_4 (   1+ T^{\be}   ) \norm{f}_{X_T} \norm{g}_{X_T}.
}

By \eqref{smallspacetimeintegral}, we can find $T_1 \in (0,T]$ so that
\[
\norm{e^{t\De} u_0}_{E^{s,p}_{T_1,m}}  \le \de= [4C_4 (   1+ T^{\be}   )]^{-1}.
\]
Then the Picard sequence $u^{(k)}$ satisfies $\norm{u^{(k)}}_{E^{s,p}_{T_1,m}}  \le 2\de$ for all $k \in \NN$, and we get
$\norm{u}_{E^{s,p}_{T_1,m}}  \le 2\de$.
Thus, $u$ satisfies the spacetime integral bound \eqref{eq-thmIII-Ebound}.

\medskip

We next show the global $E^{s,p}_{T,m}$-estimates when $u_0$ is sufficiently small in $E^3_q$.
We need to avoid $T$-dependence in the constants. In other words, we want to choose $\be_0=0$ in \eqref{beta0} and $\be=0$ in \eqref{beta1}.

To choose $\be_0=0$ in \eqref{beta0}, we first require $\frac2s + \frac3m \le \frac3q$ which makes $\alpha_0=\frac 3{2m}-\frac3{2q}+\frac 1s\leq 0$ in Lemma \ref{lemma:linear.heat.amalgam.spacetime-new}. If $\frac2s + \frac3m = \frac3q$ then $\alpha_0=0$ and we require additionally that $1<q<m<\infty$. Then we can take $\be_0=0$.

We now consider the conditions for $\be=0$ in \eqref{beta1}. When $m\ge2$, we assume further that $\frac2s + \frac3m \ge 1$ so that $\al\le 0$. We can take $\be=0$ when $\al=0$ since the condition $1<\td m < m< \infty$ in Lemma \ref{lemma:duhamel.amalgam} is met. Note $\td m=1$ when $m=2$, but then $\al<0$.

If $1<m<2$,  we assume further that $\frac3m<\frac2s + 2=3-\frac3p$. Then $\al<0$ and we can take $\be=0$.

Using \eqref{beta0} and \eqref{beta1} with $\be_0=\be=0$, we can prove the bound of $u$ in $E^{s,p}_{T=\infty,m}$ if $ \norm{u_0}_{E^p_q} < [32C_3 C_4 ]^{-1}$. 

Note that all conditions when $m \ge 2$, in particular the upper bound
 $\frac2s + \frac3m \ge 1$, are satisfied for $m=p$. Once we have shown $u \in E^{s,p}_{T=\infty,m}$ for one $m$, we have $u \in E^{s,p}_{T=\infty,\td m}$ for all $\td m \in [m,\infty]$. Hence the condition $\frac2s + \frac3m \ge 1$ can be removed.

We finally consider the $L^s_T E^p_m$ estimate of $u$. {For simplicity, we only consider $T=\infty$.} Fix $s\in [3,\infty)$ and $q\in[1,3]$, and let $p$ be given by $\frac3p+\frac 2s=1$. 
By Lemma \ref{lem-bilinear-LsEpm} with
$\td p = p/2$, $\td s=s/2$ and $\td m\ge1$ such that $\frac1{\td m} - \frac1m=\frac1{\td p} - \frac1p = \frac 1p$, and by H\"older inequality, we get
\EQN{
\norm{B(f,g)}_{L^s_TE^p_m} \lec \norm{f\otimes g}_{L^{\frac{s}2}_TE^{\frac{p}2}_{\td m}}
&\lec \norm{f}_{L^s_TE^p_m} \norm{g}_{L^s_TE^p_{p}},
\\
&\lec \norm{f}_{L^s_TE^p_m} \norm{g}_{L^s_TE^p_m},\ \text{ if }p\ge m.
}
The condition $\td m \ge 1$ is the same as $m \ge p'{=\frac{p}{p-1}}$. Hence we have
\EQ{\label{eq3.47}
\norm{B(f,g)}_{L^s_TE^p_m} \lec \norm{f}_{L^s_TE^p_m} \norm{g}_{L^s_TE^p_m},\ \text{ if }p'\le m\le p.
}

Denote the set of acceptable $m$ that we can prove $u \in L^s E^p_m$ as $\cM(s,q)$.
Since $L^s E^p_m \subset L^s E^p_{ m_2}$ if $m <  m_2$, ,  the region $\cM(s,q)$ of acceptable $m$, if non-empty, is an interval of the form
\EQ{\label{m-region}
\underline m < m \le \infty, \quad \text{or}\quad \underline m \le  m \le \infty,
}
with $\underline m =\underline m(s,q)\in [1,\infty]$. We expect $\underline m \ge q$ due to Example \ref{example2} for the heat equation.

Define $m_1$ by
\[
\frac 1{m_1} = \frac 1q - \frac 2{3s}.
\]
As $s<\infty$, $q\le 3$  and
$\frac 1{p} = \frac 13 - \frac 2{3s}$,
we have 
\EQ{\label{m1.range}
q < m_1 \le p.
}

By Lemma \ref{lemma:linear.heat.amalgam.spacetime-new} with $d=r=3$, we have
\EQ{\label{eq3.50}
\norm{e^{t\De} f}_{L^s_{T=\infty} E^p_{m} }  
\lesssim  \|f\|_{E^3_q},
}
with $T$-independent constants, if one of the following holds:
\EN{
\item [A.] $m_1 \le m\le s$, (and $m_1<m\le s$ if $q=1$), %

\item [E$_1$.] $ s<p=m$,  

\item [E$_2$.] $s<m$ and $\frac 1q \ge\frac{5}{3s}$.
}
If the linear estimate \eqref{eq3.50} holds, and if $p'\le m \le p$ so that the bilinear estimate \eqref{eq3.47} holds, then we can prove that the Picard sequence $u^{(n)}$ converges in $L^s E^p_m$, if $\norm{u_0}_{E^q}$ is sufficiently small.

Case A has nontrivial $m$ as soon as $m_1 \le s$, i.e., $\frac 1q \ge\frac{5}{3s}$. Note that we have strict inequality $m_1<s$ when $q=1$. 
By \eqref{m1.range} and $p'<2<s$,
the number 
$$m^*(s,q)=\max(p', m_1)$$ 
is in both $[p',p]$ and $[m_1,s]$.
Hence the Picard sequence $u^{(n)}$ converges in $L^s E^p_m$ for $m=m^*$, or for $m$ slightly larger than $m^*$ when $q=1$ and $3 \le s\le 4$  
{(When $q=1$ and $4<s<\infty$, we are in region I of Figure \ref{fig:m*}, and $m^*=p'>m_1$.)} 
By imbedding, $u \in L^s E^p_m$ for all $m > m^*$.

Case E$_1$ applies when $3\le s<5$ and we can take $m=p$. In particular, it covers region III: $1\le q \le 3 \le s <5$ and $\frac 1q <\frac{5}{3s}$. %
In region III we define $m^*(s,q)=p$.

Case E$_2$ requires $\frac 1q \ge\frac{5}{3s}$ and does not give smaller $m$ than Case A.

This completes the proof of $L^s E^p_m$-estimates, and concludes the proof of
Theorem \ref{thrm:critical2}.
\end{proof}

\section{Related remarks on weak solutions}\label{Sec4}

\subsection{Global weak solutions in $E^2_q$ for $1\leq q<2$}\label{sec:global-weak}
  
In \cite{BT4} weak solutions are constructed by the first and third authors for data in $E^2_q$ where $2\leq q<\I$. These solutions are viewed as bridging the gap between Leray weak solutions where the data is in $L^2=E^2_2$ and the time-global weak solutions of Lemari\'e-Rieusset \cite{LR2} where the data is in $E^2$. The motivation for \cite{BT4} was to identify scalings at which properties for Leray solutions, such as eventual regularity, and Lemari\'e-Rieusset solutions break down. It appears that the construction in \cite{BT4} can be extended to amalgam spaces $E^2_q$ where $1 \leq q< 2$.
In this subsection, we extend the range of the exponent $q$ in \cite[Theorem 1.3 (Eventual regularity), Theorem 1.4 (Explicit growth rate), Theorem 1.5 (Global existence)]{BT4} down to $q=1$.
The structure of the proofs is similar, but parts of the old proof break down for small $q${---for example, in the last display of the proof of \cite[Lemma 4.1]{BT4} (see \cite[p.2013]{BT4}), the first term $\| \tilde K(k)*(\alpha^2)\|_{q/2}$ is infinite for $q$ close to 1.   To address this and to simplify the overall argument, we use the fact that the time-scale in Lemma \ref{lem.A0qbound} can be pushed to large times by passing to large scales. This avoids the iterative extension of a time-local solution to a time-global solution in \cite{BT4}, which involves solving a perturbed system.  This idea only works for $q<6$ (see   Lemma \ref{lem.A0qbound}  for $1\leq q <2$ and \cite[Lemma 3.1]{BT4}  for $2\leq q<6$) and so cannot be used in place of the construction in \cite{BT4} for $6\leq q< \I$.
}

\subsubsection{Eventual regularity}

We recall the notation $u\in\mathcal N(u_0)$ in \cite[Definition 1.1]{BT4} meaning that $u$ is a local energy solution to the Navier-Stokes equations \eqref{NS} {in $\R^3$}, and define as in \cite[\S 2]{BT4} the quantities
\[
N^0_{q,R}(u_0) = \frac1R \bke{\sum_{k\in\ZZ^3}\bke{\int_{B_R(kR)} |u_0|^2\, dx}^{q/2}}^{2/q},\ 1\le q<\infty,
\]
\[
N^0_{\infty,R} = N^0_R := \sup_{x_0\in\R^3} \frac1R \int_{B_R(x_0)} |u_0|^2\, dx.
\]
By the concavity of $f(t) = t^p$, $0<p<1$,
\EQ{\label{concave}
f(a+b)\le f(a)+f(b),\quad\textstyle
f(\sum_{i=1}^n a_i ) \le \sum_{i=1}^n f(a_i ), \quad \text{if }a_i\ge0.
}

We prove the following lemma corresponding to \cite[Lemma 2.2]{BT4}.

\begin{lemma}\label{bt4-lem-2.2}
Let $1\le q<2$. If $u_0\in E^2_q$, then 
\[
\lim_{R\to\infty} N^0_R(u_0) = 0.
\]
\end{lemma}
\begin{proof}
Denote $a_i = \norm{u_0}_{L^2(B_1(i))}$, $i\in\ZZ^3$.
Fix $k=(k_1,k_2,k_3),\, k'=(k_1',k_2',k_3')\in \ZZ^3$ so that $|k-k'|\geq 2$. Without loss of generality $|k_1- k_1'|\geq 2$. 
Suppose $i$ satisfies 
$
|i-kR|<R,
$
then, because we have $|k_1 R - k_1'R |\geq 2R$, which implies $R\leq 2R -|i-kR|\leq |i-k'R| $, it follows that 
$
|i-k' R|\geq R.
$
So, sets of the form $ S_k= \{  i: |i-kR|<R \}$ form a  cover of $\ZZ^3$ in which $S_k$ only overlaps those $S_{k'}$ where $|k' -k|<2$.  Hence, there are $27$ sets that can overlap $S_k$. 
Using this fact and {concavity \eqref{concave}} we have
for $R\ge1$,
\EQS{\label{eq-3.17-BT4}
N^0_{q,R}(u_0) 
&\le \frac{C}R \bke{\sum_{k\in\ZZ^3} \bke{\sum_{|i-kR|<R} a_i^2}^{q/2} }^{2/q}
\le   \frac{C}R \bke{\sum_{k\in\ZZ^3} {\sum_{|i-kR|<R} a_i^q} }^{2/q}
\\
&\le   \frac{C}R \bke{27 \sum_{i\in\ZZ^3} a_i^q  }^{2/q}
\le \frac{C}R \norm{u_0}_{E^2_q}^{{2}}.
}
The lemma follows as $N^0_R(u_0) \le N^0_{q,R}(u_0)$.
\end{proof}

Using Lemma \ref{bt4-lem-2.2}, the following eventual regularity result,  which corresponds to \cite[Theorem 1.3]{BT4}, is a direct consequence of {\cite[Theorem 1.2 (1)]{BT8} (also} \cite[Theorem 2.1]{BT4}).

\begin{theorem}[Eventual regularity in $E^2_q$]\label{bt4-thm-1.3}%
Assume $u_0\in E^2_q$ where $1\le q<2$, is divergence free and $u\in\mathcal N(u_0)$. 
Then $u$ has eventual regularity, i.e., there is $t_1<\infty$ such that $u$ is regular at $(x,t)$ whenever $t\ge t_1$, and 
\[
\norm{u(\cdot,t)}_{L^\infty} \lec t^{1/2},
\]
for sufficiently large $t$.
\end{theorem}
{In fact, Theorem \ref{bt4-thm-1.3} also follows from \cite[Theorem 1.3]{BT4} without using Lemma \ref{bt4-lem-2.2} since $u_0\in E^2_q \subset L^2$ for $q<2$. However, Lemma \ref{bt4-lem-2.2} will also be useful in the proof of Theorem \ref{thm-1.4BT4} and Lemma \ref{BT8-lem3.3}.}

\subsubsection{A priori bounds and explicit growth rate}

Now we extend the main tool involved in proving the global existence, an \textit{a priori} estimate \cite[Lemma 3.1]{BT4} for $2\leq q<\I$, to the range $1 \leq q< 2$.    Note that, aside from changing the range of $q$, the only difference here compared to \cite{BT4} is that we have restricted the permissible values of $\sigma$ in \eqref{th2.2-2} and changed the definitions  of $\la_0$ and $\la_R$. {Recall the $\ell^q$ local energy space ${\bf LE}_q(0,T)$ is defined by the norm \eqref{ETq.def}.}

\begin{lemma}\label{lem.A0qbound}
Assume $u_0\in E^2_q$ for some $1\leq q<  2$ 
 is divergence free and that  $u\in\mathcal N(u_0)$ satisfies, for some  $T_2>0$, 
 \begin{equation}\label{th2.2-0}
{\norm{u}_{{\bf LE}_q(0,T_1)}<\infty,
\quad \text{for all }T_1 \in (0,T_2).}
\end{equation}
Then there are positive constants $C_1$ and $\la_0<1$, both independent of $q$ and $R$ such that, for all $R>0$ with $\la_R R^2\le T_2$,
\begin{equation}\label{th2.2-1}
\bigg\|\esssup_{0\leq t \leq \la_R R^2} \int_{B_R(x_0R) }\frac {|u|^2}2  \,dx+ \int_0^{\la_R R^2}\int_{B_R(x_0R) } |\nabla u|^2\,dx\,dt \bigg\| _{\ell^{q/2}(x_0\in \ZZ^3)}
\leq  C_1 A_{0,q}(R),
\end{equation}
where 
\[
A_{0,q}(R)  = R N^0_{q,R}  =  \bigg\|  \int_{B_R(x_0R)  }  |u_0(x)|^2 \,dx   \bigg\|_{\ell^{q/2}(x_0\in \ZZ^3)} ,
\]
and 
\[
\la_R =
\min \bigg\{\la_0,  \la_0 R^2, \frac{\la_0 R^2}{A_{0,q}(R)^{2}}\bigg\}	. 
\]
Furthermore, for all $R>0$ and $\sigma \geq \frac 3 5$,   
\begin{equation}\label{th2.2-2}
\bigg\|\int_0^{\la_R R^2}\!\!\int_{B_R(x_0R) } 
|u|^{\frac {10}3} +|\pi-c_{Rx_0,R}(t)|^{\frac53}\,dx\,dt \bigg\|_{\ell^{{\sigma}}(x_0 \in \ZZ^3)}
\le  C A_{0,q}(R)^{\frac 53},
\end{equation}
where $c_{Rx_0,R}$ is a function of time given in the definition of local energy solutions in \cite{BT4}.
 \end{lemma}

The bound \eqref{th2.2-1} is stronger than the \textit{a priori} bound for Leray's weak solutions, 
as 
\EQN{&\sup_{0\leq t\leq \la_R R^2}
\int_{\R^3}\frac  {|u|^2} 2\,dx + \int_0^{\la_R R^2}\int_{\R^3} |\nabla u|^2\,dx\,dt
\\&\lesssim  \bigg\|  \esssup_{0\leq t \leq \la_R R^2} \int_{B_R(x_0R) }\frac {|u|^2}2  \,dx\bigg\|_{\ell^1} +  \bigg\|\int_0^{\la_R R^2}\int_{B_R(x_0R) } |\nabla u|^2\,dx\,dt\bigg\|_{\ell^1} 
\\&\leq \bigg\|\esssup_{0\leq t \leq \la_R R^2} \int_{B_R(x_0R) }\frac {|u|^2}2  \,dx+ \int_0^{\la_R R^2}\int_{B_R(x_0R) } |\nabla u|^2\,dx\,dt \bigg\| _{\ell^{q/2}(x_0\in \ZZ^3)}.
}
{The exponent $\si \ge \frac 35$ in \eqref{th2.2-2} is larger than $\frac{3q}{10}$ in \cite[(3.6)]{BT4} when $q<2$.}

\begin{proof}[Proof of Lemma \ref{lem.A0qbound}]

Let $\phi_0\in C_c^\I(\R^3)$ be radial, non-increasing, identically $1$ on $B_1(0)$, supported on $B_2(0)$, and satisfy $|\nabla \phi_0(x)|\lesssim  1$ and $|\nabla \phi_0^{1/2}(x)|\lesssim 1$.  Let $R>0$ be as in the statement of the lemma. Let $\phi(x)=\phi_0(x/R)$.  Let $0 < \la \le 1$.

We decompose the pressure according to \cite[Definition 1.1]{BT4}, namely for a fixed ball {$B_{2R}(\kappa)$ where $\kappa\in \R^3$}, we write
\EQN{\label{eq.pressure}
\pi(x,t)&=-\Delta^{-1}\div \div [(u\otimes u )\chi_{4R} ({x}-\ka)]
\\&\quad - \int_{\R^3} (K(x-y) - K(\ka -y)) (u\otimes u)(y,t)(1-\chi_{4R}(y-\ka))\,dy+c_{{\kappa},R}(t)
\\&=\pi_1(x,t)+\pi_2(x,t) +c_{{\kappa},R}(t).
}

For $\kappa\in R\ZZ^3$, let \begin{align*}
e_{R,\la}(\ka):=& \esssup_{0\leq t\leq \la R^2}\int |u(t)|^2 \phi(x-\ka)   \,dx +\int_0^{\la R^2}\!\! \int |\nb u|^2 \phi(x-\ka)  \,dx\,dt,
\end{align*}
and
\begin{align*}
E_{R,q,\la}&:= \bigg\|   \esssup_{0\leq t\leq \la R^2}\int |u(t)|^2 \phi(x-Rk)   \,dx 
+\int_0^{\la R^2}\!\int |\nabla u|^2 \phi(x-Rk)  \,dx\,ds \bigg\|_{\ell^{q/2}(k\in \ZZ^3)}^{q/2}.
\end{align*}
The idea of \cite[Proof of Lemma 3.1]{BT4} is to first  bound each $e_{R,\la}(\kappa)$ using the local energy inequality and then bound $E_{R,q,\la}$ by summing the bounds for $e_{R,\la}(\kappa)$. Following \cite[Proof of Lemma 3.1]{BT4} exactly, we obtain  
\EQS{\label{eq3.11}
e_{R,\la}(\ka)&\leq  \int |u_0|^2\phi(x-\ka)  \,dx+C \la  \sum_{\ka'\in R\ZZ^3; |\ka'-\ka|\leq 2R} e_{R,\la}(\ka')
\\&+ C    \frac {\la^{1/4}} {R^{1/2}}  \sum_{\ka'\in R\ZZ^3; |\ka'-\ka|\leq 10R} (e_{R,\la}(\ka'))^{3/2}
+\int_0^{\la R^2} \int 2 \pi_2 u\cdot \nb \phi( x-\ka) \,dx\,ds,
}
provided $\la \leq 1$. Recall $\pi_2$ in {$B_{2R}(\kappa)$} is bounded by $R^{-3} \overline K * e_{R,\la}(\kappa)$, and the convolution  $\overline K*e_{R,\la}$ is understood over $R\ZZ^3$, and, for $x \in R\ZZ^3$,
\[
\overline K(x) = \frac 1 {|x/R|^4}, \quad \text{if } |x|>4R; \quad 
\overline K(x) =0 \quad \text{otherwise}. 
\] 
{Using the fact that $\pi_2$ in $B_R(k)$ is bounded by $R^{-3} \overline K * e_{R,\la}(k)$,
the last integral in \eqref{eq3.11}  is bounded  as
\EQS{\label{6.2.22.a}
\int_0^{\la R^2} \int 2 \pi_2 u\cdot \nb \phi( x-\ka) \,dx\,ds &\lesssim 
\frac 1 {R^4} {\la R^{\frac 7 2}} \| u\|_{L^\I(0,\la R^2; L^2(B_{2R}(\ka)))}  |\overline K * e_{R,\la}(\ka)  |
\\&\lesssim 
{\frac {\la} {R^{1/2}}} \bigg( \| u\|_{L^\I(0,\la R^2; L^2(B_{2R}(\ka)))}^2+  |\overline K * e_{R,\la}(\ka)  |^2 \bigg)
\\&\lesssim {\frac {\la} {R^{1/2}}} \bigg( \sum_{\ka'\in R\ZZ^3; |\ka'-\ka|\leq 2R} e_{R,\la}(\ka') +  |\overline K * e_{R,\la}(\ka)  |^2 \bigg).
}} In the above  {estimates  \eqref{eq3.11} and \eqref{6.2.22.a}}, the constants do not depend on $q$. Note that the exponent of $|\overline K * e_{R,\la}(\ka)  |^2$ is 2, instead of 3/2 in \cite[(3.12)]{BT4}, so that when we take its $\ell^{q/2}$-norm for {$1\le q $} we can still apply Young's convolution inequality.

Next, raise both sides of the inequality \eqref{eq3.11} to the power $q/2$ and sum over $\ka\in R\ZZ^3$.  The left hand side becomes $E_{R,q,\la}$.   For the non-convolution terms on the right hand side of \eqref{eq3.11} where the last term is replaced by the bound in \eqref{6.2.22.a}, we have
\[
 \sum_{\ka\in R\ZZ^3}\bigg(   \int |u_0|^2\phi(x-\ka)  \,dx \bigg)^{q/2} \le C^q A_{0,q}(R)^{q/2},
\] 
\[ \sum_{\ka\in R\ZZ^3}\bigg(  C\big(\la +{\frac {\la} {R^{1/2}}} \big)  \sum_{\ka'\in R\ZZ^3; |\ka'-\ka|\leq 2R} e_{R,\la}(\ka')   \bigg)^{q/2} 
\leq C^q \bigg(\la+   {\frac {\la} {R^{1/2}}}\bigg)^{q/2}    E_{R,q,\la} ,
\]
and
\begin{align*}
&\sum_{\ka\in R\ZZ^3}
\bigg( C  \frac {\la^{1/4}} {R^{1/2}} \sum_{\ka\in R \ZZ^3; |\ka'-\ka|\leq 10R} (e_{R,\la}(\ka'))^{3/2}\bigg)^{q/2}  
\\
&\quad \leq C^q
\bigg( \frac {\la} {R^{2}}  \bigg)^{\frac q8} \sum_{\ka\in R\ZZ^3}
{\sum_{\ka\in R \ZZ^3; |\ka'-\ka|\leq 10R}  e_{R,\la}(\ka')^{\frac{3q}4}}
\leq C^q
\bigg(  \frac {\la} {R^{2}}  \bigg)^{\frac q8} \sum_{\ka\in R\ZZ^3} e_{R,\la}(\ka)^{\frac{3q}4}.
\end{align*}
{Above we have used
 $(\sum_{i=1}^n a_i )^p \le \sum_{i=1}^n a_i ^p$ for $a_i \ge 0$ since $0<q/2<1$, 
see \eqref{concave}.}
Also note,
\[
\sum_{\ka\in R\ZZ^3} e_{R,\la}(\ka)^{\frac{3q}4} = \norm{e_{R,\la}}_{\ell^{\frac{3q}4}}^{\frac{3q}4} 
\le  \norm{e_{R,\la}}_{\ell^{\frac{q}2}}^{\frac{3q}4} = E_{R,q,\la}^{\frac32}.
\]

For the  convolution term we use Young's convolution inequality to obtain 
\EQN{
&\sum_{\ka\in R\ZZ^3} \bigg({\frac {C\la} {R^{1/2}}}   |\overline K * e_{R,\la}(\ka)  |^2 \bigg)^{q/2}
= \bke{\frac {C\la} {R^{1/2}}}^{\frac q2} \sum_{\ka\in R\ZZ^3}  |\overline K * e_{R,\la}(\ka)  |^q
\\&\quad \le \bke{\frac {C\la} {R^{1/2}}}^{\frac q2}  	\|\overline K\|_{\ell^1}	 ^{q}	 	  	\|e_{R,\la}\|_{\ell^q}	 ^{q}
\le \bke{\frac {C\la} {R^{1/2}}}^{\frac q2} \|e_{R,\la}\|_{\ell^{q/2}} ^{q}				
\le  \bigg( \frac {C\la} {R^{1/2}} \bigg) ^{\frac q2}	{E_{R,q,\la}^2},
} 
where we used the fact that $ \| \overline K \|_{\ell^1(R\ZZ^3)}$ is bounded independently of $R$.

We conclude that, for some constant $C_2\ge1$ independent of $q,R$, we have
\EQS{\label{Eq.est}
E_{R,q,\la} &\leq C_2^q A_{0,q}(R)^{q/2} + C_2^{q} \bke{\la+ \frac {\la} {R^{1/2}} }^{q/2}  E_{R,q,\la}+  C_2^q  \bigg( \frac {\la} {R^{2}}  \bigg)^{\frac q8}  E_{R,q,\la}^{3/2}
\\
&\quad +C_2^q \bigg( \frac {\la} {R^{1/2}} \bigg) ^{\frac q2}{E_{R,q,\la}^2}
}

The right side is finite for $\la<R^{-2}T_2$ by assumption \eqref{th2.2-0}. The same argument in \cite[page 2005]{BT4} shows that $E_{R,q,\la}$ is continuous in $\la$. Then, from \eqref{Eq.est} and a continuity argument, we conclude,
\[
E_{R,q,\la} \leq 2E_0, \quad E_0 = C_2^q (A_{0,q}(R))^{q/2} ,
\]
provided
\[
C_2^q\bke{\la+ \frac {\la} {R^{1/2}} }^{q/2} +  C_2^q  \bigg( \frac {\la} {R^{2}}  \bigg)^{\frac q8}  (2E_0)^{1/2}
 +C_2^q \bigg( \frac {\la} {R^{1/2}} \bigg) ^{\frac q2}(2E_0) \le \frac 12.
\]
This is achieved if  
\[
\la \le  \la_R:= \min \bigg\{\la_0,  \la_0 R^2, \frac{\la_0 R^2}{A_{0,q}(R)^{2}}\bigg\},
\quad
\la_0 =  c C_2^{-12}
\]
{for some positive $c<1$.}
This shows the first estimate  \eqref{th2.2-1} of Lemma \ref{lem.A0qbound} with $C_1 = CC_2^2$. Note that the constants $C_2,\la_0$ and $C_1$ do not depend on $q$ and $R$.%

\medskip

We now show \eqref{th2.2-2}. 
Denote $N= \sup_{0\leq t\leq \la R^2}\int_{B_{R}} |u(t)|^2    \,dx +2\int_0^{\la R^2} \int_{B_{R}} |\nb u|^2 \,dx\,dt$ with $\la=\la_R$. We have by the Gagliardo-Nirenberg inequality
\EQN{\label{u103.est}
\int_0^{\la R^2} \int_{B_{R}} |u|^{\frac {10}3} dx \,dt 
&\lesssim N^{2/3} \int_0^{\la R^2}\bke{\int_{B_{R}} |\nb u|^2} dt  + R^{-2}  N^{5/3} \la R^2
\\
&\lesssim N^{5/3}  + \la N^{5/3}  \lesssim    N^{5/3},
}
using $\la \le 1$.
For $k \in R\ZZ^3$ and $Q(k) = B_R(k) \times (0,\la_R R^2)$, by the preceding estimate with $B_R$ replaced by $B_R(k)$,
we have ${N(B_R(k))} \le e_{R,\la}(k)$ and hence, for $\sigma>0$ satisfying $\frac{5\sigma} 3\ge \frac q 2$ (which is implied if $\sigma\geq {3/5}$), we have  
\[
\sum_{k \in R\ZZ^3} \bke{\int_{Q(k)} |u|^{\frac {10}3} dx\,dt}^{\sigma} 
\le C \sum_{k \in R\ZZ^3}
e_{R,\la}(k)^{\frac {5\sigma}3}  \le C \bigg(\sum_{k \in R\ZZ^3}
e_{R,\la}(k)^{\frac q2}\bigg)^{\frac 2 q \cdot \frac {5\sigma}3} \le C   E_0^{\frac 2 q \cdot\frac {5\sigma}3}. 
\]
By Calderon-Zygmund estimates,
\[
\sum_{k \in R\ZZ^3} \bke{\int_{Q(k)} |\pi_1|^{\frac 53} dx\,dt}^{\sigma} 
\le C\sum_{k \in R\ZZ^3}\sum_{k' \in R\ZZ^3; |k-k'|<10R}\bke{\int_{Q(k')} |u|^{\frac {10}3} dx\,dt}^{\sigma} 
\le C  E_0^{\frac 2 q \cdot\frac {5\sigma}3}.
\]

We now address the second component of the pressure. Recall $\pi_2$ in $B_R(k)$ is bounded by $R^{-3} \overline K * e_{R,\la}(k)$. 
Hence,
\[
\int_{Q(k)} |\pi_2|^{\frac53} dx\,dt \le C\la  ( \overline K * e_{R,\la}(k))^{\frac 53}.
\]
Thus,  if
$\frac {5\sigma}3 \geq 1$,  which is implied by our assumptions, then
we have
\EQN{
\sum_{k \in R\ZZ^3} \bke{\int_{Q(k)} |\pi_2|^{\frac 53} dx\,dt}^{\sigma} 
&\le  C\la^{\sigma}  \sum_{k \in R\ZZ^3}( \overline K * e_{R,\la}(k))^{\frac {5\sigma} 3 }
\\
&\le   C\la^{\sigma}  \norm{ \overline K }_{ \ell^1}^{\frac  {5\sigma}3} \sum_{k \in R\ZZ^3}( e_{R,\la}(k))^{\frac {5\sigma} 3}
\le   C   E_0^{\frac 2 q \cdot\frac {5\sigma}3}.
}
We conclude that
\[
\bigg\|  \int_{Q(k)} |u|^{\frac {10}3} +|\pi_1+\pi_2|^{\frac53}\,dx\,dt \bigg\|_{\ell^{\sigma}(k \in R\ZZ^3)}
\le C  E_0^{\frac 2 q \cdot\frac {5}{3}}  
=  C  A_{0,q}(R)^{\frac {5}3}.
\]
This shows \eqref{th2.2-2}.
\end{proof}

{By \eqref{eq-3.17-BT4}, we have $A_{0,q}(R) \lec \norm{u_0}_{E^2_q}^2$.
By Lemma \ref{lem.A0qbound},} we immediately have the following theorem of explicit growth rate that corresponds to \cite[Theorem 1.4]{BT4}.
\begin{theorem}[Explicit growth rate in $E^2_q$]\label{thm-1.4BT4}
Assume $u_0\in E^2_q$ where $1\le q<2$, is divergence free and $u\in\mathcal N(u_0)$ satisfies, for some $T_2>0$,
\[
\norm{u}_{{\bf LE}_q(0,T_1)}<\infty,\quad \forall T_1\in(0,T_2).
\]
Then, for any $R\ge 1$, with $T=\min\bke{\la_1(1+\norm{u_0}_{E^2_q})^{-4}R^2, T_2}$, we have
\EQ{\label{eq-thm1.4BT4}
\bigg\|\esssup_{0\leq t \leq T} \int_{B_R(Rk) } |u|^2  \,dx+ \int_0^T\int_{B_R(Rk) } |\nabla u|^2\,dx\,dt \bigg\| _{\ell^{q/2}(k\in \ZZ^3)}
\leq  C \norm{u_0}_{E^2_q}^2,
}
for positive constants $\la_1$ and $C$ independent of $u_0$ and $R$.
In particular, if $T_2=\infty$ then $T\to\infty$ as $R\to\infty$.
\end{theorem}

\subsubsection{Global existence}

{To prove global existence we} follow the approach of \cite[Theorem 1.5]{BT8}, {(modified from \cite[\S3]{KwTs}), via the localized and regularized Navier-Stokes equations}
\EQS{\label{eq-regularizedNS}
&\pd_t u^\epsilon - \De u^\epsilon + (\mathcal J_\epsilon(u^\epsilon)\cdot\nb)(u^\epsilon \Phi_\epsilon) + \nb \pi^\epsilon = 0,\\
&\div u^\epsilon = 0,
}
where $\mathcal J_\epsilon(f) = \eta_\epsilon * f$ for a spatial mollifier {$\eta_\epsilon(x) = \epsilon^{-3} \eta(x/\epsilon)$} and $\Phi_\epsilon(x) = \Phi(\epsilon x)$ for a fixed radially decreasing cutoff function $\Phi$ satisfying $\Phi=1$ on $B_1(0)$ and $\supp(\Phi)\subset B_{3/2}(0)$.

Next we construct a mild solution in ${{\bf LE}_q(0,T)}$, {defined in \eqref{ETq.def},} of the regularized Navier-Stokes equations \eqref{eq-regularizedNS}.
The following lemma corresponds to \cite[Lemma 3.3]{KwTs}.

\begin{lemma}\label{KwTs-lem3.3}
Let $q\ge1$. For each $0<\epsilon<1$ and $u_0$ with $\div u_0=0$ and $\norm{u_0}_{E^2_q}\le B$, if $0<T<\min(1,c\epsilon^3B^{-2})$, we can find a unique solution $u=u^\epsilon$ to the integral form of \eqref{eq-regularizedNS}
\EQ{\label{eq3.5-KwTs}
u(t) = e^{t\De} u_0 - \int_0^t e^{(t-\tau)\De} \mathbb P\nb\cdot (\mathcal J_\epsilon(u)\otimes u\Phi_\epsilon)(\tau)\, d\tau
}
satisfying
\[
\norm{u}_{{\bf LE}_q(0,T)} \le 2C_0 B,
\]
where $c>0$ and $C_0>1$ are absolute constants.
\end{lemma}

\begin{proof}
Let $\Psi(u)$ be the map defined by the right side of \eqref{eq3.5-KwTs} for $u\in{{\bf LE}_q(0,T)}$.
By Lemma \ref{KwTs-lem2.4} and $T\le 1$,%
\EQN{
\norm{\Psi(u)}_{{\bf LE}_q(0,T)} 
&\lec \norm{u_0}_{E^2_q} + \norm{\mathcal J_\epsilon(u)\otimes u\Phi_\epsilon}_{E^{2,2}_{T,q}}\\
&\lec \norm{u_0}_{E^2_q} + \norm{\mathcal J_\epsilon(u)}_{L^\infty(0,T;L^\infty(\R^3))} \norm{u}_{E^{2,2}_{T,q}}\\
&\lec \norm{u_0}_{E^2_q} + \epsilon^{-\frac32} \sqrt T \norm{u}_{E^{\infty,2}_{T,q}}^2.
}
Thus
\[
\norm{\Psi(u)}_{{\bf LE}_q(0,T)} \le C_0 \norm{u_0}_{E^2_q} + C_1\epsilon^{-\frac32} \sqrt T \norm{u}_{{\bf LE}_q(0,T)}^2,
\]
for some constants $C_0, C_1>0$.
Similarly, for $u,v\in {{\bf LE}_q(0,T)}$,
\[
\norm{\Psi(u) - \Psi(v)}_{{\bf LE}_q(0,T)}
\le C_1\epsilon^{-\frac32} \sqrt T \bke{\norm{u}_{{\bf LE}_q(0,T)} + \norm{v}_{{\bf LE}_q(0,T)}} \norm{u - v}_{{\bf LE}_q(0,T)}.
\]
By the Picard contraction theorem, if $T$ satisfies
\[
T<\frac{\epsilon^3}{64(C_0C_1B)^2} = c\epsilon^3B^{-2},
\]
then we can find a unique fixed point $u\in{{\bf LE}_q(0,T)}$ of $u=\Psi(u)$, i.e., \eqref{eq3.5-KwTs}, satisfying 
$\norm{u}_{{\bf LE}_q(0,T)} \le 2C_0B$.
\end{proof}

The following lemma corresponds to \cite[Lemma 3.4]{KwTs}.

\begin{lemma}\label{KwTs-lem3.4}
Let $u_0\in E^2_q$, $q\ge1$, with $\div u_0=0$. For each $\epsilon\in(0,1)$, we can find $u^\epsilon$ in ${{\bf LE}_q(0,T)}$ and $\pi^\epsilon$ in $L^\infty(0,T;L^2(\R^3))$ for some positive $T=T(\epsilon,\norm{u_0}_{E^2_q})$ which solve the regularized Navier-Stokes equations \eqref{eq-regularizedNS} in the sense of distributions, and {$u^\epsilon(t) \to u_0$ in $L^2(E)$ as $t\to 0^+$} 
for any compact subset $E$ of $\R^3$.
\end{lemma}

\begin{proof}
The proof is nearly identical to \cite[Proof of Lemma 3.4]{KwTs} except 
that Lemma \ref{KwTs-lem3.3} ($E^2_q$-version of \cite[Lemma 3.3]{KwTs}) is applied
and that the first displayed estimate {in \cite[Proof of Lemma 3.4]{KwTs}} is adjusted as 
\EQN{
\norm{u^\epsilon - e^{t\De}u_0}_{E^{\infty,2}_{t',q} } 
&= \norm{\int_0^t e^{(t-\tau)\De}\mathbb P\nb\cdot(\mathcal J_\epsilon(u)\otimes u\Phi_\epsilon)(\tau)\, d\tau}_{E^{\infty,2}_{t',q} } \\
&\lec \norm{\mathcal J_\epsilon(u)\otimes u\Phi_\epsilon}_{E^{{2,2}}_{t',q} }
\lec \epsilon^{-\frac32} \sqrt {t'}\norm{u}_{E^{\infty,2}_{t',q}}^2,
}
{where we have used Lemma \ref{KwTs-lem2.4} and are assuming $t'\le T\le 1$.}
\end{proof}

We next show global existence for the regularized system \eqref{eq-regularizedNS}.
The following lemma corresponds to \cite[Lemma 3.3]{BT8}. Thanks to Lemma \ref{bt4-lem-2.2}, {the decay condition \cite[(1.12)]{BT8} on initial data is implied by $u_0\in E^2_q$,} $q<2$. 

\begin{lemma}\label{BT8-lem3.3}
Assume $u_0\in E^2_q$, $1\le q<2$, is divergence free, and fix $\epsilon\in(0,1)$. Then, there exists a global solution $(u^\epsilon,\pi^\epsilon)$ to the regularized Navier-Stokes equations \eqref{eq-regularizedNS} 
such that $u^\epsilon \in {\bf LE}_q(0,T)$ for any $T<\infty$, and $(u^\epsilon,\pi^\epsilon)$ satisfies the a priori bounds \eqref{th2.2-1}-\eqref{th2.2-2} in Lemma \ref{lem.A0qbound} for all $R=n\in \NN$ up to time $T_n$ defined in \eqref{Tn.def}. 
\end{lemma}

\begin{proof}
We will take radius $R=n\in \NN$. 
By \eqref{eq-3.17-BT4} of Lemma \ref{bt4-lem-2.2}, $A_{0,q}(n) \le c_3 \norm{u_0}_{E^2_q}^2$ for all $n$. For $n\in\NN$, let 
\EQ{\label{Tn.def}
T_n = \la_0 n^2 \min\bket{1, n^2(c_3  \norm{u_0}_{E^2_q}^2)^{-2} } \le \la_n n^2,
}
where $\la_0$ and $\la_n$ are defined in Lemma \ref{lem.A0qbound}.
The sequence $T_n$ is increasing and $T_n\to\infty$. By the same proof of Lemma \ref{lem.A0qbound}, if a solution $(u^\epsilon,\pi^\epsilon)$ of  \eqref{eq-regularizedNS} satisfies $u^\epsilon \in {\bf LE}_q(0,T)$, then it satisfies the a priori bounds \eqref{th2.2-1} and \eqref{th2.2-2} for $R=n$ up to time $\min(T,T_n)$.

As the system \eqref{eq-regularizedNS} can be considered as a Stokes system with localized and regularized source, it has a global unique solution $(u^\epsilon,\pi^\epsilon)$. Its uniqueness makes it agree with the ${\bf LE}_q$-solution of Lemma \ref{KwTs-lem3.4}, hence $u^\epsilon \in {\bf LE}_q(0,\tau)$ for some $\tau=\tau(\epsilon,\norm{u_0}_{E^2_q})>0$.
Fix $n \in \NN$. By \eqref{th2.2-1} with $R=n$, $\norm{u^\epsilon(\tau)}_{E^2_q} \le  C(n) \norm{u_0}_{E^2_q}$. 
By Lemma \ref{KwTs-lem3.4}, there is an ${\bf LE}_q$-solution on $(\tau, \tau+\tau_1)$ for some $\tau_1 = \tau_1(\epsilon,C(n)\norm{u_0}_{E^2_q})>0$. By uniqueness, it agrees with $u^\epsilon$ and we have $u^\epsilon \in {\bf LE}_q(0,\tau+\tau_1)$ with the a priori bound \eqref{th2.2-1} up to $
\tau+\tau_1$. We can repeat this extension to show $u^\epsilon \in {\bf LE}_q(0,\tau+k\tau_1)$, $k \in \NN$, until $\tau+k\tau_1\ge T_n$. We conclude, for each $n \in \NN$, $u^\epsilon \in {\bf LE}_q(0,T_n)$ and satisfies the a priori bound \eqref{th2.2-1} for $R=n$ up to time $T_n$.

As $T_n\to \infty$, the lemma is proved.
\end{proof}

With the above lemmas, we are ready to consider the Navier-Stokes equations with $E^2_q$ data, $1\le q<2$, and construct a global-in-time local energy solution in ${\bf LE}_q(0,T)$.

\begin{proof}[Proof of Theorem \ref{th4.8}]
For $k\in \NN$, let $u^k$ be the solution of the regularized system \eqref{eq-regularizedNS} 
with $\epsilon=1/k$, given in Lemma \ref{BT8-lem3.3}. They share the same a priori bound 
\eqref{th2.2-1} for $R=n$ up to time $T_n$, thus
\[
\sup_{k \in \NN} \|u^k\|_{{\bf LE}_q(0,T_n)}<\infty, \quad \forall n \in \NN.
\]
Using this \textit{a priori} bound, we can construct the desired global solutions as the limit of $u^k$ defined in $(0,T_k)$, $T_k\to\infty$, in the same manner as \cite[Proof of Theorem 1.5]{BT8}. 
The only difference is that all the supremums $\sup_{x_0\in\R^3}$ are replaced by $\norm{\cdot}_{\ell^q(x_0\in\ZZ^3)}$.
For example, \cite[(3.7)]{BT8} is replaced by the $\ell^q$-version, \eqref{eq-thm1.4BT4}.
To avoid redundancy, we omit further details.
\end{proof}

\subsection{Application to the stability of suitability for the perturbed Navier-Stokes equations}
\label{Sec4.2}
 
In \cite{BT4}, the first and third authors constructed time-global local energy solutions with initial data in the $L^2$-based Wiener amalgam classes $E^2_q$, $2\le q<\infty$, which is related to work in \cite{LR2,KiSe,KwTs,BT8,BKT1,FDLR2}.  As a technical step, a suitable weak solution was needed solving the perturbed Navier-Stokes equations  
\EQ{\label{PNS}
\partial_t u -\Delta u+u\cdot\nb u+v\cdot \nb u  + u\cdot\nb v +\nb \pi = 0;\qquad \nb \cdot u=0, {\ \ \mbox{in}\ \ \R^3\times (0,\infty),}
}
and $v\in L^\I L^p_{\uloc}$ where $p>3$ is itself a solution to \eqref{NS} {in $\R^3$}. 
 It is natural to ask  if an analogous statement holds when $v$ is a small $L^3_\uloc$ solution in the sense of Maekawa and Terasawa \cite{MaTe} and our Theorem \ref{thrm:critical}. The new spacetime estimates in  Theorem \ref{thrm:critical} allow us to confirm this, a claim which is explored presently.

Note that in the global existent proof for $E^2_q$ where $1\leq q<2$ the perturbed system is \textit{not} needed; so, this subsection is independent of the work done above.

The following is an exact copy of \cite[Lemma 4.4]{BT4} except the condition \[\esssup_{0<t\leq T_0}\| v(t) \|_{L^4_\uloc}<\I,\] has been replaced by  scaling invariant space-time integral   which is finite for the solutions constructed in Theorem \ref{thrm:critical}.

\begin{proposition}\label{corollary} Let $c_0$ and $\la_0$ be the constants in \cite[Lemma 4.4]{BT4} %
Assume $u_0\in E^2_q$ {$2\le q<\infty$}  is divergence free, and $v:\R^3\times [0,T_0]\to \R^3$ satisfies
$\div v=0$ and  
\[
\esssup_{0<t\leq T_0}\| v(t) \|_{L^3_\uloc}<\delta\le c_0%
;\quad \| v\|_{E^{5,5}_{T_0,\infty}} <\I.
\]
Let $T= \min (T_0,\la_0,\la_0A_{0,q}^{-2})$.    
Then, there exists a weak solution $u$ and pressure $\pi$ satisfying \eqref{PNS}
in the distributional sense on $\R^3\times [0,T]$.
Furthermore, $u$ and {$\pi$} are a local energy solution to the perturbed Navier-Stokes equations \eqref{PNS} satisfying
\EQN{\label{ineq.target3}
\norm{u}_{\operatorname{\mathbf{LE}}_q(0,T)} 
\leq C \|u_0\|_{E^2_q},
}
for a constant $C$.
\end{proposition}

\begin{proof}[Proof of Proposition \ref{corollary}]
We adopt all of the notation and setup of \cite[Proof of Lemma 4.4]{BT4}. The only difference in the present proof compared to \cite[Proof of Lemma 4.4]{BT4} involves the local energy inequality, which should hold for non-negative test functions $\phi \in C_c^\I( \R^3\times [0,T))$. 
The steps in \cite[pp. 2017-2018]{BT4} apply here except for those using the assumption that $\esssup_{0<t\leq T_0}\| v(t) \|_{L^4_\uloc}<\I$, namely the estimation of $I_{1,n}$ and $I_{3,n}$. Note for any ball $B\subset \R^3$ and taking $\{u_n\}$ to be the sequence defined on \cite[p. 2017]{BT4}, $u_n \to u$ weakly in $L^2(0,T_0; L^6(B))$ and  strongly in $L^s(0,T_0; L^q(B))$, $\frac 2s +\frac 3q=\frac 32$, $q<6$. 
Letting $B$ be a ball containing the spatial support of $\phi(t)$ for all $0\leq t\leq T_0$, we have,
\EQN{
I_{1,n}&=\bigg|\int_0^{T_0} \int ((u_n- u)\cdot \nb u_n)\cdot (v\phi ) \,dx\,dt\bigg| 
\\&\lec_\phi \norm{u-u_n}_{L^{10/3}(B\times(0,T_0))} \norm{v}_{E^{5,5}_{T_0,\infty}} \norm{\nb u_n}_{L^2(B\times(0,T_0))} \to 0.
}
For the remaining term we have that
\EQN{
I_{3,n}  &=  \bigg| \int_0^{T_0} \int  (u_n\cdot \nb u_n )\cdot (\eta_{\e_n} * v-v) \phi \,dx\,dt \bigg|
\\&\lesssim \norm{u_n}_{L^{10/3}(B\times(0,T_0))} \norm{\eta_{\e_n} * v-v}_{L^5(0,T_0; L^5(B))} \norm{\nb u_n}_{L^2(B\times(0,T_0))} .
}
Recall from \cite{BT4} that $\eta_{\e_n} * v$ is a spatial mollifier applied to $v$. 
So, by properties of spatial mollifiers and the dominated convergence theorem, the middle term above vanishes.  
This and the work in \cite[pp. 2017-2018]{BT4} establishes the local energy inequality for $\phi \in C_c^\I(\R^3\times [0,T))$.
\end{proof}

Note that without the spacetime integral estimate we would only have the inclusion   
$v \in L^\infty(0,T_0; L^3_\uloc)$ for our drift velocity. 
 The best this gives is when bounding $I_{1,n}$ is
\EQN{
&\norm{\nb u_n}_{L^2(B\times(0,T_0))}\norm{u_nv-uv}_{L^2(B\times(0,T_0))}  
\\&\lec \norm{\nb u_n}_{L^2(B\times(0,T_0))}
\norm{u_n-u}_{L^2(0,T_0;L^6)} \norm{v}_{L^\I(0,T_0;L^3_\uloc)},
}
which is not implied to vanish since we only have weak convergence in $L^2\dot H^1_\loc$.
This is why the assumption $v\in L^\I L^4_\uloc$ is included in \cite
[Lemma 4.4]{BT4}.  On the other hand, the bound for $I_{3,n}$ requires a space-time Lebesgue norm is finite with finite spatial integrability exponent in order to apply the dominated convergence theorem.  

\begin{remark}
In \cite[p.~2018]{BT4}, it is asserted that $(\eta_\e *v )\phi \to v\phi$ in $L^\I(0,T;L^{3}_\uloc)$. This appears to be false. In the reference, and adopting its notation, this can be fixed using the following estimate
\EQN{
&\int_0^t \int(u_n\cdot\nb u_n)\cdot (\eta_\e *v  - v)\phi\,dx\,ds 
\\&\lesssim \| u_n\|_{L^\I (0,T;L^2_\uloc)} ^{{1/4}}\int_0^t 	{\| u_n	 	\|_{H^{1}(B)}^{7/4}}	 \|  \eta_\e *v - v \|_{L^4(B)}\,ds
\\&\lesssim \| u_n\|_{L^\I (0,T;L^2_\uloc)} ^{{1/4}} {\| u_n \|_{L^2(0,T; H^1(B))}^{7/4}} \|  \eta_\e *v - v \|_{L^{8}(0,T;L^4(B))}.
}

The last quantity above vanishes as $\e\to 0$ by the dominated convergence theorem and the fact that $v\in L^\I (0,T; L^4_\uloc)$.  Since the other terms are uniformly bounded in $n$, the left hand side vanishes as well. 
\end{remark}

\section{Appendix}\label{sec5}

In this appendix, we first give details of the endpoint case of Giga's estimates, and then give examples showing the strict inclusions between $L^p$, $E^p_q$, $E^{s,p}_{T,q}$ and $L^s_T E^p_q$.
\subsection{Endpoint case of Giga's estimates}\label{sec5.1}
It is mentioned in \cite[Acknowledgments]{Giga} that the case $r=s>1$ in \eqref{Giga-est} is also valid if one appeals to the generalized Marcinkiewicz theorem. Here we give the details. Let $\Om$ be a bounded smooth open set in {$\R^d$} or {$\R^d$} itself.
Let $A=-\De$ or the Stokes operator.
For fixed $p\in(1,\infty)$,
the \emph{subadditive} map $U(a) = \norm{ e^{-tA} \mathbb{P} a}_p$ for $a \in L^r(\Om)$ or $a \in L^r_\si(\Om)$,
\[
  U: L^r (\Om) \to L^s_{wk}(0,\infty)
\]
is of weak type $(r,s)$ if $2/s=d/r -d/p$, by the $L^r$-$L^p$ estimate of $e^{-tA}$. By Marcinkiewicz’s theorem in Lorentz spaces,
Theorem V.3.15 of Stein-Weiss \cite[p.197]{SteinWeiss},  %
\[
\norm{U(a)}_{L^{s,m}(0,\infty)} \lec \norm{a}_{ L^{r,m}(\Om)}, 
\quad \forall m \in [1,\infty], \quad \forall a \in L^r\cap L^{r,m},
\]
when $r,s\in(1,\infty)$
satisfy $2/s=d/r -d/p$.  In particular, when $r=m \le s$,
\[
\norm{U(a)}_{L^s(0,\infty)}=\norm{U(a)}_{L^{s,s}(0,\infty)} \le\norm{U(a)}_{L^{s,r}(0,\infty)} \lec \norm{a}_{ L^{r,r}(\Om)}=\norm{a}_{ L^{r}(\Om)}.
\]

Marcinkiewicz's theorem for subadditive maps in Lorentz spaces for full range of exponents
was first shown by Hunt \cite{Hunt64} and
Calder\'on \cite{Calderon66}. Indeed, the theorem of Hunt is for quasi-linear maps: $|T(f+g)| \le K(|Tf|+|Tg|)$. A subadditive map is when $K=1$.
 There is a proof for linear maps in Lorentz spaces in Bergh-L\"ofstr\"om \cite[Theorem 5.3.2]{BL}. See \cite{Hunt64,Calderon66,Hunt66},
\cite[p.20]{BL} and \cite[p.216]{SteinWeiss} for more references.

The case $s=\infty$ in \eqref{Giga-est} is a direct consequence of the $L^p$ estimate of Stokes semigroup: 
\[
\norm{e^{-tA} \mathbb{P}a}_p \lec \norm{a}_p.
\]

\subsection{Strict inclusions between functional spaces}
\begin{lemma}
Consider $L^p(\R^d)$ and $E^p_q(\R^d)$ for $p,q\in [1,\oo]$. When $p>q$,
\EQ{\label{inclusion}
E^p_q \subset (L^p \cap L^q).
}
In contrast, when $p<q$, then
\EQ{\label{inclusion2}
L^p + L^q \subset E^p_q.
}
Both inclusions are strict.
\end{lemma}

\begin{proof}
Both inclusions are straight-forward to verify.
We now show they are strict.
Let
\[
a(x) = \sum_{k\in \ZZ^d} c_k \phi\bke{\frac {x-k}{\de_k}},
\]
for $\phi= \chi_{B_{1/2}}$ and $c_k ,\de_k\in (0,1]$.
We have
\EQ{\label{0823a}
\norm{a}_{E^p_q} = c_{p,q} \bigg( \sum_{k\in \ZZ^d} c_k ^q \de_k^{dq/p} \bigg)^{1/q}.
}

When $p>q$, using \eqref{0823a}, we have $a \in  (L^p \cap L^q) \setminus E^p_q$ if
\[
\sum_{k\in \ZZ^d} c_k^q \de_k^d < \infty, \quad \sum_{k\in \ZZ^d} c_k^q \de_k^{dq/p} = \infty.
\]
We may take for example
\[
c_k = 1, \quad \de_k = (1+|k|)^{-\frac p{q}}.
\]
This shows that the inclusion \eqref{inclusion} is strict.

When $p<q$, using \eqref{0823a},
we have $a \in E^p_q$, $a\not\in L^p $ and $a\not\in L^q $ if
\[
\sum_{k\in \ZZ^d} c_k^p \de_k^d = \infty, \quad \sum_{k\in \ZZ^d} c_k^q \de_k^{dq/p} < \infty.
\]
We may take for example
\[
c_k = 1, \quad \de_k = (1+ |k|)^{-1}.
\]
For this choice of $c_k$ and $\de_k$, we claim  $a \in E^p_q \setminus  (L^p + L^q) $. Suppose $a=f+g$ with $f \in L^p$ and $g \in L^q$. We may assume $f \ge 0$ by replacing $f,g$ by
\[
\td f = f \chi_{\{f >0\}},\quad \td g = a - f \chi_{\{f>0\}} = f\chi_{\{f<0\}} + g.
\]
Note $0 \le \td f \le f \in L^p$ and, where $f <0$, %
$0\le a=\td g \le g \in L^q(f<0)$,
and where $f\ge 0$, $\td g= g \in L^q(f\ge 0)$. Similarly we may assume $g \ge 0$. For $k\in \ZZ^d$, let $\e_k = 1$ if the set $ \{f >\frac 13\}\cap B_{\de_k/2}(k)$ has measure at least $\frac13$ of $B_{\de_k/2}$, and $\e_k=0$ otherwise. Also let $\mu_k = 1$ if the set $ \{g >\frac 13\}\cap B_{\de_k/2}(k)$ has measure at least $\frac13$ of $B_{\de_k/2}$, and $\mu_k=0$ otherwise. Since $a=1$ on $B_{\de_k/2}(k)$, we have
\EQ{\label{0823b}
\e_k+\mu_k \ge 1 ,\quad \forall k \in \ZZ^d.
}
Let
\[
F(x) = \sum_{k\in \ZZ^d} c_k \e_k \phi\bke{\frac {x-k}{\de_k}}, \quad 
G(x) = \sum_{k\in \ZZ^d} c_k \mu_k \phi\bke{\frac {x-k}{\de_k}}.
\]
By definition of $\e_k$, $\norm{f}_{L^p(B_1(k))} \ge C\norm{F}_{L^p(B_1(k))}$ for all $k$, hence $\norm{f}_{L^p(\R^d)} \ge C\norm{F}_{L^p(\R^d)}$. Similarly, 
$\norm{g}_{L^q(\R^d)} \ge C\norm{G}_{L^q(\R^d)}$.
However, by \eqref{0823a}, $c_k=1$,  and \eqref{0823b}, 
\EQN{
\int_{\R^d} |f|^p + |g|^q &\ge C\norm{F}_{L^p(\R^d)} ^p +C\norm{G}_{L^q(\R^d)} ^q
\\
&=  C\sum_{k\in\ZZ^d} (c_k \e_k)^p \de_k^d + C\sum_{k\in\ZZ^d} (c_k \mu_k)^q \de_k^d
\\
&=  C\sum_{k\in\ZZ^d}  \e_k \de_k^d + C\sum_{k\in\ZZ^d} \mu_k\de_k^d
\\
&\ge C \sum_{k\in\ZZ^d} \de_k^d = \oo,
}
which is a contradiction to $f\in L^p$ and $g\in L^q$. This shows that the inclusion \eqref{inclusion2} is strict.
\end{proof}

\begin{example}\label{example52}
Here we show that \eqref{ineq:embedding.parabolic.space} and \eqref{ineq:embedding.parabolic.space2} may fail if the order of $q$ and $s$ are switched.
For \eqref{ineq:embedding.parabolic.space},
take for example $s=p=1$, $q=\infty$, and $T=1$. We have $\norm{u}_{E^{1,1}_{T,\uloc}} \le  \norm{u}_{L^1_T L^1_\uloc}$ by \eqref{ineq:embedding.parabolic.space2}.
Let
\EQ{
u(x,t) =  \sum_{k\in \NN}  \frac {2^k}{|B_1|}\cdot\one_{B_1( 2ke_1)}(x) \cdot  \one _{(2^{-k},2^{-k+1}]}(t).
}
then $\norm{u(t)}_{L^1_\uloc}=2^k$ for $t \in (2^{-k},2^{-k+1}]$. Hence
\[
\norm{u}_{E^{1,1}_{T=1,\uloc}} =1, \quad\norm{u}_{L^1_T L^1_\uloc} = \sum_k 1 = \infty,
\quad
\norm{u}_{E^{1,1}_{T=1,\uloc}}  \not \ge C \norm{u}_{L^1_T L^1_\uloc}.
\]
This shows the failure of \eqref{ineq:embedding.parabolic.space} if $q>s$.
As another example, take $p=q=1$ and $s=\infty$. We have $\norm{u}_{L^\infty_T E^1_1} \le \norm{u}_{E^{\infty,1}_{T,1}}  $ by \eqref{ineq:embedding.parabolic.space}.
Let $x_0(t)=(\cot t, 0,\ldots,0)$ and
\EQ{
u(x,t) = 1 \quad \text{if }  |x - x_0(t)|<1; \quad u(x,t)=0 \quad \text{otherwise}.
}
Then $\norm{u}_{L^\infty_{T=\pi} E^1_1} < \infty$ while  $\norm{u}_{E^{\infty,1}_{T=\pi,1}}  =\infty$.
This shows the failure of \eqref{ineq:embedding.parabolic.space2} if $q<s$.
\end{example}

\section*{Acknowledgments}
 Bradshaw was supported in part by the Simons Foundation (635438). 
The research of both Lai and Tsai was partially supported by the NSERC grant RGPIN-2018-04137. %
Lai acknowledges support by the Simons Foundation Math + X Investigator Award \#376319 (Michael I. Weinstein).  A preliminary version of this paper was presented in a RIMS (Research Institute for
Mathematical Sciences, Kyoto University) workshop on December 7, 2021, and received fruitful questions. 
We thank Professor Yoshikazu Giga for sharing two insightful remarks and referring us to several related papers.

\section*{Disclosure statement}
The authors report there are no competing interests to declare.

\addcontentsline{toc}{section}{\protect\numberline{}{References}}

%Zachary Bradshaw, Department of Mathematics, University of Arkansas, Fayetteville, AR 72701, USA;
%e-mail: zb002@uark.edu
%\medskip
 
%Chen-Chih Lai, Department of Mathematics, Columbia University, New York, NY 10027, USA;
%e-mail: cl4205@columbia.edu
%\medskip
 
%Tai-Peng Tsai, Department of Mathematics, University of British
%Columbia, Vancouver, BC V6T 1Z2, Canada;
%e-mail: ttsai@math.ubc.ca


\begin{thebibliography}{XX}

\bibitem{Abe2015}
Abe, Ken, The Navier-Stokes Equations in a Space of Bounded Functions. Comm. Math. Phys. 338(2), 849--865 (2015).

\bibitem{AG}
Abe, Ken; Giga, Yoshikazu, Analyticity of the Stokes semigroup in spaces of bounded functions. Acta Math. 211, 1--46, 2013.

\bibitem{AB} D.~Albritton and R.~Beekie, Long-time behavior of scalar conservation laws with critical dissipation. %
 Ann. Inst. H. Poincaré C Anal. Non Linéaire 39 (2022), no. 1, 225-243.
 
\bibitem{Arriera}
J.M. Arriera, A. Rodriguez-Bernal, J.W. Cholewa, and T. Dlotko, Linear parabolic equations in locally uniform spaces, {Math. Models Methods Appl. Sci.}, 14 (2004), 253-293.

\bibitem{BCD}
Bahouri, Hajer; Chemin, Jean-Yves; Danchin, Rapha\"el, Fourier analysis and nonlinear partial differential equations. Grundlehren der Mathematischen Wissenschaften [Fundamental Principles of Mathematical Sciences], 343. Springer, Heidelberg, 2011. xvi+523 pp. 


\bibitem{BP2020}Barker, Tobias; Prange, Christophe, Localized smoothing for the Navier-Stokes equations and concentration of critical norms near singularities. Arch. Ration. Mech. Anal. 236 (2020), no. 3, 1487-1541.

\bibitem{BL}
J{\"o}ran Bergh and J{\"o}rgen L{\"o}fstr{\"o}m, \emph{Interpolation spaces.
  {A}n introduction}, Springer-Verlag, Berlin, 1976, Grundlehren der
  Mathematischen Wissenschaften, No. 223.
  
\bibitem{BKT1}   Bradshaw, Z., Kukavica, I.~and Tsai, T.-P., Existence of global weak solutions to the Navier-Stokes equations in weighted spaces, 
Indiana Univ. Math. J. 71 No. 1 (2022), 191-212.


\bibitem{BT8} Z. Bradshaw and T.-P. Tsai, Global existence, regularity, and uniqueness of infinite energy solutions to the Navier-Stokes equations, {Comm. Partial Differential Equations} 45 (2020), no. 9, 1168-1201.

\bibitem{BT4}
Z. Bradshaw and T.-P. Tsai, 
Local energy solutions to the Navier-Stokes equations in Wiener amalgam spaces, SIAM J. Math. Anal.  53 (2021) no.2, 1993-2026.


\bibitem{BT7} Z. Bradshaw and T.-P. Tsai,  On the local pressure expansion for the Navier-Stokes equations, J. Math. Fluid Mech. 24, 3 (2022).

\bibitem{BS}
Busby, Robert C.; Smith, Harvey A., Product-convolution operators and mixed-norm spaces. Trans. Amer. Math. Soc. 263 (1981), no. 2, 309--341. 

\bibitem{Calderon66}
Calder\'on, A.-P. Spaces between $L^1$ and $L^\infty$ and the theorem of Marcinkiewicz. Studia Math. 26 (1966), 273--299.

\bibitem{CKS}
Cunanan, J., Kobayashi, M.~and Sugimoto, M., Inclusion relations between $L^p$-Sobolev and Wiener amalgam spaces. J. Funct. Anal. 268 (2015), no. 1, 239-254.

\bibitem{FJR}Fabes, E. B.; Jones, B. F.; Rivière, N. M. 
The initial value problem for the Navier-Stokes equations with data in $L^p$. 
Arch. Rational Mech. Anal. 45 (1972), 222-240. 

\bibitem{FDLR2}   Fern\'andez-Dalgo, P.~G. and  Lemari\'e-Rieusset, P.~G., %
\emph{Weak solutions for Navier-Stokes equations with initial data in weighted $L^2$ spaces}, %
 Arch. Ration. Mech. Anal. 237 (2020), no. 1, 347-382.
 
\bibitem{FDLR}    Fern\'andez-Dalgo, P.~G. and  Lemari\'e-Rieusset, P.~G., Characterisation of the pressure term in the incompressible Navier-Stokes equations on the whole space, %
Discrete Contin. Dyn. Syst. Ser. S 14 (2021), no. 8, 2917-2931. 
 
\bibitem{FuKa} 
Fujita, Hiroshi; Kato, Tosio, On the Navier-Stokes initial value problem. I. Arch. Rational Mech. Anal. 16 (1964), 269-315.


\bibitem{FoSt}
Fournier, J. J. F. and Stewart, J.,
Amalgams of $L^p$ and $ l^q$.
Bull. Amer. Math. Soc. (N.S.) 13 (1985), no. 1, 1-21.

\bibitem{GIM}Giga, Yoshikazu; Inui, Katsuya; Matsui, Shin'ya, On the Cauchy problem for the Navier-Stokes equations with nondecaying initial data. Advances in fluid dynamics, 27-68, Quad. Mat., 4, Dept. Math., Seconda Univ. Napoli, Caserta, 1999.

\bibitem{Giga}
Giga, Yoshikazu,
\emph{Solutions for semilinear parabolic equations in {$L^p$} and
  regularity of weak solutions of the {N}avier-{S}tokes system}, J.
  Differential Equations \textbf{62} (1986), no.~2, 186--212.
  
\bibitem{GiMi}
Giga, Yoshikazu; Miyakawa, Tetsuro, Solutions in $L_r$ of the Navier-Stokes initial value problem. Arch. Rational Mech. Anal. 89 (1985), no. 3, 267--281.

\bibitem{GS1}
Giga, Yoshikazu; Sohr, Hermann, On the Stokes operator in exterior domains. J. Fac. Sci. Univ. Tokyo Sect. IA Math. 36 (1989), 103--130.

\bibitem{GS2}
Giga, Yoshikazu; Sohr, Hermann, Abstract $L^p$-estimates for the Cauchy problem with applications to the Navier-Stokes equations in exterior domains. J. Funct. Anal. 102 (1991), 72--94.

\bibitem{HM}
Hieber, Matthias; Maremonti, Paolo, Bounded analyticity of the Stokes semigroup on spaces of bounded functions. Recent developments of mathematical fluid mechanics. Birkh\"{a}user, Basel, 2016. 275--289.

\bibitem{Holland}
Holland, F., Harmonic analysis on amalgams of $L^p$ and $l^q$. J. London Math. Soc. (2) 10 (1975), 295-305.

\bibitem{JS} Jia, Hao; \v Sver\'ak, Vladim\'ir, Local-in-space estimates near initial time for weak solutions of the Navier-Stokes equations and forward self-similar solutions. Invent. Math. 196 (2014), no. 1, 233--265.

\bibitem{Hunt64} Hunt, Richard A. An extension of the Marcinkiewicz interpolation theorem to Lorentz spaces. Bull. Amer. Math. Soc. 70 (1964), 803--807. 

\bibitem{Hunt66} Hunt, Richard A. On {$L(p,q)$} spaces. Enseign. Math. (2) 12 (1966), 249--276. 


\bibitem{Kato}Kato, Tosio, Strong $L^p$-solutions of the Navier-Stokes equation in $\R^m$, with applications to weak solutions. Math. Z. 187 (1984), no. 4, 471-480.

\bibitem{KNTYY}
Kikuchi, N.; Nakai, E.; Tomita, N.; Yabuta, K.; Yoneda, T.,
Calder\'on-Zygmund operators on amalgam spaces and in the discrete case.
J. Math. Anal. Appl. 335 (2007), no. 1, 198-212.

\bibitem{KiSe} Kikuchi, N.; Seregin, G.,
Weak solutions to the Cauchy problem for the Navier-Stokes equations satisfying the local energy inequality. (English summary) Nonlinear equations and spectral theory, 141–164, 
Amer. Math. Soc. Transl. Ser. 2, 220, Adv. Math. Sci., 59, Amer. Math. Soc., Providence, RI, 2007. 

\bibitem{Kukavica} Kukavica, Igor, On local uniqueness of weak solutions of the Navier-Stokes system with bounded initial data. J. Differential Equations 194 (2003), no. 1, 39–50.

\bibitem{KMT} Kang, K., Miura, H. and Tsai, T.-P., Short time regularity of Navier-Stokes flows with locally $L^3$ initial data and applications, Int. Math. Res. Not. rnz327. https://doi.org/10.1093/imrn/rnz327.

\bibitem{KwTs} Kwon, Hyunju; Tsai, Tai-Peng,
Global Navier-Stokes flows for non-decaying initial data with slowly decaying oscillation. (English summary) 
Comm. Math. Phys. 375 (2020), no. 3, 1665–1715. 

\bibitem{Lakey}
  Lakey, Joseph D. Weighted Fourier transform inequalities via mixed norm Hausdorff-Young inequalities. Canad. J. Math. 46 (1994), no. 3, 586--601.

\bibitem{LR2} Lemari\'e-Rieusset, P. G.,  \emph{The Navier-Stokes problem in the 21st century}. CRC Press, Boca Raton, FL, 2016.

\bibitem{MaTe}  Maekawa, Yasunori and Terasawa, Yutaka, The Navier-Stokes equations with initial data in uniformly local $L^p$ spaces. Differential Integral Equations 19 (2006), no. 4, 369--400. 
 
 
\bibitem{Oseen} Oseen, C. W., ``Neuere Methoden und Ergebnisse in der Hydrodynamik,'' Akademische Verlags-gesellschaft, Leipzig, 1927.


\bibitem{Solonnikov}
Solonnikov, V.~A.,
 Estimates for solutions of a non-stationary linearized system of
  {N}avier-{S}tokes equations.
 {\em Trudy Mat. Inst. Steklov.}, 70:213--317, 1964.
 English translation in A.M.S. Translations, Series II 75:1-117, 1968.

\bibitem{Stein70}
Elias~M. Stein, \emph{Singular integrals and differentiability properties of
  functions}, Princeton Mathematical Series, No. 30, Princeton University
  Press, Princeton, N.J., 1970.

\bibitem{SteinWeiss}
Stein, Elias M.; Weiss, Guido,
\emph{Introduction to Fourier analysis on Euclidean spaces}.
Princeton Mathematical Series, No. 32. Princeton University Press, Princeton, N.J., 1971. x+297 pp. 
  
  
\bibitem{Stein-Wainger-discreteHLS}
Stein, E. M.; Wainger, S. Discrete analogues in harmonic analysis. II. Fractional integration. J. Anal. Math. 80 (2000), 335--355.

\bibitem{Tsai-book}Tsai, T.-P.,  \emph{Lectures on Navier-Stokes Equations}.
Graduate Studies in Mathematics, 192. American Mathematical Society, Providence, RI, 2018.






  
\end{thebibliography}
\end{document}